\documentclass[11pt]{article}
\usepackage{amsmath,amssymb}
\usepackage{mathrsfs}
\usepackage{graphicx}  
\usepackage{hyperref}
\usepackage{subfigure}
\usepackage{float}
\usepackage{color} 
\usepackage{hyperref}
\usepackage{titlesec}
\usepackage{listings}
\usepackage{algorithm}
\usepackage{algpseudocode}
\usepackage{algpseudocode}   %
\usepackage{subeqnarray}
\usepackage{cases}
\titleformat{\section}
{\raggedright\large\bfseries}
{\thesection .\quad}
{0pt}
{}

\titleformat{\subsection}
{\raggedright\normalsize\bfseries}
{\thesection .}
{0pt}
{}
\makeatletter
\@addtoreset{equation}{section}
\@addtoreset{table}{section}
\makeatother

\usepackage{amsthm}

\newtheorem{thm}{Theorem}[section]
\newtheorem{theorem}{Theorem}[section]
\newtheorem{corollary}[thm]{Corollary}
\newtheorem{lemma}[thm]{Lemma}
\newtheorem{proposition}[thm]{Proposition}
\newtheorem{definition}[thm]{Definition}
\newtheorem{remark}[thm]{Remark}
\newtheorem{example}[thm]{Example}
\newtheorem{assumption}[thm]{Assumption} 

\numberwithin{equation}{section}
\numberwithin{table}{section}
\numberwithin{figure}{section}

\allowdisplaybreaks[4]
\makeatletter\@addtoreset{equation}{section}

\usepackage{array}

\usepackage{booktabs}

\title{Ergodicity and invariant measure approximation of the stochastic Cahn–Hilliard equation via an explicit fully discrete scheme}
 
\author{Nan Deng\thanks{Email: deng\_nan@seu.edu.cn}, Yibo Wang\thanks{Email: wangyb@seu.edu.cn}, Wanrong Cao\thanks{Corresponding author. Email address: wrcao@seu.edu.cn. Tel/Fax: +86 25 52090590}\\\\School of Mathematics, Southeast University, Nanjing 210096, P.R.China.} 
\date{}
\begin{document}
\maketitle	

\begin{abstract} 
This paper  investigates the stochastic Cahn--Hilliard equation (SCHE) driven by additive space--time white noise. 
	We first refine the analytical ergodic theory by proving that the continuum equation admits a unique invariant measure in the more regular state space $H_\alpha$, extending the classical result of \cite{da1996stochastic} on the negative Sobolev space $\dot{H}^{-1}_\alpha$. 
	To approximate long-time behaviour, we introduce an explicit fully discrete scheme that combines a finite-difference spatial discretization with a strongly tamed exponential Euler method in time. 
	Uniform-in-time moment bounds in the $L^\infty$-norm are established for the numerical solution, and a uniform strong convergence estimate with an explicit rate is derived for the fully discrete approximation. 
	Exploiting a mass-preserving minorization tailored to Neumann boundary conditions, we further show that the numerical scheme is geometrically ergodic and possesses a unique invariant measure, together with polynomial-order error bounds for approximating the exact invariant measure. 
	Strong laws of large numbers are proved for both the continuous and discrete systems, ensuring almost-sure convergence of temporal averages to the corresponding ergodic limits. 
	Numerical experiments corroborate the theoretical findings, including the long-time strong convergence and the accuracy of invariant measure approximation. 
	Overall, the results provide a complete analytical and numerical framework for investigating the long-time statistical behaviour of the SCHE.\\

	
\end{abstract}

\section{Introduction} 
\par The stochastic Cahn--Hilliard equation (SCHE)
\begin{subequations}\label{eq:CH_FractionalNoise}
	\begin{numcases}{}
		\frac{\partial u}{\partial t}+\Delta^2 u=\Delta f(u)+\sigma \frac{\partial W}{\partial t}, 
		\quad \text{in } [0,\infty)\times\mathcal{O}, \label{eq:CH_FractionalNoise a}\\
		\frac{\partial u}{\partial n}=\frac{\partial \Delta u}{\partial n}=0,
		\qquad\qquad\qquad\quad \text{on } [0,\infty)\times\partial\mathcal{O}, 
		\label{eq:CH_FractionalNoise b}\\
		u(0,\cdot)=u_0, \label{eq:CH_FractionalNoise c}
	\end{numcases}
\end{subequations}
with a cubic polynomial $f$ and additive space--time white noise 
is a prototypical phase-field model for binary alloys subject to thermal fluctuations 
\cite{Cahn1961,cook1970brownian,langer1971theory}. 
Its fourth-order dissipative structure, non-globally Lipschitz drift, 
and mass conservation under Neumann boundary conditions jointly create substantial analytical 
and numerical challenges, particularly in the study of long-time behaviour and invariant measures.

The ergodic properties of the SCHE have been established only in low-regularity spaces. 
The classical work of \cite{da1996stochastic} proved existence and uniqueness of an invariant measure 
on the negative Sobolev space $\dot{H}^{-1}_\alpha$, and subsequent refinements for singular 
nonlinearities or degenerate noise \cite{debussche2007,debussche2011,goudenege2015} remain confined 
to this setting. 
The only existence result in a more regular space---due to \cite{jiang2014}---does not yield uniqueness. 
Consequently, the ergodic theory of the SCHE in physically relevant Hilbert spaces remains incomplete.

On the numerical side, substantial progress has been made on finite-time discretizations of the SCHE. 
Strong convergence rates have been obtained for additive noise using Galerkin finite element or spectral 
methods combined with backward Euler, accelerated implicit Euler, or tamed exponential Euler schemes 
\cite{Qi2020,CuiHongSun2021,cai2023strong}, 
and for fractional Brownian forcing via finite difference and tamed exponential integrators 
\cite{deng2025}. 
Related advances for multiplicative noise appear in 
\cite{	cui2020absolute, cui2022wellposedness, Hong2024}. 
However, all these works are restricted to finite-time analysis, 
and none address uniform moment bounds, ergodicity, or invariant measure approximation.

Invariant measure approximation for SPDEs is well developed for linear problems and semilinear equations 
with globally Lipschitz or monotone drifts 
\cite{Brehier2016,chen2017approximation,chen2020approximation,Hong2019,CuiHongSun2021}. 
These works typically rely on implicit time discretizations, whose computational cost becomes prohibitive 
for long-time simulations. 
Motivated by this, recent studies have explored explicit exponential or tamed exponential schemes for 
nonlinear SPDEs, yielding polynomial-in-time or uniform moment bounds for the stochastic Allen--Cahn 
and Burgers--Huxley equations 
\cite{Brehier2022,WangZhangCao2025,WangY2024}. 
The unified explicit framework of \cite{Jiang2025} further demonstrates that geometric ergodicity 
can be achieved for a broad class of semilinear SPDEs. 
Nevertheless, none of these approaches apply to the SCHE: 
the cubic nonlinearity appears under an additional Laplacian, which leads to an unbounded operator 
applied to a non-globally Lipschitz drift; 
combined with the fourth-order structure and the mass constraint imposed by the Neumann boundary condition, 
this destroys the dissipativity and minorization arguments used in existing long-time analyses for 
second-order equations such as the stochastic Allen--Cahn equation. 
Consequently, long-time accuracy, ergodicity, and invariant measure approximation for the SCHE remain 
completely open.

These limitations motivate the following fundamental questions:  
(i) Can the SCHE admit a unique invariant measure on a function space more regular than 
$\dot{H}^{-1}_\alpha$, consistent with the physical state space?  
(ii) Can one design an explicit fully discrete scheme that achieves uniform-in-time moment bounds and 
uniform strong convergence, despite the equation’s fourth-order structure and nonlinearity?  
(iii) Can such schemes be shown to possess unique invariant measures that converge quantitatively to 
the exact one, together with convergence of numerical time averages to the ergodic limit?

This paper provides affirmative answers to all three questions.  
First, we refine the analytical theory by deriving uniform-in-time $L^\infty$ moment bounds for the exact 
solution via a pair of perturbed auxiliary equations. 
These bounds enable us to extend the invariant measure theory from $\dot{H}^{-1}_\alpha$ to the more 
regular state space $H_\alpha$ associated with mass conservation. 
We prove existence and uniqueness of an invariant measure in $H_\alpha$, while showing that uniqueness 
fails in $H=L^2(\mathcal{O})$, thereby identifying the precise regularity threshold for ergodicity.

Second, we construct an explicit fully discrete scheme by combining a finite difference spatial 
discretization with a strongly tamed exponential Euler method in time. 
The taming strategy, originally introduced in \cite{WangY2024}, is incorporated into our fourth-order 
setting and allows us to establish uniform-in-time moment bounds for the fully discrete solution---a 
property that has not previously been proved for numerical approximations of the SCHE. 
These bounds form the key ingredient for our long-time analysis. 
On this basis, we prove a uniform strong convergence result
\[
\sup_{t\ge0} \|u(t)-u^{\tau,N}(t)\|_{L^p(\Omega;L^\infty)}
\le C\big(h^{1-\varepsilon} + \tau^{ \frac{3}{8}-\frac{\varepsilon}{2}}\big),
\]
which provides, to the best of our knowledge, the first uniform-in-time strong error estimate---with an 
explicit convergence rate---for an explicit fully discrete scheme applied to the SCHE.

Third, we develop a new ergodicity framework suitable for Neumann boundary conditions. 
Since the conserved mass invalidates the classical minorization condition of \cite{mattingly2002}, 
we introduce a mass-preserving minorization posed on affine hyperplanes and combine it with a coupling 
argument to prove existence, uniqueness, and geometric ergodicity of the numerical invariant measure. 
Furthermore, we obtain a quantitative approximation estimate
\[
\Big| \int\phi \, \mathrm{d}\pi - \int\phi \, \mathrm{d}\tilde{\pi}^{\tau,N} \Big|
\le C\big(h^{1-\varepsilon} + \tau^{\frac{3}{8}-\frac{\varepsilon}{2}}\big),
\]
and show that numerical time averages converge both in expectation and almost surely to the ergodic limit. 
These results imply that long-time statistical properties can be efficiently computed from a single 
trajectory of the explicit scheme.

In summary, this work develops a complete analytical and numerical ergodic theory for the SCHE. 
We identify a more regular state space in which the continuum equation is uniquely ergodic, 
design the first explicit fully discrete scheme with uniform moment bounds and uniform strong convergence, 
and establish invariant measure convergence together with ergodic consistency of numerical time averages. 
Taken together, these results provide new mathematical insight into the long-time behaviour of the SCHE 
and yield a practical computational framework for exploring its invariant states.  

The remainder of the paper is organized as follows.  
Section~\ref{section2} introduces notation, functional settings, and preliminary results.  
Section~\ref{section3} establishes existence and uniqueness of the invariant measure for the SCHE in the 
space $H_\alpha$.  
Section~\ref{section4} presents the fully discrete scheme and proves its uniform strong convergence.  
Section~\ref{section5} develops the numerical ergodicity theory, including geometric ergodicity, 
invariant measure approximation, and convergence of time averages.  
Section~\ref{section6} reports numerical experiments illustrating the theoretical findings.

	\section{Preliminaries}\label{section2}  
Recall that $\mathcal{O}=(0, \pi)$. We define $\mathcal{C}^k(\mathcal{O})$ as the space of $k$-times continuously differentiable functions on $\mathcal{O}$ for $k \in \mathbb{N}$. 
In particular,  $\mathcal{C}^0(\mathcal{O}) = \mathcal{C}(\mathcal{O})$ denotes the space of all continuous functions. 
For any integer $d \geq 1$,  the Euclidean norm and inner product on $\mathbb{R}^d$ are denoted by $|\cdot|$ and $\langle \cdot, \cdot \rangle$, respectively.  For  $ p \geq 1$, $L^p(\mathcal{O}) $ is the space of all functions defined on $\mathcal{O}$ that are $p$-th integrable,  equipped with the inner product $\langle \cdot, \cdot \rangle_{L^p}$ and norm $\|\cdot\|_{L^p}$.
When $p=2$, we denote  $H=L^2(\mathcal{O})$. For any   $k \in \mathbb{N}$ and $ p \geq 1$, we denote by $W^{k,p}(\mathcal{O})$ the Sobolev space   endowed with the norm $\|\cdot\|_{W^{k,p}}$.   When $p=2$, we similarly denote $H^k =W^{k,2}(\mathcal{O})$. 
Let $(\Omega, \mathcal{F}, \mathbb{P})$ be a probability space.	Moreover, given a Banach space $(\mathcal{X}, \|\cdot\|_{\mathcal{X}})$, we denote by $L^p(\Omega; \mathcal{X})$ the space of all $\mathcal{X}$-valued random variables whose $p$-th moment is finite, endowed with the norm 
$\|\cdot\|_{L^p(\Omega;\mathcal{X})} := \big(\mathbb{E}[\|\cdot\|_{\mathcal{X}}^p]\big)^{\frac{1}{p}}$. 
Throughout this paper, let $C$ be a generic positive constant that is independent of the discretization parameters and may change from line to line.  \par

Considering the Neumann boundary condition, we define  $\dot{H} :=\left\{v \in H:\left.\frac{\partial v}{\partial n}\right|_{x \in \partial O}=0\right\}$. The operator $A := \Delta$ denotes the Laplacian with  Neumann boundary conditions, whose domain is $\mathcal{D}(A) := H^{2} \cap \dot{H}$.  There exists an orthonormal basis  $\left\{\phi_j\right\}_{j=0}^\infty$ of $\dot{H} $ and an increasing sequence of eigenvalues $\left\{\lambda_j\right\}_{j=0}^\infty$ satisfying $-A\phi_j=\lambda_j\phi_j$.   
Specifically,  $\phi_0(x)= \sqrt{1 / \pi}$ and $\phi_j(x)= \sqrt{2 / \pi} \cos (j x)$ for  $x\in \mathcal{O}$ and $j\geq 1$. The corresponding eigenvalues are  $\lambda_j=j^2$ for $j\geq 0$.   
We define the fractional powers of $-A$ on $\dot{H}$   by  $$(-A)^\beta v=\sum_{j=1}^{\infty} \lambda_j^\beta \left\langle v, \phi_j\right\rangle_{L^2} \phi_j,\quad  \beta \in \mathbb{R}.$$  
The space $\dot{H}^\beta:=\operatorname{dom}\big(A^{\frac{\beta}{2}}\big)$ is a Banach space equipped with the following norm  
$$
\|v\|_{\dot{H}^\beta}:= \big\|(-A)^{\frac{\beta}{2}} v\big\|_{L^2}= \Big(\sum_{j=1}^{\infty} \lambda_j^\beta\left|\left\langle v, \phi_j\right\rangle_{L^2}\right|^2\Big)^\frac{1}{2}, \quad \beta\in \mathbb{R}.
$$ 
To study  the invariant measure, we follow  \cite[]{da1996stochastic}  and introduce the spaces corresponding to $\alpha \in \mathbb{R}$:
\begin{equation}\label{eq:nespace}
	\dot{H}^{-1}_\alpha :=\left\{v \in \dot{H}^{-1}:\frac{1}{\pi} \int_{\mathcal{O}}v \mathrm{ d}x=\alpha\right\}
	\ \text{and}\  H_\alpha := \left\{v \in  \dot{H} : \frac{1}{\pi}\int_{\mathcal{O}}v \mathrm{ d}x=\alpha\right\}.
\end{equation}
\par 
It is straightforward to verify that  $-A$ is a densely defined, self-adjoint, and non-negative definite operator, which is unbounded but has a compact inverse. Consequently, it follows that the operator $ A^2$ generates an analytic semigroup $ e^{-A^2 t} $ for $t \geq 0 $.  \par 
Moreover, for any $\alpha\in[0,\infty)$ and $\hat{\alpha}\in[0,1)$, there exists  a constant  $C>0$  such that
\begin{equation} \label{estimate: e^{-x} 1}
	e^{-x} \leq C  x^{-\alpha}  \quad \text{and} \quad  	1-	e^{-x} \leq C  x^{\hat{\alpha}}, \quad   x> 0.
\end{equation}  
These inequalities will be repeatedly used in the subsequent analysis.  \par 
The  properties of the semigroup $e^{-A^2 t}$ are presented in the following lemma, whose proof is given in Appendix A. 
\begin{lemma}   \label{lem: e {-A^2 t}}
	For any  $ \gamma, \gamma_1,   \gamma_2 \in[0, \infty),\ \gamma_3 \in[0,1)$ satisfying $ \gamma_1-\gamma_2+2\gamma_3 >0$ , and $\hat{u} \in   H^{1} \cap  \dot{H} $, 
	the following inequalities hold:
	\begin{align}
		&\left\|(-A)^\gamma e^{-A^2t}\hat{u}\right\|_{L^{2}}  \leq C\left(t^{-\frac{\gamma}{2}}e^{-\frac{\lambda_{1}^2}{2}t}\left\|\hat{u}\right\|_{L^2}+g(\gamma)\left|\left\langle \hat{u}, \phi_0\right\rangle_{L^2}\right|\right),\quad t>0,\  
		\label{ineq: e^{-A^2t}(-A)^r to l2} \\ 
		&\left\|(-A)^{\gamma_1} e^{-A^2t}\left( I- e^{-A^2s}\right)\hat{u}\right\|_{L^{2}}  \leq C\left( t^{-\frac{\gamma_1-\gamma_2+2\gamma_3}{2}}s^{\gamma_3}e^{-\frac{\lambda_{1}^2}{2}t}\left\|(-A)^{\gamma_2}\hat{u}\right\|_{L^2}+g(\gamma_1)\left|\left\langle \hat{u}, \phi_0\right\rangle_{L^2}\right|\right), \nonumber	  \\
			&\qquad 	\qquad 	\qquad 	\qquad 	\qquad 	\qquad 	\qquad 	\qquad   	\qquad 	\qquad 	\qquad 	\qquad 	\qquad \qquad 	\qquad 	\qquad 	\qquad s,t>0,\label{ineq: (-A)^re^{-A^2t}(I- e^{-A^2t}) to l2}  
	\end{align}
	where  $g:[0,\infty)\to\mathbb{R}$ is defined by $g(0)=1$ and $g(x)=0$ for all $x>0$, and $C>0$ is a constant independent of $s$, $t$, and $\hat{u}$.
\end{lemma}

We next state the assumptions on the nonlinear term $f$, the  initial value  $u_0$, and the noise $W(t,x)$.  
\begin{assumption}\label{assu:1}
	Suppose that $f$ satisfies the following two conditions:\\
	\textbf{(i)} $f(x) =  a_0 x^3 + a_1 x^2 + a_2 x + a_3$, where  $a_0 > 0$ and $a_i \in \mathbb{R}$ for $i = 1, 2, 3$.\\
	\textbf{(ii)} $L_f < \lambda_1$, where $L_f =- \sup_{x \in \mathcal{O}} f'(x)$. 
\end{assumption}
The above assumption ensures that $f$ satisfies 
\begin{align}
	(y-x)\bigl(f(x)-f(y)\bigr) &\leq L_f\,|x-y|^2, \qquad\  \qquad \quad\ x,y \in \mathbb{R},\\
	|f(x)-f(y)| &\leq C_f\bigl(1 + x^2 + y^2\bigr) |x-y|, \quad  x,y \in \mathbb{R}, \label{es: f(x)-f(y)| }\\
	(1-\epsilon_0)\left\langle A\hat{u},\hat{u}  \right\rangle_{L^2}+\left\langle F(\hat{u}-\hat{v})-F(\hat{v}),\hat{u}  \right\rangle_{L^2} 
	&\leq - \epsilon_0\| \hat{u} \|_{L^2},\qquad \qquad \ \qquad \hat{u},\hat{v}  \in L^6(\mathcal{O}) \label{eq: assu:1},
\end{align}  
where $C_f>0$ is a constant dependent of $f$, $\epsilon_0=\frac{\lambda_{1}-L_f}{1+\lambda_{ 1}}$ , and  $F: L^6(\mathcal{O})\to H$ is  a Nemytskij operator defined by $F(v)(x) := f(v(x))$ for $v \in L^6(\mathcal{O})$.  
	
	\begin{assumption}\label{assu:2}
		Suppose that the initial condition $u_0\in H^1.$  
	\end{assumption}
	\begin{assumption}\label{assu:3 noise}
		Suppose that the noise $W(t,\cdot)$ is a cylindrical Wiener process on $\dot{H}$, which can be represented by the formal series
		$$   W(t,x):=\sum_{j=1}^{\infty} \phi_j(x) \beta_j(t), \quad  (t,x)\in[0,\infty)\times\mathcal{O},   $$
		where $\{\beta_{j}\}_{j=1}^\infty$ is a sequence of independent standard Brownian motions.
	\end{assumption}
	
	The concepts of the invariant measure and the Feller property are introduced below, following \cite{Hong2019} for details. For a given Banach space $\mathcal{X}$, we denote by $B_b(\mathcal{X})$ (resp. $C_b(\mathcal{X})$) the Banach space of all bounded, Borel-measurable functions (resp. bounded continuous functions) on $\mathcal{X}$.  Furthermore,  let $\mathcal{L}(\mathcal{X})$ denote the space of all bounded linear operators on $\mathcal{X}$.
	We define the transition semigroup $P_t :[0, \infty) \rightarrow \mathcal{L}\left(B_b(\mathcal{X})\right)$  by  
	$$
	P_t \phi(u_0)=\mathbb{E}[\phi(u(u_0;t,\cdot))], \quad \forall  \phi \in B_b(\mathcal{X}),\ u_0\in \mathcal{X} ,
	$$
	where $u(u_0;t,\cdot)$ is the solution  of the SCHE \eqref{eq:CH_FractionalNoise}  with initial condition $u(0,\cdot)=u_0( \cdot) $. 
	It can be verified that $\left\{P_t\right\}_{t \geq 0}$ is a Markov semigroup.

	\begin{definition}\label{defn: invariant measure}
		A  probability measure $\pi$ on  $(\mathcal{X}, \mathcal{B}(\mathcal{X}))$  is called an invariant measure for the transition semigroup $\left\{P_t\right\}_{t \geq 0}$ if 
		$$
		\int_\mathcal{X} P_t \phi(v) \pi( \mathrm{ d}v)=\int_\mathcal{X} \phi(v)  \pi( \mathrm{ d}v), \quad \forall  \phi \in B_b(\mathcal{X}),\ t \geq 0 .
		$$
		Furthermore, an invariant measure $\pi$ is called ergodic if
		$$ \lim _{T \rightarrow \infty} \frac{1}{T} \int_0^T P_t \phi(u) \mathrm{~d} t=\int_\mathcal{X} \phi(v)  \pi( \mathrm{ d}v) \quad \text{in} \quad L^2(\mathcal{X}; \pi)$$ 
		for all $ \phi\in L^2(\mathcal{X}; \pi)$.  
	\end{definition} 
	\begin{definition}  
		The transition semigroup $\left\{P_t\right\}_{t \geq 0}$ is said to possess the Feller property if, for each $\phi\in C_b(\mathcal{X})$ and every $t>0$, it holds that $P_t\phi\in C_b(\mathcal{X})$.
	\end{definition}

	\begin{remark}
		In Definition 2.4, we define the invariant measure for the transition semigroup $P_t$, which is uniquely determined by the SCHE \eqref{eq:CH_FractionalNoise}. Therefore, in the following, the invariant measure for the semigroup, the SCHE, or the solution will refer to the same object and be used interchangeably.
	\end{remark}

	\section{Invariant measure for the stochastic Cahn--Hilliard equation}\label{section3} 
	This section is devoted to the investigation of  the invariant measure for the SCHE \eqref{eq:CH_FractionalNoise}.  In  \cite{da1996stochastic}, the authors established the existence and uniqueness of the invariant measure for \eqref{eq:CH_FractionalNoise} on the space $\dot{H}^{-1}_\alpha $. In this work, we extend these results to the space $H_\alpha$.  Employing similar arguments to those in \cite[]{da1996stochastic}, we directly conclude that  the SCHE \eqref{eq:CH_FractionalNoise} admits a unique mild solution $u$ given by
	\begin{equation}\label{eq: CH_FractionalNoise_solution_mild}
		\begin{aligned}
			u(t, x)= & \int_{\mathcal{O}} G_t(x, y) u_0\left(y\right) \mathrm{d} y+\int_0^t \int_{\mathcal{O}} \Delta G_{t-s}(x, y) f\left(u\left(s, y\right)\right) \mathrm{d} y \mathrm{d} s  +\sigma\sum_{j=1}^{\infty} \int_0^t \int_{\mathcal{O}} G_{t-s}(x, y)  \phi_j(y) \mathrm{d} y \mathrm{d}  \beta_j(s)  ,
		\end{aligned}
	\end{equation}
	where 	$G_t(x, y)$ is the Green's function given by
	\begin{equation}  \label{eq:the series decompositionof $G$}
		G_t(x, y):=\sum_{j=0}^{\infty} e^{-\lambda_j^2 t} \phi_j(x) \phi_j(y),\quad  t \in[0, \infty),\ x,\ y \in \mathcal{O}.
	\end{equation}   
	Next, we state several  properties of  $G_t(x, y)$, which is proofed in Appendix B.
	\begin{lemma}   
		For any $\alpha\in(0,1)$, there exist  constants $C>0$  such that 
		\begin{align}  
			&  \int_\mathcal{O}|G_{t}(x, y)-G_{t}(z, y)|^2 \mathrm{d} y \leq e^{-\lambda_{ 1} t}\sum_{j=0}^{\infty}\lambda_j e^{-\lambda_j^2 t} \left|x-z\right|^2 ,\quad    t\geq 0,\ x, z\in\mathcal{O},  \label{es:  G_t(x,y)-G_t(z,y)}\\
			& \int_\mathcal{O}|\Delta G_{t}(x, y)-\Delta G_{t}(z, y)|  \mathrm{d} y  \leq Ce^{-\frac{\lambda_{ 1} }{2}t}\left|x-z\right|^\alpha \sqrt{ \sum_{j=0}^{\infty}\lambda_j^{2+ \alpha }e^{-\lambda_j^2 t}},\quad t\geq 0,\ x, z\in\mathcal{O}.   \label{es:  DeltaG_t(x,y)-DeltaG_t(z,y)}
		\end{align} 
	\end{lemma}
	
	%
	For the subsequent analysis, we rewrite \eqref{eq:CH_FractionalNoise} as
	\begin{equation} \label{eq: CH_FractionalNoise Banach}
		\left\{ 
		\begin{aligned}
			&\mathrm{d} U_t+A^2 U_t  \mathrm{d} t  = AF\left(U_t\right)  \mathrm{d} t+ \sigma  \mathrm{d}  W_t   ,\quad  t \in(0, \infty), \\
			& U_0= u_0\in H .
		\end{aligned}
		\right.
	\end{equation}   
	Eq. \eqref{eq: CH_FractionalNoise Banach} possesses a unique mild solution given by
	$$ U_t=e^{-A^2t}u_0+\int_{0}^{t}    Ae^{- A^2(t-s)}F(U_s) \mathrm{d} s +	\mathscr{O}_t , $$
	where	$\mathscr{O}_t:=\sigma\sum_{j=1}^{\infty}  \int_{0}^{t}  e^{- A^2(t-s)}   \phi_j  \mathrm{d}  \beta_j(s)$ is the stochastic convolution.   
	
	\subsection{Key estimates for two auxiliary equations} 
	In order to demonstrate the 	uniform moment boundedness and the Feller property of $u(t,x)$, in this  subsection  we provide  essential estimates for two auxiliary equations.\par 
	We denote $V_t := U_t - \mathscr{O}_t$, which solves the following perturbed equation
	\begin{equation} \label{eq: firsst auxiliary equation}
		\left\{ 
		\begin{aligned}
			&\mathrm{d} V_t+A^2 V_t  \mathrm{d} t  = AF\left(V_t+\mathscr{O}_t\right)  \mathrm{d} t ,\quad  t \in(0, \infty), \\
			& V_0= u_0.
		\end{aligned}
		\right.
	\end{equation}
	We first introduce some  necessary notations and  inequalities. For convenience, we define the functions $\mathbb{I}_1: \mathbb{R}^3\times L([0,\infty);\mathbb{R}) \to \mathbb{R}$ and $\mathbb{I}_2:   \mathbb{R}^3\times L([0,\infty);\mathbb{R}) \to \mathbb{R}$ by
	
	\begin{equation*} 
		\begin{aligned}
			&\mathbb{I}_1(t,\alpha_1,\alpha_2,g):= \int_{0}^{t}(t-s)^{-\alpha_1} e^{-\alpha_2(t-s)}  g(s) \mathrm{d} s,\\ &\mathbb{I}_2(t,\alpha_1,\alpha_2 ,g):= \int_{0}^{t}(t-s)^{-\alpha_1} e^{-\alpha_2(t-s)} \int_{0}^{s}(s-r)^{-\alpha_1} e^{-\alpha_2(s-r)} g(r) \mathrm{d} r\mathrm{d}s.
		\end{aligned}
	\end{equation*}    
	We introduce two special Gagliardo–Nirenberg inequalities:
	\begin{equation} \label{ineq: continuous l^6 to l2} 
		\left\|v \right\|_{L^6}\leq C \left(\left\|v \right\|_{L^2}^2+ \big\|(-A)^\frac{1}{2}v \big\|_{L^2}^2+\left\| A v \right\|_{L^2}^2   \right)^\frac{1}{12}\left\|v \right\|_{L^2}^\frac{5}{6} \leq C \left(   \left\| A v \right\|_{L^2}^\frac{1}{6}\left\|v \right\|_{L^2}^\frac{5}{6}    +\left\|v \right\|_{L^2}   \right),\quad v \in H^2,
	\end{equation} 
	and 
	\begin{equation} \label{ineq: continuous l^infty to l2} 
		\left\|v \right\|_{L^\infty}\leq C \left(\left\|v \right\|_{L^2}^2+ \big\|(-A)^\frac{1}{2}v \big\|_{L^2}^2   \right)^\frac{1}{4}\left\|v \right\|_{L^2}^\frac{1}{2} \leq C \left(    \left\|v \right\|_{L^2}    +\big\|(-A)^\frac{1}{2}v \big\|_{L^2}  \right),\quad v \in H^1.
	\end{equation}  
	In the next proposition, we present a general result. In particular, when $Z_t=\mathscr{O}_t$ and $y_0=u_0$ \eqref{eq: perturbed problem Galerkin abstract}, we obtain the estimate for \eqref{eq: firsst auxiliary equation}. 
	\begin{proposition}\label{propo: perturbed problem}  
		Suppose that $Y_t$ solves
		\begin{equation} \label{eq: perturbed problem Galerkin abstract}
			\left\{ 
			\begin{aligned}
				&\mathrm{d} Y_t+A^2 Y_t  \mathrm{d} t  = AF\left(Y_t+Z_t\right)  \mathrm{d} t ,\quad   t \in(0, \infty), \\
				& Y_0= y_0 ,
			\end{aligned}
			\right.
		\end{equation}   
		where $y_0\in H^1 $ and $Z_t\in W^{1,4}(\mathcal{O})$. 
		Then for any  $p \geq 1$, there exists a constant $C>0$ independent of $t$, such that 
		\begin{equation}\label{es: solution_perturbed problem}
			\begin{aligned}
				&   \|Y_t\|_{L^2} + \big\|(-A)^\frac{1}{2}Y_t\big\|_{L^2} \\ 
				\leq&   	C \left( 1+\|y_0\|_{L^\infty}^{20}+\big\| (-A)^{\frac{1}{2}} y_0\big\|_{L^2}^{10}+ \mathbb{I}_1\big(t, 3/4, \lambda_{1}^2/2,|K_1 |^{3}	+\left\| Z \right\|_{L^6}^3 \big) 
				+ \mathbb{I}_2\big(t, 3/4,  \lambda_{1}^2/2,
				|K_1 |^{5}+\left\| Z \right\|_{L^6}^9 \big) \right) ,
			\end{aligned}
		\end{equation} 
		where $ K_1(t):= \mathbb{I}_1 \big(t,0,  \lambda_{1}^2/2,\left\|  Z  \right\|^8_{L^\infty}+\big\| \left(-A\right)^{\frac{1}{2}}  Z  \big\|^4_{L^4} \big) .   $
	\end{proposition} 
	\begin{proof} 
		We will divide the proof into two steps.\par 
		\textbf{Step 1: }	In this step we prove
		\begin{equation} \label{es: Y L2 exact}
			\begin{aligned}
				\|Y_t\|^2_{L^2}+ \int_{0}^{t} e^{-\frac{\lambda_{1}^2}{2}(t-s)} 	\|AY_s\|^2_{L^2} \mathrm{d}s    
				\leq   
				C\left(1+ \| y_0\|_{L^\infty}^4+ \big\| (-A)^{\frac{1}{2}} y_0\big\|_{L^2}^2+ K_1(t)\right).
			\end{aligned}
		\end{equation}  
		Taking the inner product of \eqref{eq: perturbed problem Galerkin abstract} with $Y_t$ yields 
		\begin{equation} \label{es: Y_t  1}
			\begin{aligned}
				\frac{1}{2}\frac{\mathrm{d}}{\mathrm{d} t} \left(\left\| Y_t\right\|^2_{L^2}\right) 
				=& -\left\|A Y_t\right\|^2_{L^2}+  \left\langle F(Y_t+Z_t) ,AY_t\right\rangle_{L^2}   .
			\end{aligned}
		\end{equation}
		We denote $y(t,x):=Y_t(x)$ and $z(t,x):=Z_t(x)$. By  Young's inequality, one can deduce that
		\begin{equation*} 
			\begin{aligned}
				\left\langle  F(Y_t+Z_t) ,AY_t  \right\rangle_{L^2}
				=&\int_{\mathcal{O}} \left(a_0 \left(y(t,x)+z(t,x)\right)^3+a_1 \left(y(t,x)+z(t,x)\right)^2 
				+a_2 \left(y(t,x)+z(t,x)\right)+a_3\right)\\
				&\qquad\qquad\qquad\times \partial_{xx} y(t,x) \mathrm{d} x  \\ 
				\leq&a_0\int_{\mathcal{O}}   \left(y^3(t,x)+ 3y^2(t,x)z(t,x) \right) \partial_{xx} y (t,x) \mathrm{d} x    
				+\frac{1}{4}\int_{\mathcal{O}} |\partial_{xx} y(t,x)|^2 \mathrm{d} x   \\
				&+C\int_{\mathcal{O}} \left(1+   y^4(t,x)+z^8(t,x) \right)  \mathrm{d} x   .    
			\end{aligned}
		\end{equation*} 
		Using the  integration by parts and  Young's inequality, we obtain
		\begin{equation*} 
			\begin{aligned}
				&\int_{\mathcal{O}}   \left(y^3(t,x)+ 3y^2(t,x)z(t,x) \right)  \partial_{xx} y(t,x) \mathrm{d} x  \\
				=&	-\int_{\mathcal{O}} \left(3 y^2(t,x)  |\partial_{x} y(t,x)|^2+ 6y(t,x)z(t,x) | \partial_{x} y(t,x)|^2+3y^2(t,x)\partial_{x} z(t,x)   \partial_{x} y(t,x) \right) \mathrm{d} x  \\
				\leq &	-\int_{\mathcal{O}} y^2(t,x)  |\partial_{x} y(t,x)|^2\mathrm{d} x 
				+C\int_{\mathcal{O}}  \left( z^2(t,x) | \partial_{x} y(t,x)|^2+ y^2(t,x)|\partial_{x} z(t,x)|^2\right) \mathrm{d} x  \\
				\leq &	-\int_{\mathcal{O}} y^2(t,x)  |\partial_{x} y(t,x)|^2\mathrm{d} x \\
				&+C\int_{\mathcal{O}}  \left( |z(t,x)   \partial_{x} z(t,x)   y(t,x) \partial_{x} y(t,x)|+|z^2(t,x)y(t,x)\partial_{xx} y(t,x)|   + y^2(t,x)|\partial_{x} z(t,x)|^2\right) \mathrm{d} x  \\
				\leq &\frac{1}{4a_0} \int_{\mathcal{O}} |\partial_{xx} y(t,x)|^2 \mathrm{d} x   + C\int_{\mathcal{O}}  \left(1+  y^4(t,x)+z^8(t,x) +|\partial_{x} z(t,x) |^4 \right) \mathrm{d} x.
			\end{aligned}
		\end{equation*} 
		Substituting the above estimations into \eqref{es: Y_t  1} yields
		\begin{equation*} \label{es: Y   1.1}
			\begin{aligned}
				\frac{\mathrm{d}}{\mathrm{d} t} \left(\left\| Y_t\right\|^2_{L^2}\right)  \leq &-\|AY_t\|^2_{L^2}+	C \left(1+ \left\| Y_t \right\|^4_{L^4}+\left\|  Z_t \right\|^8_{L^8}+\big\| \left(-A\right)^{\frac{1}{2}} Z_t \big\|^4_{L^4}\right),
			\end{aligned}
		\end{equation*} 
		which implies
		\begin{equation*} 
			\begin{aligned}
				\frac{	\mathrm{d}}{	\mathrm{d} t}  \left(e^{\frac{\lambda_{1}^2}{2}t}	\|Y_t \|^2_{L^2} \right)  
				=&e^{\frac{\lambda_{1}^2}{2}t}	\frac{\mathrm{d}}{\mathrm{d} t} \left(\left\| Y_t\right\|^2_{L^2}\right)   +\frac{\lambda_{1}^2}{2}e^{\frac{\lambda_{1}^2}{2}t}	\|Y_t \|^2_{L^2}\\
				\leq&  -e^{\frac{\lambda_{1}^2}{2}t}\|AY_t\|^2_{L^2}+	C e^{\frac{\lambda_{1}^2}{2}t}\left(1+ \left\| Y_t+Z_t \right\|^4_{L^4}+\left\|  Z_t \right\|^8_{L^\infty}+\big\| \left(-A\right)^{\frac{1}{2}} Z_t \big\|^4_{L^4}\right).
			\end{aligned}
		\end{equation*}
		Thus, we obtain
		\begin{equation}  \label{eq: (-A)^{1/2}Y(t)  1.8} 
			\begin{aligned}
				&	\|Y_t\|^2_{L^2}+ \int_{0}^{t} e^{-\frac{\lambda_{1}^2}{2}(t-s)} 	\|AY_s\|^2_{L^2} \mathrm{d}s   \\
				\leq & e^{-\frac{\lambda_{1}^2}{2}  t } \|y_0\|^2_{L^2} + C \int_{0}^{t}  e^{-\frac{\lambda_{1}^2}{2}(t-s)}\left(1+ \left\| Y_s+Z_s\right\|^4_{L^4}+\left\|  Z_s \right\|^8_{L^\infty}+\big\|  \left(-A\right)^{\frac{1}{2}} Z_s \big\|^4_{L^4}\right)  \mathrm{d} s.
			\end{aligned}
		\end{equation} 
		
		We next estimate  $\int_{0}^{t}e^{-\frac{\lambda_{1}^2}{2}(t-s)}\left\|Y_s + Z_s\right\|^4_{L^4} \mathrm{d} s$. According to \eqref{eq: perturbed problem Galerkin abstract}, we have $\left\langle Y_t, \phi_0\right\rangle_{L^2} = \left\langle y_0, \phi_0\right\rangle_{L^2}$  and
		\begin{equation*} 
			\mathrm{d} \left\langle  Y_t-y_0,\phi_j\right\rangle_{L^2} +\lambda_{j}^2 \left\langle  Y_t,\phi_j\right\rangle_{L^2} \mathrm{d} t=-\lambda_{j} \left\langle F(Y_t+Z_t),\phi_j\right\rangle_{L^2} \mathrm{d} t, \quad  j\in \mathbb{N}.
		\end{equation*}
		It follows from $\lambda_{0}=0$ and $\left\langle  Y_t- y_0,\phi_0\right\rangle =0$ that
		\begin{equation}\label{eq: (-A)^{1/2}Y(t)  1.1} 
			\begin{aligned}
				& \sum_{j=1}^{\infty}	\frac{\mathrm{d} }{\mathrm{d} t }\left( \frac{ \left\langle  Y_t-y_0,\phi_j\right\rangle^2_{L^2} }{ 2\lambda_{j}}\right)+  \big\| (-A)^{\frac{1}{2}} Y_t\big\|^2_{L^2} \\
				=   	&\big\langle(-A)^{\frac{1}{2}}Y_t,(-A)^{\frac{1}{2}} y_0\big\rangle_{L^2} - \langle F(Y_t +Z_t),  Y_t+Z_t \rangle_{L^2}
				+ \langle F(Y_t +Z_t),  Z_t+ y_0 \rangle_{L^2}.\\
			\end{aligned}
		\end{equation}
		By employing Young's inequality, one can verify that
		\begin{equation}\label{eq: < F(Y+Z), Y+Z >} 
			\begin{aligned}
				& \big\langle F (Y_t +Z_t),Y_t +Z_t \big\rangle  \geq \frac{7}{8}a_0  \left\|Y_t +Z_t\right\|_{L^4}^4 -C,
			\end{aligned}
		\end{equation}
		and
		\begin{equation}\label{eq: < F(Y+Z), Z >} 
			\begin{aligned}
				\big\langle F (Y_t +Z_t), Z_t+ y_0 \big\rangle  
				\leq& \frac{5}{8}a_0  \left\|Y_t +Z_t\right\|_{L^4}^4
				+C\left(1+\|y_0\|_{L^\infty}^4+\left\|Z_t\right\|_{L^4}^4 \right) . \\
			\end{aligned}
		\end{equation}
		Substituting  \eqref{eq: < F(Y+Z), Y+Z >}  and \eqref{eq: < F(Y+Z), Z >}    into \eqref{eq: (-A)^{1/2}Y(t)  1.1}, then applying  H\"{o}lder's inequality and Young's inequality,  we obtain
		\begin{equation}\label{eq: (-A)^{1/2}Y(t)  1.3} 
			\begin{aligned}
				\sum_{j=1}^{\infty}	\frac{\mathrm{d} }{\mathrm{d} t }\left( \frac{ \left\langle  Y_t-y_0,\phi_j\right\rangle^2_{L^2} }{ 2\lambda_{j}}\right)
				\leq&- \frac{1}{2}  \big\| (-A)^{\frac{1}{2}} Y_t\big\|^2_{L^2}- \frac{a_0}{4}  \left\|Y_t +Z_t\right\|_{L^4}^4
				+C    \left( 1+\| y_0\|_{L^\infty}^4+ \big\| (-A)^{\frac{1}{2}} y_0\big\|_{L^2}^2+\left\|Z_t\right\|_{L^4}^4\right) .
			\end{aligned}
		\end{equation}
		Since
		$$ \frac{1}{2}\big\| (-A)^{\frac{1}{2}} Y_t\big\|^2_{L^2}= \frac{1}{2}\sum_{j=1}^{\infty}\lambda_{j} {\left\langle  Y_t ,\phi_j\right\rangle^2_{L^2} }
		\geq \frac{\lambda_{1}^2}{2} \sum_{j=1}^{\infty}\frac{\left\langle  Y_t-y_0,\phi_j\right\rangle^2_{L^2} }{2\lambda_{j}}-C\| y_0\|_{L^2},$$ 
		we have
		\begin{equation*}\label{eq: (-A)^{1/2}Y(t)  1.6} 
			\begin{aligned}
				\sum_{j=1}^{\infty}\frac{\mathrm{d} }{\mathrm{d} t }\left( \frac{e^{\frac{\lambda_{1}^2}{2}t} \left\langle  Y_t- y_0,\phi_j\right\rangle^2_{L^2} }{ 2\lambda_{j}}\right) 
				=  	& e^{\frac{\lambda_{1}^2}{2}t}\sum_{j=1}^{\infty}\frac{\mathrm{d} }{\mathrm{d} t }\left( \frac{ \left\langle  Y_t- y_0,\phi_j\right\rangle^2_{L^2} }{ 2\lambda_{j}}\right)  + \frac{\lambda_{1}^2}{2}e^{\frac{\lambda_{1}^2}{2}t} \sum_{j=1}^{\infty}\frac{  \left\langle  Y_t-y_0,\phi_j\right\rangle^2  }{2\lambda_{j}} \\
				\leq & -\frac{a_0}{4} e^{\frac{\lambda_{1}^2}{2}t} \left\|Y_t +Z_t\right\|_{L^4}^4 +  C e^{\frac{\lambda_{1}^2}{2}t}   \left( 1+\| y_0\|_{L^\infty}^4+ \big\| (-A)^{\frac{1}{2}} y_0\big\|_{L^2}^2+\left\|Z_t\right\|_{L^4}^4\right),
			\end{aligned}
		\end{equation*}
		which yields 
		\begin{equation*}\label{eq: (-A)^{1/2}Y(t)  1.7} 
			\begin{aligned}
				&	\sum_{j=1}^{\infty}	 \frac{  \left\langle  Y_t-y_0,\phi_j\right\rangle^2_{L^2} }{ 2\lambda_{j}}  
				+\frac{a_0}{4} 	\int_{0}^{t}  e^{-\frac{\lambda_{1}^2}{2}(t-s)} \left\|Y_s +Z_s\right\|_{L^4}^4 \mathrm{d} s\\
				\leq &   C\int_{0}^{t}  e^{-\frac{\lambda_{1}^2}{2}(t-s)}    \left( 1+\| y_0\|_{L^\infty}^4+ \big\| (-A)^{\frac{1}{2}} y_0\big\|_{L^2}^2+\left\|Z_t\right\|_{L^4}^4\right) \mathrm{d} s.
			\end{aligned}
		\end{equation*}  
		This combined with \eqref{eq: (-A)^{1/2}Y(t)  1.8}   verifies \eqref{es: Y L2 exact}.

		\textbf{Step 2: }	In this step, we prove
		\begin{equation}  \label{es: (-A_n)^1/2 Y}
			\begin{aligned}
				\big\|(-A)^\frac{1}{2}Y_t\big\|_{l^2_N} 
				\leq&   	C \Big( 1+\|y_0\|_{L^\infty}^{20}+\big\| (-A)^{\frac{1}{2}} y_0\big\|_{L^2}^{10}+ \mathbb{I}_1 (t, 3/4,  \lambda_{1}^2/2,|K_1 |^{3}	+\left\| Z \right\|_{L^6}^3  ) \\
				&\qquad\qquad +\mathbb{I}_2(t, 3/4,  \lambda_{1}^2/2,
				|K_1 |^{5}+\left\| Z \right\|_{L^6}^9 )\Big).\\ 
			\end{aligned}
		\end{equation}   
		The mild solution of  \eqref{eq: perturbed problem Galerkin abstract} is given by
		$$ Y_t=e^{-A^2t} y_0+\int_0^t Ae^{-A^2(t-s)}  F\left(Y_s+Z_s\right)  \mathrm{d} s.$$ 
		Using \eqref{ineq: e^{-A^2t}(-A)^r to l2}, we have
		\begin{equation} \label{es: (-A_n)^1/2 Y 1.1}
			\begin{aligned}
				\big\|(-A)^\frac{1}{2}Y_t\big\|_{L^2} \leq&   \big\|(-A)^\frac{1}{2}e^{-A^2t}y_0\big\|_{L^2} +
				\int_{0}^{t}  \big\|(-A)^\frac{3}{2}e^{-A^2(t-s)}  F(Y_s+Z_s)\big\|_{L^2} \mathrm{d} s\\
				\leq&  C\left(1+\big\|(-A)^\frac{1}{2} y_0\big\|_{L^2}
				+\mathbb{I}_1\Big(t, 3/4,  \lambda_{1}^2/2,\left\| Y\right\|_{L^6}^3+\left\| Z\right\|_{L^6}^3  \Big)  \right).
			\end{aligned}
		\end{equation}  
		Applying  \eqref{ineq: continuous l^6 to l2}, \eqref{ineq: continuous l^infty to l2}, \eqref{ineq: e^{-A^2t}(-A)^r to l2}, and H\"{o}lder's inequality,   one can obtain
		\begin{align}
			\left\|Y_t\right\|_{L^6}\leq&   \big\|e^{-A^2t}y_0\big\|_{L^6}+
			\int_{0}^{t}   \big\| Ae^{- A^2(t-s)}F(Y_s+Z_s) \big\|_{L^6}\mathrm{d} s\nonumber\\
			\leq&  \big\|e^{-A^2t}y_0\big\|_{L^\infty}+
			C	\int_0^t \Big(   \big\|  A^2e^{- A^2(t-s)}F(Y_s+Z_s)   \big\|_{L^2}^\frac{1}{6} \big\| Ae^{- A^2(t-s)}F(Y_s+Z_s)  \big\|_{L^2}^\frac{5}{6}  \nonumber\\
			&\qquad\qquad\qquad\qquad\qquad\qquad  + \big\| Ae^{- A^2(t-s)}F(Y_s+Z_s) \big\|_{L^2}   \Big)  \mathrm{d} s\nonumber \\ 
			\leq&  C\bigg(\|y_0\|_{L^2}+\big\|(-A)^\frac{1}{2} y_0\big\|_{L^2}
			+\int_0^t    \left((t-s)^{-\frac{7}{12}}+ (t-s)^{-\frac{1}{2}}\right) e^{-\frac{ \lambda_{1}^2}{2}(t-s)} \left\|   F(Y_s+Z_s)\right\|_{L^2}        \mathrm{d} s  \bigg). \label{es: Y_l6 1.1 exact} 
		\end{align}
		By \eqref{ineq: continuous l^6 to l2} and Young's inequality, we derive that for any $0\leq\alpha<\frac{3}{4}$, 
		\begin{equation} \label{es: Y_l6 1.2 exact}
			\begin{aligned}
				&	\int_0^t    (t-s)^{-\alpha} e^{-\frac{ \lambda_{1}^2}{2}(t-s)} \left\|   F(Y_s+Z_s)\right\|_{L^2}        \mathrm{d} s \\
				\leq&  C\left( 1+ \int_0^t  (t-s)^{-\alpha} e^{-\frac{ \lambda_{1}^2}{2}(t-s)}\left( \left\|A Y_s\right\|_{L^2}^{\frac{1}{2}}\|Y_s\|_{L^2}^{\frac{5}{2}} +  \left\| Y_s\right\|^3_{L^2}   +\left\| Z_s\right\|^3_{L^2}     \right) \mathrm{d} s\right)  \\ 
				\leq&  C\bigg( 1
				+\int_0^t  (t-s)^{-\alpha} e^{-\frac{ \lambda_{1}^2}{2}(t-s)}\left(\left\| Y_s\right\|_{L^2}^3+\left\| Z_s\right\|_{L^2}^3\right)    \mathrm{d} s  \\ &\quad +\int_0^t    e^{-\frac{ \lambda_{1}^2}{2}(t-s)} \left\|A Y_s\right\|_{L^2}^2   \mathrm{d} s+\int_0^t   (t-s)^{-\frac{4\alpha}{3} }  e^{-\frac{ \lambda_{1}^2}{2}(t-s)}\left\| Y_s\right\|_{L^2}^\frac{10}{3}   \mathrm{d} s\bigg). 
			\end{aligned}
		\end{equation}
		By combining  \eqref{es: Y_l6 1.1 exact}, \eqref{es: Y_l6 1.2 exact}, and  H\"{o}lder's inequality, we obtain 
		\begin{equation} \label{es: Y_l6 1 exact}
			\begin{aligned}
				\left\|Y_t\right\|_{L^6}^3\leq&   C\bigg(1+\|y_0\|_{L^2}^3+\big\|(-A)^\frac{1}{2} y_0\big\|_{L^2}^3
				\bigg)+C \left| \int_0^t\left((t-s)^{-\frac{7}{12}}+ (t-s)^{-\frac{1}{2}}\right) e^{-\frac{ \lambda_{1}^2}{2}(t-s)}\left(\left\| Y_s\right\|_{L^2}^3+\left\| Z_s\right\|_{L^2}^3\right)    \mathrm{d} s\right|^3\\
				& +C \left|  \int_0^t    e^{-\frac{ \lambda_{1}^2}{2}(t-s)} \left\|A Y_s\right\|_{L^2}^2   \mathrm{d} s\right|^3
				+C \left| \int_0^t  \left((t-s)^{-\frac{7}{9}}+ (t-s)^{-\frac{2}{3}}\right)  e^{-\frac{ \lambda_{1}^2}{2}(t-s)}\left\| Y_s\right\|_{L^2}^\frac{10}{3}   \mathrm{d} s\right|^3\\
				\leq&   C\bigg(1+\|y_0\|_{L^2}^3+\big\|(-A)^\frac{1}{2} y_0\big\|_{L^2}^3
				+ \mathbb{I}_1\Big(t, \frac{2}{3},  \frac{ \lambda_{1}^2}{2},\left\| Y \right\|_{L^2}^{10}+\left\| Z \right\|_{L^2}^9  \Big) 
				+\left|  \int_0^t    e^{-\frac{ \lambda_{1}^2}{2}(t-s)} \left\|A Y_s\right\|_{L^2}^2   \mathrm{d} s\right|^3
				\bigg)  .
			\end{aligned}
		\end{equation} 
		It follows from  \eqref{es: Y L2 exact},  \eqref{es: (-A_n)^1/2 Y 1.1}, and \eqref{es: Y_l6 1 exact} that \eqref{es: (-A_n)^1/2 Y} holds. \par 
		Finally, by gathering \eqref{es: Y L2 exact} and \eqref{es: (-A_n)^1/2 Y}, we obtain \eqref{es: solution_perturbed problem}. This completes the proof.
	\end{proof} 
	
	\begin{remark} \label{re: compare prop 1}
		A similar estimate for the perturb equation  in the $\dot{H}^{-1}$-norm can be found in \cite[Theorem 3.1]{da1996stochastic}. Nevertheless, such an estimate is insufficient to guarantee the existence of invariant measure on the space $H_\alpha$.  In Proposition \ref{propo: perturbed problem}, we improve this estimate to establish the uniform moment boundedness of $u(t,\cdot)$ in the more stringent $L^\infty$-norm.
	\end{remark}
	
	We next consider another auxiliary equation, which is used to prove the Feller property of $u(t,x)$. For any $ \hat{u}_0,\ \tilde{u}_0 \in H_\alpha$, we let $U_{\hat{u}_0,t}:=u(\hat{u}_0;t,\cdot) $ and $U_{\tilde{u}_0,t}:=u(\tilde{u}_0;t,\cdot) $ denote the solutions of \eqref{eq:CH_FractionalNoise} with initial conditions $\hat{u}_0$ and $\tilde{u}_0$, respectively. Define $\hat{E}_t=U_{\hat{u}_0,t}-U_{\tilde{u}_0,t}$, which solves
	\begin{equation}  \label{eq: Feller 0}
		\left\{ 
		\begin{aligned}
			&\mathrm{d} \hat{E}_t+A^2 \hat{E}_t\mathrm{d} t= A\big(F(\hat{E}_t + U_{\tilde{u}_0,t} )-F (U_{\tilde{u}_0,t}) \big) \mathrm{d} t,\quad  t \in(0, \infty), \\
			&  \hat{E}_0=\hat{u}_0-\tilde{u}_0 .
		\end{aligned}
		\right.
	\end{equation} 
	In the next proposition, we state a general result. In particular, when $\hat{X}_t=0$, $\hat{Z}_t=U_{\tilde{u}_0,t}$, and $\hat{y}_0=\hat{u}_0-\tilde{u}_0$ in \eqref{eq:Feller}, we obtain the estimate for \eqref{eq: Feller 0}.
	\begin{proposition} \label{propo: error problem  exact} 
		Suppose that Assumption \ref{assu:1} holds and $\hat{Y}_t$ solves
		\begin{equation} \label{eq:Feller}
			\left\{ 
			\begin{aligned}
				&\mathrm{d} \hat{Y}_t+A^2 \hat{Y}_t\mathrm{d} t= A\left(F(\hat{Y}_t+\hat{X}_t+\hat{Z}_t)-F(\hat{X}_t)\right) \mathrm{d} t,\quad  t \in(0, \infty), \\
				& \hat{Y}_0=\hat{y}_0 ,
			\end{aligned}
			\right.
		\end{equation} 
		where $\left\langle\hat{y}_0,\phi_0\right\rangle_{L^2}=0$, $y_0\in H$, and $\hat{X}_t,\ \hat{Z}_t \in  W^{1,2}(\mathcal{O})$.
		Then
		for any   $p \geq 1$, there exists a constant $C>0$, independent of $t$, such that 
		\begin{equation}\label{es: solution ERROR problem}
			\begin{aligned}
				\|\hat{Y}_t\|_{L^p(\Omega;H)}
				&\le C e^{-\frac{\epsilon_0}{2}t}\,\|\hat{y}_0\|_{L^2}
				+ C\sqrt{K_3(t)}\,
				\Big(
				\mathbb{I}_1\big(t,1/2,1/2,\|\hat{Z}\|_{L^{2p}(\Omega;L^{2}(\mathcal{O}))}^2\big)
				\Big)^{\frac{1}{2}} \\[1mm]
				&\quad
				+ C\sqrt{K_2(t)K_3(t)}\,
				\Big(
				\mathbb{I}_1\big(t,0,\epsilon_0,\|\hat{Z}\|_{L^{2p}(\Omega;L^4(\mathcal{O}))}^4\big)
				\Big)^{\frac{1}{4}},
				\qquad t\in(0,\infty).
			\end{aligned}
		\end{equation} 
		Here 
		\begin{equation*} 
			\begin{aligned}
				&	K_2(t):=\mathbb{I}_1\big(t,0,\epsilon_0,1+
				\big\|\hat{X} +\hat{Y} \big\|_{L^{2p}(\Omega;L^8(\mathcal{O}))}^8+\big\|\hat{Z} \big\|_{L^{2p}(\Omega;L^8(\mathcal{O}))}^8\big)  ,\\
				&	K_3(t):=\mathbb{I}_1\big(t, {1}/{2}, \epsilon_0/{2},1+	\big\|\hat{X}_s+\hat{Y} +\hat{Z} \big\|_{L^{2p}(\Omega; W^{1,2}(\mathcal{O}))}^4+\big\|\hat{X} +\hat{Y} \big\|_{L^{2p}(\Omega; W^{1,2}(\mathcal{O}))}^4+	 	\big\| \hat{X} \big\|_{L^{2p}(\Omega; W^{1,2}(\mathcal{O}))}^4\big)   ,
			\end{aligned}
		\end{equation*} 
		where $\epsilon_0$ is given in \eqref{eq: assu:1}.
	\end{proposition}
	\begin{proof}
		It follows from \eqref{eq:Feller} that for $t>0$,  $\left\langle\hat{Y}_t,\phi_0\right\rangle_{L^2} =0 $ and 
		\begin{equation*} 
			\begin{aligned}
				\frac{\mathrm{d} \left\langle \hat{Y}_t,\phi_j\right\rangle_{L^2}}{\mathrm{d} t} +\lambda_{j}^2 \left\langle  \hat{Y}_t,\phi_j\right\rangle_{L^2} 
				=& -\lambda_{j}  \big\langle F(\hat{Y}_t+\hat{X}_t+\hat{Z})-F(\hat{X}_t), \phi_j \big\rangle_{L^2}  ,\quad   j\geq 1.
			\end{aligned}
		\end{equation*}
		Multiplying the above equation by $\left\langle \hat{Y}_t,\phi_j\right\rangle_{L^2}\big/{\lambda_{j}}$ and summing over $j\geq 1$ yields
		$$		\begin{aligned}
			& \frac{\mathrm{d}}{\mathrm{d}t} \left(\sum_{j=1}^{\infty}\frac{\big\langle \hat{Y}_t,\phi_j\big\rangle^2_{L^2}}{2\lambda_{j}}\right)+  \big\|(-A)^{\frac{1}{2}}\hat{Y}_t\big\|^2_{L^2} 
			=- \left\langle F(\hat{Y}_t+\hat{X}_t+\hat{Z})-F(\hat{X}_t),\hat{Y}_t\right\rangle_{L^2}. 
		\end{aligned}$$ 
		It follows from Assumption \ref{assu:1}, Young's inequality, \eqref{es: f(x)-f(y)| } and H{\"o}lder's inequality that 
		\begin{equation*} \label{es: (-A_N)^{-1/2}(U_t-tilde{U}(t))    1}
			\begin{aligned}
				\frac{\mathrm{d}}{\mathrm{d}t} \left(\sum_{j=1}^{\infty}\frac{\big\langle \hat{Y}_t,\phi_j\big\rangle^2_{L^2}}{2\lambda_{j}}\right)
				=  &-\big\|(-A)^{\frac{1}{2}}\hat{Y}_t\big\|^2_{L^2}+
				\Big\langle F (\hat{X}_t)-F(\hat{X}_t+\hat{Y}_t),  \hat{Y}_t \Big\rangle_{L^2} \\
				&+ \Big\langle F (\hat{X}_t+\hat{Y}_t)-F(\hat{X}_t+\hat{Y}_t+\hat{Z}_t),  \hat{Y}_t\Big\rangle_{L^2}  \\
				\leq&-\epsilon_0\left\|(-A )^{\frac{1}{2}}\hat{Y}_t\right\|^2_{L^2}   
				+ C  \left\| F(\hat{X}_t+\hat{Y}_t)-F(\hat{X}_t+\hat{Y}_t+\hat{Z}_t)\right\|_{L^2}^2\\
				\leq&  - \epsilon_0\big\|(-A)^{\frac{1}{2}}\hat{Y}_t\big\|^2_{L^2}
				+ C \left(1+\big\|\hat{X}_t+\hat{Y}_t\big\|_{L^{ 8}}^4+\|\hat{Z}_t\|_{L^8}^4\right)\|\hat{Z}_t\|_{L^{4}}^2.
			\end{aligned}
		\end{equation*}  
		This together with the fact $\sum_{j=1}^{\infty}\frac{  \langle \hat{Y}_t,\phi_j\rangle^2_{L^2} }{ \lambda_{j}}\leq \big\|(-A)^{\frac{1}{2}}\hat{Y}_t\big\|^2_{L^2} $         leads to
		\begin{equation*} \label{eq: (-A)^{1/2}Y(t)  1.6} 
			\begin{aligned}    
				\frac{\mathrm{d}  }{\mathrm{d} t}  \left( \sum_{j=1}^{\infty}\frac{ e^{\epsilon_0t}\big\langle \hat{Y}_t,\phi_j\big\rangle^2_{L^2} }{2\lambda_{j}}\right) 
				= & e^{ \epsilon_0 t} \frac{\mathrm{d}  }{\mathrm{d} t}  \left( \sum_{j=1}^{\infty}\frac{  \big\langle \hat{Y}_t,\phi_j\big\rangle^2_{L^2} }{2\lambda_{j}}\right) +  \epsilon_0e^{ \epsilon_0 t}\sum_{j=1}^{\infty}\frac{  \big\langle \hat{Y}_t,\phi_j\big\rangle^2_{L^2} }{2\lambda_{j}}\\
				\leq & -\frac{\epsilon_0e^{\epsilon_0 t}}{2} 
				\big\|(-A)^{\frac{1}{2}}\hat{Y}_t\big\|^2_{L^2}+ C \left(1+\big\|\hat{X}_t+\hat{Y}_t\big\|_{L^{ 8}}^4+\|\hat{Z}_t\|_{L^8}^4\right)\|\hat{Z}_t\|_{L^{4}}^2 .
			\end{aligned}
		\end{equation*}   
		Thus, we obtain
		\begin{equation*} \label{eq: (-A)^{1/2}Y(t)  1.7} 
			\begin{aligned}
				&\sum_{j=1}^{\infty}\frac{\big\langle \hat{Y}_t,\phi_j\big\rangle^2_{L^2} }{2\lambda_{j}}+\frac{\epsilon_0}{2} \int_{0}^{t}e^{-\epsilon_0 (t-s)} \big\|(-A)^{\frac{1}{2}}\hat{Y}_s\big\|^2_{L^2}\mathrm{d} s \\
				\leq &e^{-\epsilon_0 t}\sum_{j=1}^{\infty}\frac{\big\langle \hat{y}_0,\phi_j\big\rangle^2_{L^2} }{2\lambda_{j}}+C \int_{0}^{t}e^{-\epsilon_0 (t-s)} \left(1+\|\hat{X}_s+\hat{Y}_s\|_{L^{ 8}}^4+\|\hat{Z}_s\|_{L^8}^4\right)\|\hat{Z}_s\|_{L^{4}}^2\mathrm{d} s \\
				\leq&e^{-\epsilon_0 t}\|\hat{y}_0\|_{L^2}^2 +C \int_{0}^{t}e^{-\epsilon_0 (t-s)} \left(1+\|\hat{X}_s+\hat{Y}_s\|_{L^{ 8}}^4+\|\hat{Z}_s\|_{L^8}^4\right)\|\hat{Z}_s\|_{L^{4}}^2\mathrm{d} s. \\
			\end{aligned}
		\end{equation*} 
		By the above estimate and  H{\"o}lder's inequality, we obtain
		\begin{equation}  \label{ineq: int  Y 1,N}
			\begin{aligned}
				\left\|\int_{0}^{t}e^{-\epsilon_0(t-s)} \big\|(-A)^{\frac{1}{2}}\hat{Y}_s\big\|^2_{L^2}\mathrm{d} s \right\|_{L^p(\Omega) }    
				\leq&C e^{-\epsilon_0  t}\|\hat{y}_0\|_{L^2}^{2}+C K_2(t) \left(\int_{0}^{t}  e^{-\epsilon_0(t-s)}  \big\|\hat{Z}_s\big\|_{L^{2p}(\Omega;L^4(\mathcal{O}))}^4    \mathrm{d}s\right)^\frac{1}{2},
			\end{aligned}
		\end{equation}  
		where $$K_2(t):=\left(\int_{0}^{t}e^{-\epsilon_0(t-s)}   \left(1+
		\left\|\hat{X}_s+\hat{Y}_s\right\|_{L^{2p}(\Omega;L^8(\mathcal{O}))}^8+\big\|\hat{Z}_s\big\|_{L^{2p}(\Omega;L^8(\mathcal{O}))}^8\right) \mathrm{d}s\right)^\frac{1}{2}. $$
		
		It can be verified that
		\begin{equation}  \label{ineq: A(uv) to Au Av}
			\| (-A)^\frac{1}{2}(v_1v_2) \|_{L^2}\leq   C \|v_1  \|_{W^{1,2}}  \|(-A)^\frac{1}{2} v_2   \|_{L^2},\quad  v_1,v_2\in W^{1,2}(\mathcal{O}) \ \text{and}\ \langle v_2,\phi_0\rangle_{L^2}=0.
		\end{equation}  
		Since the mild solution to \eqref{eq:Feller} can be written as $$ \hat{Y}_t=e^{-A^2t}\hat{y}_0+\int_0^t Ae^{-A^2(t-s)}\left(F(\hat{X}_s+\hat{Y}_s+\hat{Z}_s)-F(\hat{X}_s)\right) \mathrm{d} s,$$  
		it can be derived from \eqref{ineq: e^{-A^2t}(-A)^r to l2},  \eqref{es: f(x)-f(y)| } and \eqref{ineq: A(uv) to Au Av}    that 
		\begin{equation} \label{ineq: ||Y_t||_L^2}
			\begin{aligned}
				\big\| \hat{Y}_t \big\|_{L^2} \leq& \big\|  e^{-A^2t}\hat{y}_0 \big\|_{L^2}+\int_0^t\big\|A e^{-A^2(t-s)}\big(F(\hat{X}_t+\hat{Y}_t+\hat{Z}_t)-F(\hat{X}_t)\big)\big\|_{L^2} \mathrm{d} s \\
				\leq&Ce^{-\frac{\lambda_{1}^2}{2}t}\left\| \hat{y}_0 \right\|_{L^2}+ C \int_0^t(t-s)^{-\frac{1}{2}} e^{- \frac{\lambda_{ 1}^2}{2}(t-s)} \left\|F(\hat{X}_s+\hat{Y}_s+\hat{Z}_s)-F(\hat{X}_s+\hat{Y}_s) \right\|_{L^2}\mathrm{d} s\\
				&+C \int_0^t(t-s)^{-\frac{1}{4}}  e^{-\frac{\lambda_{1}^2}{2}(t-s)}\big\|\left(-A\right)^{\frac{1}{2}}\big(F(\hat{X}_s+\hat{Y}_s)-F(\hat{X}_s)\big) \big\|_{L^2}\mathrm{d} s  \\
				\leq&Ce^{-\frac{\lambda_{1}^2}{2}t}\left\|\hat{y}_0 \right\|_{L^2}+ C \int_0^t(t-s)^{-\frac{1}{2}}e^{- \frac{\lambda_{1}^2}{2}(t-s)}\big(1+\big\|\hat{X}_s+\hat{Y}_s+\hat{Z}_s\big\|_{L^{\infty}}^2+\big\|\hat{X}_s+\hat{Y}_s \big\|_{L^{\infty}}^2\big)\big\|\hat{Z}_s\big\|_{L^2}\mathrm{d} s \\
				&+ C \int_0^t(t-s)^{-\frac{1}{4}}e^{- \frac{\lambda_{1}^2}{2}(t-s)}\big(1 + 	\| \hat{X}_s+\hat{Y}_s \|_{W^{1,2} }^2+	\| \hat{X}_s\|_{W^{1,2} }^2\big)\big\| (-A)^\frac{1}{2}\hat{Y}_s\big\|_{L^2}\mathrm{d} s.
			\end{aligned}
		\end{equation} 
		By applying the H\"older inequality, \eqref{ineq: ||Y_t||_L^2}, \eqref{ineq: int Y 1,N}, and using the Sobolev embedding  $W^{1,2}(\mathcal{O}) \hookrightarrow L^\infty(\mathcal{O})$,  we obtain
		\begin{equation*} 
			\begin{aligned}
				&	  \| \hat{Y}_t \|_{L^p(\Omega;H)}  \\
				\leq 	&Ce^{-\frac{\lambda_{1}^2}{2}t}\left\| \hat{y}_0 \right\|_{L^2}	 + C \int_0^t (t-s)^{-\frac{1}{2}} e^{-\frac{\lambda_{1}^2}{2}t} 
				\Big(1 + \|\hat{X}_s+\hat{Y}_s+\hat{Z}_s\|_{L^{2p}(\Omega;L^{\infty}(\mathcal{O}))}^2 \\
					& \qquad \qquad  \qquad \qquad  \qquad \qquad + \|\hat{X}_s+\hat{Y}_s\|_{L^{2p}(\Omega;L^{\infty}(\mathcal{O}))}^2 \Big) \|\hat{Z}_s\|_{L^{2p}(\Omega;H)} \, \mathrm{d} s \\
				&+C	   \left(\int_0^t(t-s)^{-\frac{1}{2}} e^{-(\lambda_{1}^2-\epsilon_0)(t-s)}\left(1+	\| \hat{X}_s+\hat{Y}_s \|_{L^p(\Omega; W^{1,2}(\mathcal{O}))}^4+	\| \hat{X}_s\|_{L^p(\Omega; W^{1,2}(\mathcal{O}))}^4\right)\mathrm{d} s \right)^\frac{1}{2}   \\
					& \qquad \qquad  \qquad \qquad  \times
				\left\|\int_0^t  e^{-\epsilon_0(t-s)} \left\| (-A)^\frac{1}{2}\hat{Y}_s \right\|_{L^2}\mathrm{d} s \right\|_{L^{p}(\Omega)}^\frac{1}{2}\\
				\leq 	&C e^{-\frac{ \epsilon_0  }{2}t}\left\| \hat{y}_0 \right\|_{L^2} + C\sqrt{K_3(t)} 
				\left(\mathbb{I}_1\Big(t,1/2,1/2,\big\|\hat{Z} \big\|_{L^{2p}(\Omega;L^{2}(\mathcal{O}))}^2  \Big)  \right)^\frac{1}{2}  \\
					&+C \sqrt{K_2(t)K_3(t)} \left(
				\mathbb{I}_1\Big(t,0,\epsilon_0, \big\|\hat{Z} \big\|_{L^{2p}(\Omega;L^4(\mathcal{O}))}^4    \Big)
				\right)^\frac{1}{4}.
			\end{aligned}
		\end{equation*}    
		This completes the proof.
	\end{proof}
	
	\begin{corollary} \label{coro: Feller}
		Suppose that Assumptions \ref{assu:1}-\ref{assu:3 noise} hold, then the transition semigroup $\{P_t\}_{t\in[0,\infty)}$ of  SCHE \eqref{eq:CH_FractionalNoise} possesses the Feller property.
	\end{corollary} 
	\begin{proof}
		For any $t \in [0,\infty)$, it follows from  \eqref{eq: Feller 0} and Proposition \ref{propo: error problem  exact} that
		\begin{equation} \label{ineq: Feller 01} 
			\left\|u(\hat{u}_0;t,\cdot)- u(\tilde{u}_0;t,\cdot)\right\|_{L^2 (\Omega;H)}  			\leq  C e^{-\frac{ \epsilon_0  }{2}t}\left\| \hat{u}_0-\tilde{u}_0  \right\|_{L^2},\quad \hat{u}_0,\tilde{u}_0  \in H. 
		\end{equation}   
		For any $\epsilon > 0$ and $\phi \in C_b(H)$, there exists $\delta > 0$ such that   $|\phi(v_1) - \phi(v_2)| < \frac{\epsilon}{2}$  whenever $\|v_1 - v_2\|_{L^2} < \delta$.  
		Define $$\Omega_{t,\delta}:=\big\{\omega\in\Omega\ \big| \  \|u(\hat{u}_0;t,\cdot)- u(\tilde{u}_0;t,\cdot) \|_{L^2}<\delta  \big\}.$$ 
		
		By using the continuity and boundedness of $\phi$, \eqref{ineq: Feller 01}, and the H\"older  inequality, we obtain
		\begin{equation*}  
			\begin{aligned}
				|P_t  \phi(\hat{u}_0)- P_t  \phi(\tilde{u}_0) | 
				\leq& \mathbb{E}\left[  | \phi(u(\hat{u}_0;t,\cdot))-\phi(u(\tilde{u}_0;t,\cdot))| \cdot \chi_{\Omega_{t,\delta}} \right] 
				+ \mathbb{E}\left[ | \phi(u(\hat{u}_0;t,\cdot))-\phi(u(\tilde{u}_0;t,\cdot))|\cdot \chi_{\Omega_{t,\delta}^C} \right]\\
				\leq &\frac{\epsilon}{2}+ \sqrt{\mathbb{E}\left[ | \phi(u(\hat{u}_0;t,\cdot))-\phi(u(\tilde{u}_0;t,\cdot))|^2   \right] }\sqrt{\mathbb{E}\left[  \chi_{\Omega_{t,\delta}^C}  \right] }\\
				\leq &\frac{\epsilon}{2}+ 2 \sup_{v\in H} | \phi(v)|\cdot \frac{\left\|u(\hat{u}_0;t,\cdot)- u(\tilde{u}_0;t,\cdot)\right\|_{L^2 (\Omega;H)}  }{\delta }\\
				\leq &\frac{\epsilon}{2}+  C e^{-\frac{ \epsilon_0  }{2}t}\left\| \hat{u}_0 -\tilde{u}_0\right\|_{L^2},
			\end{aligned}
		\end{equation*} 
		where ${B}^C$ denotes the complement of a set ${B}$, and $\chi_{{B}}$ denotes its indicator function.
		Therefore, there exists a constant $\hat{\delta}>0$, such that,  whenever  $\|\hat{u}_0- \tilde{u}_0\|_{L^2} < \hat{\delta},$ we have
		$ |P_t  \phi(u_0)- P_t  \phi(\hat{u}_0) |  < \epsilon ,$   
		which means	 $P_t$ is continuous. 
		It follows from the boundedness of  $\phi$ that $P_t$ is  bounded. In summary, $\{P_t\}_{t\in[0,\infty)}$ possesses the Feller property
	\end{proof}

	\subsection{Existence and uniqueness of invariant measure}
	In this subsection, we establish the existence and uniqueness  of invariant measures. We first need to prove the uniform moment boundedness of $u(t,x)$ to obtain the existence of invariant measures. 
	By  Proposition \ref{propo: perturbed problem}, it is sufficient to establish the regularity of $ \mathscr{O}_t$  which is stated in the following lemma.
	
	\begin{lemma}  \label{lem: 3.1 regularity of o(t,x)}
		For any $p \geq 1$, $0\leq\beta<\frac{3}{2}$, and $\gamma<\frac{1}{2}$, it holds that
		\begin{align} 
			&	\sup_{t\geq 0} \mathbb{E}\left[\|  \mathscr{O}_t\|^p_{\dot{H}^\beta}\right]+\sup_{t\geq 0} \mathbb{E}\left[\| \left(-A\right)^{\frac{1}{2}}  \mathscr{O}_t\|^p_{W^{\gamma,2}}\right] <\infty.\label{ineq: regularity of o(t,x) continuous 1} 
		\end{align}
		Moreover, denote $o(t,x):=\mathscr{O}_t(x)$, then for  any   $0\leq s<t$ it holds that
		\begin{equation}\label{ineq: continuous of o(t,x)}
			\mathbb{E}\left[\| \mathscr{O}_t- \mathscr{O}_s\|^p_{\dot{H}^\beta}\right]\leq C(t-s)^{ \left(\frac{3}{8}-\frac{\beta }{4}
				-\epsilon \right) p},
		\end{equation} 
		\begin{equation}\label{ineq: futher property of  continuous O_t}
			\sup_{t\geq 0}\mathbb{E}\left[|o(t,x)-o(t,z)|^p\right]\leq C|x-z|^p,  \quad  x,\ z \in \mathcal{O},
		\end{equation} 
		where $C>0$ is a constant independent of $s, t, x$ and $z$ and  $\epsilon$ is an arbitrarily small positive number. 
	\end{lemma}
	\begin{proof}
		By employing Parseval’s identity, the orthonormality of $\left\{\phi_j\right\}_{j\geq 0} $, Itô's isometry, and \eqref{estimate: e^{-x} 1},  we obtain, for any $p \geq 1$ and  $0\leq\beta<\frac{3}{2}$,
		\begin{equation*} 
			\begin{aligned}
				\|\mathscr{O}_t-\mathscr{O}_s\|_{L^{2p}(\Omega;\dot{H}^\beta)} 
				=&\Big\|\sum_{j=1}^{\infty}\lambda_{j}^\beta \sigma^2 \Big| \int_0^s     \Big(e^{-\lambda_j^2 (t-r)}- e^{-\lambda_j^2 (s-r)}\Big)   \mathrm{d}  \beta_j(r) +\int_s^t   e^{-\lambda_j^2 (t-r)} \mathrm{d}  \beta_j(r)  \Big|^2	\Big\|_{L^p(\Omega;\mathbb{R})}^\frac{1}{2}\\
				\leq &C\Big(  \sum_{j=1}^{\infty}\lambda_{j}^\beta  \Big(\big(1-e^{-\lambda_j^2(t-s)}\big)^2 \Big\|  \int_0^s   e^{-\lambda_j^2(s-r)} \mathrm{d}  \beta_j(r) \Big\|_{L^{2p}(\Omega;\mathbb{R})}^2 \\
				&\qquad\qquad\qquad+	\Big\|  \int_s^t   e^{-\lambda_j^2(t-r)} \mathrm{d}  \beta_j(r)  \Big\|_{L^{2p}(\Omega;\mathbb{R})}^2 \Big)\Big)^\frac{1}{2}   \\
				\leq &C\Big(  \sum_{j=1}^{\infty}\lambda_{j}^\beta  \Big(\big(1-e^{-\lambda_j^2(t-s)}\big)^2   \int_0^s    e^{-2\lambda_j^2(s-r)} \mathrm{d} r+	   \int_s^t    e^{-2\lambda_j^2(t-r)}   \mathrm{d} r \Big)\Big)^\frac{1}{2}   \\
				\leq &C\Big(   \sum_{j=1}^{\infty}\lambda_j^{\beta-2}\Big( \big(1-e^{-\lambda_j^2(t-s)}\big)^2\big(1-e^{-2\lambda_j^2s}\big)    + \big(1-e^{-2\lambda_j^2(t-s)}\big) \Big)
				\Big)^\frac{1}{2}\leq C(t-s)^{  \frac{3}{8} -\frac{\beta }{4}	-\epsilon   },
			\end{aligned}
		\end{equation*}  
		where  $\epsilon$ is an arbitrarily small positive number. 
		By an analogous derivation, we obtain	 
		\begin{equation} \label{regularity of o(t,x)  1.0}	
			\sup_{t\geq0} \mathbb{E}\left[\|  \mathscr{O}_t\|^p_{\dot{H}^\beta}\right] < \infty.
		\end{equation}  
		Let  $\psi_j(x)= 
		\sqrt{2/ \pi} \sin (j x)$ for $j\geq 1$, which
		form an orthonormal basis of  $L^2_0(\mathcal{O})$ given by
		$$L^2_0(\mathcal{O}):=\left\{v \in L^2(O):\ v|_{x \in \partial O}=0\right\}.$$  
		It can be verified that    
		$
		\left|\psi_j(x)-\psi_j(y)\right| \leq C \lambda_j^{\frac{\alpha}{2}}|x-y|^\alpha 
		$
		for any $\alpha \in[0,1]$ and $x, y \in \mathcal{O}$.
		It follows from  $$  \partial_x o(t,x)-\partial_y o(t,y)= -\sigma\sum_{j=1}^{\infty}  \int_{0}^{t}\lambda_j^\frac{1}{2} e^{-\lambda_j^2 (t-s)}\left(\psi_j(x)-\psi_j(y)\right)   \mathrm{d}  \beta_j(s)    ,   $$
		the orthonormality of $\{\psi_j\}_{j\geq 1}$, and Itô's isometry  that
		\begin{equation} \label{regularity of o(t,x)  1.2}
			\begin{aligned}
				\left\|	 \partial_x o(t,x)-\partial_y o(t,y)\right\|^2_{L^{ p}(\Omega;\mathbb{R})}  
				\leq&   C\sum_{j=1}^{\infty}   \int_{0}^{t}\lambda_j  e^{-2\lambda_j^2 (t-s)}|\psi_j(x)-\psi_j(y)|^2  \mathrm{d} s 
				\leq   C\sum_{j=1}^{\infty}  \lambda_j^{-1}|\psi_j(x)-\psi_j(y)|^2  \\
				\leq &   C\sum_{j=1}^{\infty}    \lambda_j^{-\frac{1}{2}-\epsilon}|x-y|^{1-2\epsilon} 
				\leq  C |x-y|^{1-2\epsilon} .
			\end{aligned}
		\end{equation} 
		Therefore,
		\begin{equation} \label{regularity of o(t,x)  1.3}
			\begin{aligned}
				\left\| \left(\int_{\mathcal{O}} \int_{\mathcal{O}}  \frac{|\partial_x o(t,x)-\partial_y o(t,y)|^2}{|x-y|^{1+2\gamma}}\mathrm{d} x \mathrm{d} y\right)^\frac{1}{2}\right\|_{L^{p}(\Omega;\mathbb{R} )} 
				\leq & \left(   \int_{\mathcal{O}} \int_{\mathcal{O}}  \frac{\left\|	 \partial_x o(t,x)-\partial_y o(t,y)\right\|^2_{L^{ p}(\Omega;\mathbb{R})}}{|x-y|^{1+2\gamma}}\mathrm{d} x \mathrm{d} y  \right)^\frac{1}{2}\\
				\leq & C \left(   \int_{\mathcal{O}} \int_{\mathcal{O}}  |x-y|^{-2\epsilon-2\gamma}\mathrm{d} x \mathrm{d} y  \right)^\frac{1}{2} 
				\leq  C.
			\end{aligned}
		\end{equation}
		Combining  \eqref{regularity of o(t,x)  1.0}, \eqref{regularity of o(t,x) 1.3}, and the following identy
		\begin{equation*}  
			\begin{aligned}
				\| \left(-A\right)^{\frac{1}{2}} \mathscr{O}_t\|^2_{W^{\gamma,2}} 
				=  &\|  \mathscr{O}_t \|_{\dot{H}^1}^2+ \int_{\mathcal{O}} \int_{\mathcal{O}}  \frac{|\partial_x o(t,x)-\partial_y o(t,y)|^2}{|x-y|^{1+2\gamma}}\mathrm{d} x \mathrm{d} y,
			\end{aligned}
		\end{equation*}
		we obtain \eqref{ineq: regularity of o(t,x) continuous 1}. 
		In addition, \eqref{ineq: futher property of continuous O_t} follows from an argument analogous to that of \eqref{regularity of o(t,x) 1.2}. This completes the proof.
	\end{proof}

	Next we demonstrate the uniform moment boundedness of $u(t,x)$,  which is essential for ensuring the existence of invariant measure.
	\begin{theorem}  \label{th:  u_regularity}
		Suppose that Assumptions  \ref{assu:1}--\ref{assu:3 noise} hold.
		Then for any $p \geq 1$, 
		it holds that
		\begin{equation}\label{ineq: u_regularity}
			\sup_{t\in[0,\infty)} \mathbb{E}\left[\|u(t,\cdot)\|^p_{L^\infty}\right]   < \infty .
		\end{equation}
	\end{theorem}
	\begin{proof}
		It follows from   \eqref{ineq: continuous l^infty to l2}, the Sobolev embedding theorem, and  Lemma \ref{lem: 3.1 regularity of o(t,x)} that 
		\begin{equation}   \label{es: regularity of o(t,x) further}
			\begin{aligned}
				&\sup_{t\geq 0} \mathbb{E}\left[\|  \mathscr{O}_t\|^q_{L^\infty}\right]+\sup_{t\geq 0} \mathbb{E}\left[\|  (-A)^\frac{1}{2}\mathscr{O}_t\|^q_{L^4}\right] \\
				\leq&  \sup_{t\geq 0} \mathbb{E}\left[\|  \mathscr{O}_t\|^q_{L^2}\right]+\sup_{t\geq 0} \mathbb{E}\left[\|  \mathscr{O}_t\|^q_{ \dot{H}^1}\right] +\sup_{t\geq 0} \mathbb{E}\left[\| \left(-A\right)^{\frac{1}{2}}  \mathscr{O}_t\|^q_{W^{\gamma,2}}\right] <\infty, \quad \forall q\geq 1.
			\end{aligned}
		\end{equation}
		Combining \eqref{eq: firsst auxiliary equation}, \eqref{es: regularity of o(t,x) further} with Proposition \ref{propo: perturbed problem} yields
		\begin{equation*}  
			\begin{aligned}
				\sup_{t\geq 0}  \mathbb{E}\left[\|V_t\|^p_{L^2}\right] + \sup_{t\geq 0} \mathbb{E}\left[\|(-A)^\frac{1}{2}V_t\|^p_{L^2}\right] <\infty.
			\end{aligned}
		\end{equation*} 
		Applying the above estimate,  Lemma \ref{lem: 3.1 regularity of o(t,x)}, and the Sobolev embedding  $W^{1,2}(\mathcal{O}) \hookrightarrow L^\infty(\mathcal{O})$,  we obtain
		\begin{equation}  \label{es: U  uniform W1,2}
			\begin{aligned}
				\sup_{t\geq 0}	\mathbb{E}\left[\|u(t,\cdot)\|^p_{L^\infty}\right]=	\sup_{t\geq 0}	\mathbb{E}\left[\|U_t\|^p_{L^\infty}\right]\leq C\sup_{t\geq 0}	\mathbb{E}\left[\|U_t\|^p_{W^{1,2}}\right]  <\infty.
			\end{aligned}
		\end{equation}  
		The proof  is completed.
	\end{proof}
	
	By using Theorem \ref{th:  u_regularity} and Corollary \ref{coro: Feller}, we present   the existence of invariant measure for the SCHE \eqref{eq:CH_FractionalNoise} in the following theorem. We recall that $H_\alpha := \left\{v \in  \dot{H} : \frac{1}{\pi}\int_{\mathcal{O}}v \mathrm{ d}x=\alpha\right\}$. The following analysis is valid for any fixed $\alpha$.
	\begin{theorem}  \label{th: invariant measure for the SCHE 1}
		Suppose that Assumptions \ref{assu:1}-\ref{assu:3 noise} hold, then there exists an invariant  
		measure $\pi$ for the SCHE \eqref{eq:CH_FractionalNoise} on $H_\alpha$. 
	\end{theorem} 
	\begin{proof}
		Define
		$$ 
		\hat{\pi}_T(G):=\frac{1}{T} \int_0^T \mathbb{P}_t\left(u_0, G\right) \mathrm{d} t  , \quad  G \in \mathscr{B}(H_\alpha),\ T>0,
		$$
		where $\mathbb{P}_t\left(u_0, G\right)=\mathbb{E}\left[\chi_G\left(u(u_0;t,\cdot) \right)\right] $ is the transition probability of $u$
		and  $\mathscr{B}(\mathcal{X})$  denotes the Borel $\sigma$-algebra on a specific space $\mathcal{X}$.
		By Theorem \ref{th: u_regularity}, the family  $\left\{\mathbb{P}_t\left(u_0, \cdot\right)\right\}_{t\geq 0}$ is tight on $\mathcal{P}\left(H_\alpha\right):=\left\{\mu \mid \mu\right.$ is a probability measure on $\left.\mathscr{B}\left(H_\alpha\right)\right\}$. Hence, for any $\epsilon>0$, there exists a compact set $G_{\epsilon}$ such that
		$$
		\hat{\pi}_T\left(G_{\epsilon}\right)=\frac{1}{T} \int_0^T  \mathbb{P}_t\left(u_0, G_{\epsilon}\right)\ \mathrm{d} t \geq \frac{1}{T} \int_0^T(1-\epsilon) \mathrm{d} t=1-\epsilon.
		$$ 
		This implies that the family  $\left\{\hat{\pi}_t\right\}_{t\in[0,\infty)}$ is tight. Consequently, by Corollary \ref{coro: Feller} and the Krylov–Bogoliubov Theorem \cite[Theorem 1.2]{Hong2019}, the mild solution \eqref{eq: CH_FractionalNoise_solution_mild} admits an invariant measure $\pi$ on $ H_\alpha $.  
	\end{proof}
	
	\cite[]{da1996stochastic} has established the uniqueness of the invariant measures on $\dot{H}^{-1}_\alpha$. Next we aim to extend this result to  $H_\alpha$. To this end, we first establish the following lemma to explain the relationship between the uniqueness of the invariant measure on different spaces.
	
	\begin{lemma} \label{lem: uniqueness invariant measure from big to small} 
		Suppose that the following conditions hold:\\
		\textbf{(i)}  $\left(\mathcal{X},\|\cdot\|_{\mathcal{X}}\right)$ and  $\left(\mathcal{V} ,\|\cdot\|_{\mathcal{V} }\right)$ are Banach spaces;\\
		\textbf{(ii)}  $\mathcal{V}$ is continuously embedded in $\mathcal{X}$;\\
		\textbf{(iii)}   If $u_0\in \mathcal{V} $, then $u(u_0;t,\cdot)\in \mathcal{V} $ for all $t \ge 0$.\\
		If the transition semigroup $\{P_t\}_{t \geq 0}$ admits a unique invariant measure on $\mathcal{X}$, then it admits at most one invariant measure on $\mathcal{V} $.
	\end{lemma}
	\begin{proof}
		We prove this lemma by contradiction. Assume that there exist two distinct invariant measures on $\mathcal{V}$, denoted by 
		$\pi_1$ and $\pi_2$. Let $\mathcal{I}:\mathcal{V} \to\mathcal{X}$ be the inclusion map, i.e., $\mathcal{I}(v)=v$ for all $v\in\mathcal{V}$. Since $\mathcal{V}$ is continuously embedded into $\mathcal{X}$, $\mathcal{I}$ is continuous. Thus, for any $B\in\mathscr{B}(\mathcal{X})$ we have $B\cap\mathcal{V}=\mathcal{I}^{-1}(B)\in\mathscr{B}(\mathcal{V})$. Define $\tilde{\pi}_1$ and $\tilde{\pi}_2$ on $\mathcal{X}$ by
		$$ \tilde{\pi}_1(B)=\pi_1(B\cap\mathcal{V}  )\quad \text{and}  \quad  \tilde{\pi}_2(B)=\pi_2(B\cap\mathcal{V}  ),\quad  \forall B\in   \mathscr{B}(\mathcal{X}). $$
		It is straightforward to verify that $\tilde{\pi}_1$ and $\tilde{\pi}_2$ are probability measures on $\mathcal{X}$.
		
			%
		%
		
		For any $\phi\in B_b(\mathcal{X})$,  we define $\hat{\phi}:  \mathcal{V} \to \mathbb{R}$ by 
		$  \hat{\phi}(v):=  \phi (\mathcal{I}(v))   $   for all $v\in\mathcal{V}$.
		Obviously, $\hat{\phi} $ is bounded.   It follows from the measurability of $\phi$ that $\phi^{-1}(\hat{B}) \in \mathscr{B}(\mathcal{X})$ for every $\hat{B}\in\mathscr{B}(\mathbb{R})$. Consequently,
		$$
		\hat{\phi}^{-1}(\hat{B})= \left\{v \in \mathcal{V} : \phi(\mathcal{I}(v)) \in \hat{B}\right\}=\mathcal{I}^{-1}\big(\phi^{-1}(\hat{B})\big) \in \mathscr{B}\left(\mathcal{V} \right),
		$$
		which means that  $\hat{\phi}\in B_b(\mathcal{V} )$. Since $u_0\in \mathcal{V} $ implies $u(u_0;t,\cdot)\in \mathcal{V} $, by the definition of   invariant measure, we have
		\begin{equation*} 
			\begin{aligned}
				\int_\mathcal{X} P_t \phi(u) \tilde{\pi}_1( \mathrm{ d}u)=	\int_{\mathcal{V} } P_t \hat{\phi}(u)  \pi_1( \mathrm{ d}u)=\int_{\mathcal{V} } \hat{\phi}(u)  \pi_1( \mathrm{ d}u)=\int_\mathcal{X}  \phi(u)  \tilde{\pi}_1( \mathrm{ d}u).
			\end{aligned}
		\end{equation*} 
		This implies that $\tilde{\pi}_1$ is an invariant measure on $\mathcal{X}$, and so does $\tilde{\pi}_2$.  It follows from  $\pi_1 \neq  \pi_2$  that  $\tilde{\pi}_1 \neq \tilde{\pi}_2$, which contradicts the uniqueness of the invariant measure on $\mathcal{X}$. This completes the proof.
		
	\end{proof}
	
	It is time to prove the uniqueness of invariant measure, which together with existence proved in Theorem \ref{th: invariant measure for the SCHE 1} implies ergodicity.
	\begin{theorem}  \label{th: invariant measure for the SCHE}
		Suppose that Assumptions \ref{assu:1}-\ref{assu:3 noise} hold, then the  invariant  
		measure $\pi$ for the SCHE \eqref{eq:CH_FractionalNoise} is unique  on $H_\alpha$. 
	\end{theorem} 
	\begin{proof}
		According to \cite[Theorem 3.3]{da1996stochastic}, the mild solution given by \eqref{eq: CH_FractionalNoise_solution_mild} admits an unique invariant measure $\pi$ on $\dot{H}^{-1}_\alpha$. Since $H_\alpha$ is a complete subspace of $\dot{H}^{-1}_\alpha$ that is continuously embedded into $\dot{H}^{-1}_\alpha$ and $u(t,\cdot)\in H_\alpha$ for all $t \geq 0$, it follows from Lemma \ref{lem: uniqueness invariant measure from big to small} and Theorem \ref{th: invariant measure for the SCHE 1} that there exists a unique invariant measure $\pi$ for the SCHE \eqref{eq:CH_FractionalNoise} on $H_\alpha$.  
	\end{proof}
	\begin{remark} \label{rem: no unique}
		It also follows from Theorem \ref{th: invariant measure for the SCHE} that the transition semigroup $ \left\{P_t\right\}_{t \geq 0} $ does not admit a unique invariant measure on $H$. In fact, given two distinct real numbers $\alpha_1$ and $\alpha_2$, There exist invariant measures $\pi_1$ and $\pi_2$ for $\{P_t\}_{t \geq 0}$ on $H_{\alpha_1}$ and $H_{\alpha_2}$, respectively. Using the method in Lemma \ref{lem: uniqueness invariant measure from big to small}, the measures $\pi_1$ and $\pi_2$ can be extended to invariant measures $\tilde{\pi}_1$ and $\tilde{\pi}_2$ on $H$, respectively. Hence, we have $\tilde{\pi}_1 \neq \tilde{\pi}_2$. It is worth noting that the invariant measure on $H$  becomes unique if Neumann boundary conditions are replaced by Dirichlet boundary conditions. The proof is left to the readers. 
	\end{remark}

	\section{Fully discrete scheme and uniform strong convergence} \label{section4}
	In this section, following the methods developed in \cite{deng2025}, we introduce a fully discrete scheme for original equation. We use a finite difference method for the spatial discretization and a tamed exponential Euler method for the temporal discretization.  
	We establish the uniform strong convergence  of our proposed scheme. \par 
	For the spatial discretization, let  $N\in\mathbb{N}^+$  and let $h=\pi / N$ be the spatial step-size. We denote by   $x_i=\left(i-\frac{1}{2}\right) h$, $i=1,2, \ldots, N$,  the spatial grid points.
	We denote  
	$l_N^p:=\left\{  \vec{u}= \left(u_1, \ldots, u_N\right)^{\top} \Big| \|\vec{u}\|_{l_N^p}<\infty  \right\},  $
	where
	$$
	\|\vec{u}\|_{l_N^p}= \begin{cases}\left(\frac{1}{N} \sum_{i=1}^{N}\left|u_i\right|^p\right)^{\frac{1}{p}}, & 1 \leq p<\infty, \\ \max _{1 \leq i \leq N}\left|u_i\right|, & p=\infty.\end{cases} 
	$$   
	For any   $\vec{u} \in l_N^2$, we define second and fourth order difference operators with the homogeneous Neumann boundary condition by
	$$
	\delta_h^2 u_i :=\left\{\begin{array}{ll}
		\left(-u_1 +u_2 \right) / h^2, & i=1, \\
		\left(u_{i-1} -2u_i +u_{i+1} \right) / h^2, & i=2,3, \cdots, N-1, \\
		\left(u_{N-1} -u_N \right) / h^2, & i=N,
	\end{array} \quad \text { and } \quad \delta_h^4 u_i :=\delta_h^2\left(\delta_h^2 u_i \right),\ i=1,2, \cdots, N .\right.
	$$
	See \cite[Chapter 9]{Sun2023} for details.
	Considering Eq. \eqref{eq:CH_FractionalNoise} at the point $(t,x_i)$, we approximate $\Delta$ and $\Delta^2$ by $\delta_h^2$ and $\delta_h^4$.   $W(t,x_i)$ is approximated by $\sum_{j=1}^{N-1}\phi_j(x_i)\beta_j(t)$. In this way, we establish a scheme given by
	\begin{equation}  \label{eq: spatial discretization pre}
		\mathrm{d} u^N\left(t, x_i\right)+\delta_h^4 u^N\left(t, x_i\right) \mathrm{d} t=\delta_h^2 f\left(u^N\left(t, x_i\right)\right) \mathrm{d} t+\sigma \sum_{j=1}^{N-1}\phi_j(x_i)\mathrm{ d}\beta_{j}(t)  .
	\end{equation}
	Note that  \eqref{eq: spatial discretization pre} is defined only at the grid points. In order to construct its spatial continuous extension, it is necessary to introduce two operators.  We introduce the operator  $\mathcal{S}_N:\ H \to l_N^2$  defined by
	$$   \mathcal{S}_N v:=  (v(x_1),v(x_2),\cdots,v(x_N))^\top, \quad v\in H,$$
	and  $P_N:\ l^2_N \to H$,  defined for $\vec{v}=(v_1,v_2,\cdots,v_N)^\top\in l^2_N $   by
	$$  P_N\vec{v}(x) := \begin{cases}v_1, & x \in\left[0, x_1\right] \\ 
		v_i+\frac{x-x_i}{h}\left(v_{i+1}-v_i\right), & x \in\left[x_i, x_{i+1}\right], i=1,2, \cdots, N-1, \\v_N, & x \in\left[x_N, \pi\right].\end{cases}     $$  
	We additionally introduce
	$$
	U_t^N=\left(u^N\left(t, x_1\right), u^N\left(t, x_2\right), \cdots, u^N\left(t, x_N\right)\right)^{\top} \quad \text { and } \quad \phi_{N, j} =  \mathcal{S}_N \phi_j,
	$$
	where $\left\{\phi_j\right\}_{j=0}^\infty$  is  an orthonormal basis  of $\dot{H} $.
	Then it follows from \eqref{eq: spatial discretization pre}  that
	\begin{subequations}\label{eq:eq:discrete_CH_FractionalNoise_SDE_matrix}
		\begin{numcases}{}
			\mathrm{d} U_t^N+A_N^2 U_t^N \mathrm{d} t=A_N F_N(U_t^N) \mathrm{d} t+  \sigma  \sum_{j=1}^{N-1}\phi_{N, j} \mathrm{d} \beta_{j}(t)  ,\quad    t>0,\label{eq:eq:discrete_CH_FractionalNoise_SDE_matrix a} \ \\
			U_0^N=u^N_0, \label{eq:eq:discrete_CH_FractionalNoise_SDE_matrix b}  \
		\end{numcases}
	\end{subequations}  
	where 
	$$
	A_N:=\frac{N^2}{\pi^2}\left(\begin{array}{ccccc}
		-1 & 1 & 0 & \cdots & 0 \\
		1 & -2 & 1 & \ddots & \vdots \\
		0 & \ddots & \ddots & \ddots & 0 \\
		\vdots & \ddots & 1 & -2 & 1 \\
		0 & \cdots & 0 & 1 & -1
	\end{array}\right),\quad  u_0^N=\left(\begin{array}{c}
		u^N_{0,1} \\
		u^N_{0,2} \\
		\vdots \\
		u^N_{0, N}
	\end{array}\right):=\mathcal{S}_N u_0,
	$$
	and   
	$F_N(\vec{v}):=(f(v_1),f(v_2),\cdots,f(v_N))^\top $ for $\vec{v}=(v_1,v_2,\cdots,v_N)^\top\in l^2_N.$   
	The matrix $-A_N$   is positive semi-definite  and admits the eigenpairs $\left\{\lambda_{N,j}, \phi_{N, j}\right\}_{j=0}^{N-1}$, where $\lambda_{N,j}=  4N^2\pi^{-2} \sin^2(\frac{j\pi}{2N})  $.
	Moreover,  $\left\{\phi_{N, j}\right\}_{j=0}^{N-1}$ forms an orthonormal basis in  $l^2_N$.  
	Similar to the continuous case, we introduce the discrete $W^{\gamma,2}$-norm defined by
	$$   \| \vec{v} \|_{w^{\gamma,2}_N}=\left\{\begin{array}{ll}\left( \|  \vec{v}  \|_{l_N^2}^2+ \big\|\left(-A_N\right)^{\frac{1}{2}}  \vec{v} \big\|_{l_N^2}^2 \right)^\frac{1}{2}, & \gamma=1, \\ 
		\left(   \| \vec{v} \|_{l_N^2}^2+  \int_{\mathcal{O}} \int_{\mathcal{O}}  \frac{\left|	P_N \vec{v} (x)-P_N \vec{v} (y)\right|^2}{|x-y|^{1+2\gamma}}\mathrm{d} x \mathrm{d} y  \right)^\frac{1}{2} , &0<\gamma<1,\end{array}\right. 
	$$
	where  $\left(-A_N\right)^\frac{1}{2}:=\sum_{j=0}^{N-1}\lambda_{N,j}^\frac{1}{2} \phi_{N, j} \phi_{N, j}^\top$.  \par 
	We next state several important properties. 
	\begin{lemma}  \label{lem: property A_N}
		Suppose $N \geq 2$. For any $\vec{u}=\left(u_i\right)_{i=1}^N, \vec{v}=\left(v_i\right)_{i=1}^N, \vec{w}=\left(u_i v_i\right)_{i=1}^N \in l_N^2$, 
		$ \gamma,\ \gamma_1,\ \gamma_2\in[0,\infty),$ and $\gamma_3\in[0,1)$ satisfying $ \gamma_1-\gamma_2+2\gamma_3 >0$, 	the following inequalities hold:
		\begin{align}
			&\| \vec{u}\|_{l_N^{\infty}}   \leq C\|\vec{u}\|_{w^{1,2}_N}  ,   \label{ineq: l_infty to l2} \\
			&\| \vec{w}\|_{w^{1,2}_N}   \leq C\|\vec{u}\|_{w^{1,2}_N}\|\vec{v}\|_{w^{1,2}_N},  \label{ineq: uv to u v} \\  
			&\|\vec{u}\|_{l_N^{6}}   \leq
			C\Big(\left\|A_N \vec{u}\right\|_{l_N^2}^{\frac{1}{6}}\|\vec{u}\|_{l_N^2}^{\frac{5}{6}}+\|\vec{u}\|_{l_N^2}\Big)  ,     \label{ineq: l_6 to l2} \\
			&\left\|(-A_N)^\gamma e^{-A_N^2t}\vec{u}\right\|_{l_N^{2}}  \leq C\Big(t^{-\frac{\gamma}{2}}e^{-\frac{\lambda_{N,1}^2}{2}t}\left\|\vec{u}\right\|_{l_N^2}+g(\gamma)\big|\left\langle \vec{u}, \phi_{N,0}\right\rangle_{l_N^2}\big|\Big),\   t>0, \  
			\label{ineq: e^{-A_N^2t}(-A_N)^r to l2} \\ 
			&\left\|(-A_N)^{\gamma_1} e^{-A_N^2t}\left( I- e^{-A_N^2s}\right)\vec{u}\right\|_{l_N^{2}}  \leq C \Big(t^{-\frac{\gamma_1-\gamma_2+2\gamma_3}{2}}s^{\gamma_3}e^{-\frac{\lambda_{N,1}^2}{2}t}\left\|(-A_N)^{\gamma_2}\vec{u}\right\|_{l_N^2}+g(\gamma_1)\big|\left\langle \hat{u}, \phi_{N,0}\right\rangle_{l_N^2}\big|\Big),		\nonumber\\
			&\qquad 	\qquad 	\qquad 	\qquad 	\qquad 	\qquad 	\qquad 	\qquad 	 	\qquad 	\qquad 	\qquad 	\qquad 	\qquad 	\qquad 	\qquad 	\qquad 	\qquad s,t>0, 	\label{ineq: (-A_N)^re^{-A_N^2t}(I- e^{-A_N^2t}) to l2}  
		\end{align}
		where  $g$ is given by Lemma \ref{lem: e {-A^2 t}} and $C>0$ is a constant independent of $s$, $t$,  $\vec{u}$, $\vec{v}$, and $\vec{w}$.
	\end{lemma} 
	Inequalities \eqref{ineq: l_infty to l2}-\eqref{ineq: l_6 to l2} can be found in Theorem 2.1 of \cite{deng2025} and the proof of inequalities \eqref{ineq: e^{-A_N^2t}(-A_N)^r to l2}-\eqref{ineq: (-A_N)^re^{-A_N^2t}(I- e^{-A_N^2t}) to l2}  is similar to that of   \eqref{ineq: e^{-A^2t}(-A)^r to l2}-\eqref{ineq: (-A)^re^{-A^2t}(I- e^{-A^2t}) to l2}. Therefore, the proof of Lemma \ref{lem: property A_N} is omitted.
	
	We next introduce the temporal discretization. Let $\tau\in(0,1)$ be the temporal step-size and $t_i=i\tau$, $i\in\mathbb{N}$ be the temporal grid points. We  introduce a notation $\lfloor t\rfloor:=t_i $ for $t \in\left[t_i, t_{i+1}\right)$. We  employ the tame exponential Euler method for the temporal discretization, yielding the following  fully discrete scheme
	\begin{equation} \label{eq: mild_solution_full_discrete}
		\begin{aligned}
			U^{\tau,N}_t= & e^{-A_N^2 t}u_0^N+\int_0^t A_N e^{-A_N^2(t-\lfloor s\rfloor)} \tilde{F}_N(	U^{\tau,N}_{\lfloor s\rfloor})  \mathrm{d} s +\sigma  \sum_{j=1}^{N-1}   \int_0^t e^{-A_N^2(t-\lfloor s\rfloor)} \phi_{N, j} \mathrm{d} \beta_{j}(s),
		\end{aligned}
	\end{equation}  
	Here $	U^{\tau,N}_t\in l^2_N$ and  the function  $\tilde{F}_N(\cdot)$ is given by
	$$ \tilde{F}_N(\vec{v}):= {F_N(\vec{v})}\Big/{\Big(1+\tau\left\|\vec{v}\right\|_{w^{1,2}_N}^{12}\Big) }, \quad  \vec{v}\in l^2_N.   $$
	\begin{remark}
		We are aware that some recent studies (e.g., \cite{deng2025}, \cite{Qi2020}, \cite{qi2022strong}) have proved strong convergence of fully discrete schemes for the SCHE on finite-time intervals.   In our work, we aim to extend these convergence results from finite intervals to the infinite-time case. Inspired by \cite{WangY2024}, we enhance the tamed term in our numerical scheme by replacing  $1/\big(1+\tau\|F_N(\cdot)\|_{l^2_N}\big)$ with  $1/\big(1+\tau\|\cdot\|_{w^{1,2}_N}^{12} \big)$. The design of novel tamed strategies has attracted increasing attention, as exemplified by recent works such as \cite{Jiang2025} and \cite{LiuShen2025}. In addition, the fully discrete scheme \eqref{eq: mild_solution_full_discrete} can be rewritten as
		\begin{equation} \label{eq: mild_solution_full_discrete discrete}
			\begin{aligned}
				U^{\tau,N}_{t_{m+1}}= & e^{-A_N^2 \tau}	U^{\tau,N}_{t_{m}}+\tau A_N e^{-A_N^2\tau} \tilde{F}_N(		U^{\tau,N}_{t_{m}}) 
				+ \sigma   e^{-A_N^2\tau}   \Delta \beta_m  , \quad m\in \mathbb{N},
			\end{aligned}
		\end{equation}   
		where  $\Delta\beta_m=\sum_{j=1}^{N-1}\phi_{N, j}\Delta \beta_{j,m} $ and  the Wiener increments are defined as  $ \Delta \beta_{j,m}:= \beta_{j}(t_{m+1}) -\beta_{j}(t_{m}),\ m\in \mathbb{N}$.  This formulation is straightforward to implement computationally.
	\end{remark}
	
	In order to  introduce the spatially continuous form of \eqref{eq: mild_solution_full_discrete}, we define  $$\kappa_N(x):=\left\{
	\begin{array}{l} x_i, \  x \in [x_i-\frac{h}{2},\ x_i+\frac{h}{2} ),\ i=1,2, \cdots, N-1, \\
		x_N,\ x \in [x_N-\frac{h}{2}, x_N+\frac{h}{2} ].
	\end{array} 
	\right. 
	$$  
	Then by  the interpolation of $	U^{\tau,N}_t$, we  obtain
	\begin{equation}
		\begin{aligned}
			u^{\tau,N}(t, x)=  P_N	U^{\tau,N}_t (x) 
			= & \int_\mathcal{O} G_t^N(x, y) u_0\left(\kappa_N(y)\right) \mathrm{d} y+\int_0^t \int_\mathcal{O} \frac{\Delta_N G_{t-\lfloor s\rfloor}^N(x, y) f\left(u^{\tau, N}\left(\lfloor s\rfloor, \kappa_N(y)\right)\right)}{1+\tau\left\|U^{\tau,N}_{\lfloor s\rfloor}\right\|_{w^{1,2}_N}^{12}} \mathrm{ d} y \mathrm{ d} s \\
			&+\sigma\sum_{j=1}^{N-1} \int_0^t \int_{\mathcal{O}} G_{t-\lfloor s\rfloor}^N(x, y)\phi_j\left(\kappa_N(y)\right) \mathrm{d} y \mathrm{d}  \beta_j(s)  , \quad (t, x) \in[0, \infty) \times \mathcal{O},
		\end{aligned}
	\end{equation}    
	where $G_t^N(x, y)=\sum_{j=0}^{N-1} e^{-\lambda_{N, j}^2 t} \bar{\phi}_{N, j}(x) \phi_j\left(\kappa_N(y)\right)$ with $ \bar{\phi}_{N, j}:=P_N  \phi_{N, j}$. 
	Here $\Delta_N$  is defined by
	$$  \Delta_N v :=  P_N A_N   \mathcal{S}_N   v,\quad    v\in H.$$  
	It follows $\Delta_N \phi_j\left(\kappa_N(y)\right)= -\lambda_{N,j} \phi_j\left(\kappa_N(y)\right)$ that $$ \Delta_N G_{t-s}^N(x, y) =-\sum_{j=1}^{N-1}\lambda_{N,j} e^{-\lambda_{N,j}^2 t} \bar{\phi}_{N,j}(x) \phi_j\left(\kappa_N(y)\right).$$  
	The difference between $\Delta_NG^N$ and $\Delta G $  shown in the following lemma  is fundamental for the subsequent error analysis. Its proof can be found in Appendix C.
	\begin{lemma}   
		Given $\alpha_1 \in (0,1)$, there exists a constant $C>0$ such that  
		\begin{align}
			& \sup_{x \in\mathcal{O}}\int_{\mathcal{O}}\left|\Delta_N G_t^N(x, y)-\Delta G_t(x, y)\right| \mathrm{d} y   \leq Ce^{-\lambda_{N,1}t}  t^{-\frac{2\alpha_1+5}{8}} N^{-\alpha_1} , \quad t >0. \label{es: Delta_N G_s^N-Delta G_s(x, y) } 
		\end{align} 
	\end{lemma}

	\subsection{Uniform moment boundedness of the numerical solution}
	In this subsection, we establish the uniform moment boundedness of the numerical solution. To this end, we first introduce   $\mathscr{O}^{N}_t:= \sigma  \sum_{j=1}^{N-1} \int_0^t e^{-A_N^2(t-s)} \phi_{N, j} \mathrm{d} \beta_{j}(s) $  and  
	$\mathscr{O}^{\tau,N}_t:= \sigma  \sum_{j=1}^{N-1} \int_0^t e^{-A_N^2(t-\lfloor s\rfloor)} $ $\phi_{N, j} \mathrm{d} \beta_{j}(s) $.
	In the next lemma, we present some properties of $\mathscr{O}^N_t$ and $\mathscr{O}^{\tau,N}_t$. Since the proofs are similar to those of Theorem \ref{th: u_regularity}, we omit them here for brevity. 
	
	\begin{lemma}\label{lem: O^{M,N}_continuous}  
		For any $p \geq 1$, $0\leq\beta<\frac{3}{2}$, and $\gamma<\frac{1}{2}$, it holds that
		\begin{align} 
			&	\sup_{t\geq 0}	\mathbb{E}\left[  \big\|(-A_N)^{\frac{\beta}{2}}\mathscr{O}^{N}_t \big\|^p_{l^2_N}\right]+	\sup_{t\geq 0} \mathbb{E}\left[ \big\| \left(-A_N\right)^{\frac{1}{2}}   \mathscr{O}^{N}_t\big\|^p_{w^{\gamma,2}_N}\right] <  \infty,\label{ineq: regularity of O^{N}(t,x)} \\
			&	 \sup_{t\geq 0}	\mathbb{E}\left[  \big\|(-A_N)^{\frac{\beta}{2}}\mathscr{O}^{\tau,N}_t \big\|^p_{l^2_N}\right]+	\sup_{t\geq 0} \mathbb{E}\left[ \big\| \left(-A_N\right)^{\frac{1}{2}}   \mathscr{O}^{\tau,N}_t\big\|^p_{w^{\gamma,2}_N} \right]<  \infty.\label{ineq: regularity of O^{M,N}(t,x)} 
		\end{align}
		Moreover,  denote $o^{N}(t,x):=P_N\mathscr{O}^{N}_t(x)$ and $o^{\tau,N}(t,x):=P_N\mathscr{O}^{\tau,N}_t(x)$, then for  any   $0\leq s<t$ it holds that
		\begin{align} 
			&\mathbb{E}\left[\big\|(-A_N)^{\frac{\beta}{2}}\left(\mathscr{O}^{\tau,N}_t-\mathscr{O}^{\tau,N}_s\right) \big\|^p_{l^2_N}\right]+	\mathbb{E}\left[\big\|(-A_N)^{\frac{\beta}{2}}\left(\mathscr{O}^{\tau,N}_t-\mathscr{O}^{\tau,N}_s\right) \big\|^p_{l^2_N}\right]
			\leq C(t-s)^{ \left(\frac{3}{8}-\frac{\beta }{4}
				-\epsilon \right) p},\label{ineq: continuous of O^{M,N}(t,x)}\\ 
			\vspace{2em}
			&	\sup_{t\geq 0}\mathbb{E}\left[|o^{N}(t,x)-o^{N}(t,z)|^p\right]+	\sup_{{t\geq 0}}\mathbb{E}\left[|o^{\tau,N}(t,x)-o^{\tau,N}(t,z)|^p\right] 	 \leq C|x-z|^p, \quad x,\ z \in \mathcal{O}, \label{ineq: futher property of O_t}
		\end{align}
		where $C>0$ is a constant independent of $s,\ t,\ x,\ z,\ \tau $ and $N$, and $\epsilon$ is an arbitrarily small positive number.
	\end{lemma} 
	
	We next establish the uniform moment boundedness of $ u^{\tau, N}(t_m,x)$.
	
	\begin{theorem} \label{thm: the moment boundedness of the numerical solution U^{M,N}}
		Suppose that Assumptions \ref{assu:1}-\ref{assu:3 noise} hold.
		Then for any $p \geq 2$, it holds that
		\begin{equation}\label{es: U^{M,N} }
			\sup _{\tau \in (0,1), N \in \mathbb{N}} \sup _{i\geq 0 } \mathbb{E}\left[\left\|u^{\tau, N}(t_m,\cdot)\right\|_{l^\infty_N}^p\right]<\infty\quad \text{and}	\quad \sup _{\tau \in (0,1), N \in \mathbb{N}} \sup _{i\geq 0 } \mathbb{E}\left[\left\|U_{t_i}^{\tau, N}\right\|_{w^{1,2}_N}^p\right]<\infty.
		\end{equation}
	\end{theorem} 
	\begin{proof}  
		We introduce the process  $V^{\tau, N}_t:=U^{\tau, N}_t-\mathscr{O}^{\tau, N}_t$, which satisfies
		\begin{equation}\label{eq: perturbed problem 1 M N}
			\left\{  
			\begin{aligned}
				&\mathrm{d} V^{\tau, N}_t+A_N^2 V^{\tau, N}_t \mathrm{d} t=A_Ne^{-A_N^2(t-\lfloor t\rfloor)} \tilde{ F}_N(U^{\tau, N}_{\lfloor t\rfloor}) \mathrm{d} t, \quad t>0,\\
				&V^{\tau, N}_0=u_0^N.
			\end{aligned}
			\right.
		\end{equation} 
		Taking the inner product on \eqref{eq: perturbed problem 1 M N} with $ V^{\tau, N}_t $, we obtain 
		\begin{equation} \label{es: V^{M,N}  1}
			\begin{aligned}
				\frac{1}{2}\frac{\mathrm{d}}{\mathrm{d} t} \left(\big\| V^{\tau, N}_t\big\|^2_{l^2_N}\right) 
				=& -\big\| A_N V^{\tau, N}_t\big\|^2_{l^2_N}+ \big\langle\tilde{F}_N(U^{\tau, N}_{\lfloor t\rfloor}),   A_N\big(e^{-A_N^2(t-\lfloor t\rfloor)}-I\big)V^{\tau, N}_t \big\rangle_{l^2_N}  \\  
				& +    \big\langle \tilde{F}_N(U^{\tau, N}_{\lfloor t\rfloor})-\tilde{F}_N(U^{\tau, N}_t),  A_N  V^{\tau, N}_t \big\rangle_{l^2_N}  +\big \langle \tilde{F}_N(U^{\tau, N}_t), A_N  V^{\tau, N}_t  \big\rangle _{l^2_N}   .
			\end{aligned}
		\end{equation}
		Using Hölder's inequality, the estimate \eqref{ineq: (-A_N)^re^{-A_N^2t}(I- e^{-A_N^2t}) to l2}  with $\gamma_1=\frac{1}{2}$, $\gamma_2=1$, and $\gamma_3=\frac{1}{4}$, \eqref{ineq: uv to u v}  and Young's inequality, we obtain
		\begin{equation} \label{es: <  F_N(U^{M,N}(s)), A_N(I-e) V^{M,N}(s)>} 
			\begin{aligned}
				&	\big \langle \tilde{F}_N(U^{\tau, N}_{\lfloor t\rfloor} ), A_N\big(I-e^{-A_N^2(t-\lfloor t\rfloor)}\big) V^{\tau, N}_t  \big\rangle _{l^2_N} \\
				\leq &\frac{ 1 }{1+\tau \big\|  U^{\tau, N}_t  \big\|_{w^{1,2}_N}^{12}}\cdot \big\|\left(-A_N\right)^{\frac{1}{2}} F_N(U^{\tau, N}_t )\big\|_{l^2_N} \big\|  \left(-A_N\right)^{\frac{1}{2}}\big(I-e^{-A_N^2(t-\lfloor t\rfloor)}\big) V^{\tau, N}_t\big\| _{l^2_N}\\
				\leq &\frac{ 1}{1+\tau \big\|  U^{\tau, N}_t  \big\|_{w^{1,2}_N}^{12} }\cdot (t-\lfloor t\rfloor)^\frac{1}{4} \big\|\left(-A_N\right)^{\frac{1}{2}} F_N(U^{\tau, N}_{\lfloor t\rfloor} )\big\| _{l^2_N} \big\|  A_N  V^{\tau, N}_t\big\| _{l^2_N} \\
				\leq &\frac{1}{8} \big\|A_N  V^{\tau, N}_t \big\|^2_{l^2_N}  +C.
			\end{aligned}
		\end{equation}  
		In order to estimate the third term on the right-hand side of \eqref{es: V^{M,N} 1}, we need to establish the continuity of $U^{\tau,N}_t$.
		We   rewrite \eqref{eq: mild_solution_full_discrete} as 
		$$
		U_t^{\tau, N}=e^{-A_N^2(t-\lfloor t\rfloor)}U_{\lfloor t\rfloor}^{\tau, N}
		+(t-\lfloor t\rfloor)  A_Ne^{-A_N^2(t-\lfloor t\rfloor)}\tilde{F}_N\big(U_{\lfloor t\rfloor}^{\tau, N}\big)  
		+\mathscr{O}_t^{\tau, N}-e^{-A_N^2(t-\lfloor t\rfloor)} \mathscr{O}_{\lfloor t\rfloor}^{\tau, N}.
		$$
		Combining this with  \eqref{ineq: e^{-A_N^2t}(-A_N)^r to l2}, we obtain 
		\begin{equation} \label{es: I_{Omega_{R, t_i}} U_t^{M, N}}
			\begin{aligned}
				\big\|U_t^{\tau, N}\big\|_{w^{1,2}_N}  
				\leq &  \left\|e^{-A_N^2(t-\lfloor t\rfloor)} U_{\lfloor t \rfloor}^{\tau, N} \right\|_{w^{1,2}_N} 
				+ (t-\lfloor t\rfloor) \left\|A_Ne^{-A_N^2(t-\lfloor t\rfloor)}\tilde{F}_N\big(U_{\lfloor t\rfloor}^{\tau, N}\big)\right\|_{w^{1,2}_N}  \\
				& +\big\| \mathscr{O}_t^{\tau, N}\big\|_{w^{1,2}_N} 
				+\left\|e^{-A_N^2(t-\lfloor t\rfloor)}  \mathscr{O}_{\lfloor t\rfloor}^{\tau, N}\right\|_{w^{1,2}_N}   \\
				\leq &  C \bigg(  \left\|U_{\lfloor t\rfloor}^{\tau, N} \right\|_{w^{1,2}_N} 
				+(t-\lfloor t\rfloor)^\frac{1}{2}e^{-\frac{1}{2}(t-\lfloor t\rfloor) } \left\|\tilde{F}_N\big(U_{\lfloor t\rfloor}^{\tau, N}\big)\right\|_{w^{1,2}_N}   +\big\| \mathscr{O}_t^{\tau, N}\big\|_{w^{1,2}_N} 
				+\left\| \mathscr{O}_{\lfloor s\rfloor}^{\tau, N}\right\|_{w^{1,2}_N} \bigg) \\ 
				\leq & C\left(  \left\|U_{\lfloor t\rfloor}^{\tau, N} \right\|_{w^{1,2}_N} +1+\mathscr{R}_t  \right),
			\end{aligned}
		\end{equation}
		where $\mathscr{R}_t:=\big\| \mathscr{O}_t^{\tau, N}\big\|_{w^{1,2}_N} +\big\| \mathscr{O}_{\lfloor t\rfloor}^{\tau, N}\big\|_{w^{1,2}_N} $.
		Utilizing \eqref{eq: mild_solution_full_discrete}, we can deduce
		\begin{equation} \label{eq: U_t^{M, N} continuous   1}
			\begin{aligned}
				& \left\|U_t^{\tau, N}-U_{\lfloor t\rfloor}^{\tau, N}\right\|_{l_N^2} \\
				\leq&\left\|\left(e^{-A_N^2t}- e^{-A_N^2\lfloor t\rfloor}\right)u_0^N\right\|_{l^2_N} 
				+\int_{0}^{\lfloor t\rfloor}\left\|A_Ne^{-A_N^2(\lfloor t\rfloor-\lfloor s\rfloor)}\left(I-e^{-A_N^2(t-\lfloor t\rfloor)}\right) \tilde{F}_N\big(U_{\lfloor s\rfloor}^{\tau, N}\big) \right\|_{l^2_N}  \mathrm{d}s\\
				& +(t-\lfloor t\rfloor)\left\|A_Ne^{-A_N^2(t-\lfloor t\rfloor)}\tilde{F}_N\big(U_{\lfloor t\rfloor}^{\tau, N}\big)\right\|_{l^2_N}  
				+\left\|\mathscr{O}_t^{\tau, N}-\mathscr{O}_{\lfloor t\rfloor }^{\tau, N}\right\|_{l_N^2} \\
				=&: \mathscr{K}_1+\mathscr{K}_2+\mathscr{K}_3+\left\|\mathscr{O}_t^{\tau, N}-\mathscr{O}_{\lfloor t\rfloor }^{\tau, N}\right\|_{l_N^2} .
			\end{aligned}
		\end{equation} 
		By using \eqref{ineq: (-A_N)^re^{-A_N^2t}(I- e^{-A_N^2t}) to l2} with $\gamma_1=0,\ \gamma_2=1$ and $\gamma_3=\frac{1}{2}$, and Assumption \ref{assu:2}, we have
		\begin{equation} \label{es: eq: U_t^{M, N} continuous  I1}
			\begin{aligned}
				\mathscr{K}_1 = \left\|e^{ -A_N^2 \lfloor t\rfloor}\left(e^{-A_N^2( t-\lfloor t\rfloor)}-I\right)u_0^N\right\|_{l^2_N}    
				\leq C( t-\lfloor t\rfloor)^\frac{1}{2} \left\|  A_N u_0^N\right\|_{l^2_N} 
				\leq C\tau^\frac{1}{2}.
			\end{aligned}
		\end{equation}
		It follows from \eqref{ineq: (-A_N)^re^{-A_N^2t}(I- e^{-A_N^2t}) to l2} with $\gamma_1=1,\ \gamma_2=\frac{1}{2}$ and $\gamma_3=\frac{1}{2}$ that
		\begin{equation} \label{es: eq: U_t^{M, N} continuous  I2}
			\begin{aligned}
				\mathscr{K}_2    	\leq&  C\int_0^{ \lfloor t\rfloor }(\lfloor t\rfloor-\lfloor s\rfloor)^{-\frac{3}{4}}(t-\lfloor t\rfloor)^{\frac{1}{2}}
				e^{-\frac{1}{2}(\lfloor t\rfloor-\lfloor s\rfloor)}	\left\|   (-A_N)^{\frac{1}{2}}\tilde{F}_N\big(U_{\lfloor s\rfloor}^{\tau, N} \big) \right\|_{l^2_N}  \mathrm{d}s \\
				\leq&  C\int_0^{ \lfloor t\rfloor }\frac{(\lfloor t\rfloor-\lfloor s\rfloor)^{-\frac{3}{4}} 
					e^{-\frac{1}{2}(\lfloor t\rfloor-\lfloor s\rfloor)}\tau^{\frac{1}{2}}	\big\|   U_{\lfloor s\rfloor}^{\tau, N} \big\|_{w^{1,2}_N}^3    }{1+\tau\left\| U_{\lfloor s\rfloor}^{\tau, N}\right\|_{w^{1,2}_N}^{12}} \mathrm{d}s 	\leq  C\tau^{\frac{1}{4} } .
			\end{aligned}
		\end{equation} 
		Similarly, by \eqref{ineq: (-A_N)^re^{-A_N^2t}(I- e^{-A_N^2t}) to l2} and \eqref{ineq: uv to u v} we can obtain
		\begin{equation}\label{es: eq: U_t^{M, N} continuous  I3}
			\begin{aligned}
				\mathscr{K}_3
				\leq 
				\frac{(t-\lfloor t\rfloor )^{ \frac{3}{4}}  \left(1+	\big\|   U_{\lfloor s\rfloor}^{\tau, N} \big\|_{w^{1,2}_N}^3\right)    }{1+\tau\big\| U_{\lfloor s\rfloor}^{\tau, N} \big\|_{w^{1,2}_N}^{12}}   \leq 	C \tau^\frac{1}{2}. 
			\end{aligned}
		\end{equation}
		By substituting \eqref{es: eq: U_t^{M, N} continuous  I1}-\eqref{es: eq: U_t^{M, N} continuous  I3} into \eqref{eq: U_t^{M, N} continuous   1}, we obtain  
		\begin{equation}
			\begin{aligned} \label{es: I_{Omega_{R, t_i}} U_t^{M, N}-U_t^{M, N}}
				\left\|	U_s^{M, N}-U_{\lfloor s\rfloor}^{M, N}\right\|_{l^2_N}  
				\leq& C\tau^{\frac{1}{4} } +
				\left\|\mathscr{O}_s^{\tau, N}- \mathscr{O}_{\lfloor s\rfloor}^{\tau, N}\right\|_{l^2_N},
			\end{aligned}
		\end{equation}
		which combined with Young's inequality, \eqref{es: f(x)-f(y)| }, \eqref{ineq: l_infty to l2} and \eqref{es: I_{Omega_{R, t_i}} U_t^{M, N}}  yields
		\begin{equation} \label{es: I_3   U^MN}
			\begin{aligned}  
				&\big\langle \tilde{F}_N(U^{\tau, N}_{\lfloor t\rfloor})-\tilde{F}_N(U^{\tau, N}_t),  A_N  V^{\tau, N}_t  \big\rangle_{l^2_N}\\
				\leq&\frac{1}{8}\big\| A_N V^{\tau, N}_t\big\|^2_{l^2_N}+C	\left\| \tilde{F}_N(U^{\tau, N}_{\lfloor t\rfloor})- \tilde{F}_N(U^{\tau, N}_t)\right\|_{l^2_N}^2\\ 
				\leq&\frac{1}{8}\big\| A_N V^{\tau, N}_t\big\|^2_{l^2_N}+  \frac{C\left(  1+ \left\|U_{\lfloor t\rfloor}^{M, N} \right\|_{w^{1,2}_N}^4+\mathscr{R}_t^4  \right) \left( \tau^{\frac{1}{2} } +
					\left\|\mathscr{O}_t^{\tau, N}- \mathscr{O}_{\lfloor t\rfloor}^{\tau, N}\right\|_{l^2_N}^2\right)}{1+\tau^2 \left\|  U^{\tau, N}_{\lfloor t\rfloor}  \right\|^{24}}\\
				\leq& \frac{1}{8}\big\| A_N V^{\tau, N}_t\big\|^2_{l^2_N}+ C\left( 1+\mathscr{R}_t^8+\tau^{-\frac{1}{3}} \left\|\mathscr{O}_t^{\tau, N}- \mathscr{O}_{\lfloor t\rfloor}^{\tau, N}\right\|_{l^2_N}^2+\left\|\mathscr{O}_t^{\tau, N}- \mathscr{O}_{\lfloor t\rfloor}^{\tau, N}\right\|_{l^2_N}^4\right).
			\end{aligned}
		\end{equation} 
		Then it follows from \eqref {es: V^{M,N}  1}, \eqref{es: <  F_N(U^{M,N}(s)), A_N(I-e) V^{M,N}(s)>} and \eqref{es: I_3   U^MN}
		\begin{equation} \label{es: V^{M,N}  2}
			\begin{aligned}
				\frac{1}{2}\frac{\mathrm{d}}{\mathrm{d} t} \left(\big\| V^{\tau, N}_t\big\|^2_{l^2_N}\right) 
				=& -\frac{3}{4}\big\| A_N V^{\tau, N}_t\big\|^2_{l^2_N}
				+ C\left( 1+\mathscr{R}_t^8+\tau^{-\frac{1}{3}} \left\|\mathscr{O}_t^{\tau, N}- \mathscr{O}_{\lfloor t\rfloor}^{\tau, N}\right\|_{l^2_N}^2+\left\|\mathscr{O}_t^{\tau, N}- \mathscr{O}_{\lfloor t\rfloor}^{\tau, N}\right\|_{l^2_N}^4\right)\\
				& + \frac{1 }{1+\tau\left\|U^{\tau, N}_{\lfloor t\rfloor}\right\|_{w^{1,2}_N}^{12}} \cdot \big \langle F_N(U^{\tau, N}_t), A_N  V^{\tau, N}_t  \big\rangle _{l^2_N}  .  \\
			\end{aligned}
		\end{equation}
		Similar to the treatment of \eqref{es: Y_t 1}, we can obtain
		\begin{equation*} 
			\begin{aligned}
				&   	\|V^{\tau, N}_t\|_{L^2} + \left\|(-A_N)^\frac{1}{2}V^{\tau, N}_t\right\|_{l^2_N}  \\
				\leq&   	C \Big( 1+\|u^N_0\|_{l^\infty_N}^{20}+\left\| (-A_N)^{\frac{1}{2}} u^N_0\right\|_{l^2_N}^{10}+\mathbb{I}_1\big(t, 3/4,  \lambda_{N,1}^2/2,|K_4 |^{3}	+\left\| \mathscr{O}^{\tau,N} \right\|_{l^6_N}^3 \big) 
				\\
				&\qquad+\mathbb{I}_2\big(t, 3/4,  \lambda_{N,1}^2/2,
				|K_4 |^{5}+\left\|  \mathscr{O}^{\tau,N} \right\|_{l^6_N}^9  \big) \Big),
			\end{aligned}
		\end{equation*}  
		where  $\mathbb{I}_1$ and $\mathbb{I}_2$ are given in Section 3.1 and 
		\[
		\begin{aligned}
			K_4(t)
			:=&\int_0^t e^{-\frac{\lambda_1^2}{2}(t-s)}\Big(
			1+\|\mathscr{O}^{\tau,N}_s\|_{l_N^\infty}^8
			+\|(-A_N)^{\frac12}\mathscr{O}^{\tau,N}_s\|_{l_N^4}^4
			+\mathscr{R}_s^8 \\
			&\qquad 
			+\tau^{-1/3}\|\mathscr{O}^{\tau,N}_s-\mathscr{O}^{\tau,N}_{\lfloor s\rfloor}\|_{l_N^2}^2
			+\|\mathscr{O}^{\tau,N}_s-\mathscr{O}^{\tau,N}_{\lfloor s\rfloor}\|_{l_N^2}^4
			\Big)\,ds .
		\end{aligned}
		\] 
			Then by Lemma \ref{lem: O^{M,N}_continuous},   we obtain
		$$   \sup _{\tau \in (0,1), N \in \mathbb{N}} \sup _{i\geq 0 } \mathbb{E}\left[\left\|V_{t_i}^{\tau, N}\right\|_{w^{1,2}_N}^p\right]<\infty.$$  
		Thus, it follows that
		\begin{equation*} 
			\begin{aligned}
				&\sup _{\tau \in (0,1), N \in \mathbb{N}} \sup _{i\geq 0 } \mathbb{E}\left[\left\|u^{\tau, N}(t_m,\cdot)\right\|_{l^\infty_N}^p\right]\\
					  \leq& C \sup _{\tau \in (0,1), N \in \mathbb{N}} \sup _{i\geq 0 } \mathbb{E}\left[\left\|U_{t_i}^{\tau, N}\right\|_{w^{1,2}_N}^p\right]\\
			  \leq& C \sup _{\tau \in (0,1), N \in \mathbb{N}} \sup _{i\geq 0 } \mathbb{E}\left[\left\|V_{t_i}^{\tau, N}\right\|_{w^{1,2}_N}^p\right]+    C \sup _{\tau \in (0,1), N \in \mathbb{N}} \sup _{i\geq 0 } \mathbb{E}\left[\left\|\mathscr{O}_{t_i}^{\tau, N}\right\|_{w^{1,2}_N}^p\right] <\infty ,      
			\end{aligned}
		\end{equation*} 
		which completes the proof.
	\end{proof}

	\subsection{Uniform strong convergence of the numerical scheme}
	To establish the uniform strong convergence of the numerical scheme, we require an assumption stronger than Assumption \ref{assu:2}.
	\begin{assumption}\label{assu:2.2}
		Suppose that the initial condition   $u_0\in \mathcal{C}^5(\mathcal{O}).$  
	\end{assumption}
	The main result of this subsection is given in the following theorem.
	\begin{theorem}\label{thm: ||u^{M,N}-u||}  
		Suppose that Assumptions  \ref{assu:1}, \ref{assu:3 noise}, \ref{assu:2.2} hold and $h^{-1}\tau^9\leq1$.
		Then for $p \geq 1$, there exists a constant $C>0$, independent of $\tau$ and $h$,  such that  
		\begin{equation}\label{ineq: ||u^{M,N}-u||}
			\sup_{t \geq 0}\left\|u(t,\cdot)-u^{\tau,N}(t, \cdot)\right\|_{L^p(\Omega;L^\infty)}   \leq C\left(h^{1-\epsilon}+ \tau^{\frac{3}{8}-\frac{\epsilon}{2}} \right),
		\end{equation} 
		where $\epsilon$ is an arbitrarily small positive number.  
	\end{theorem}
	
	To proof Theorem \ref{thm: ||u^{M,N}-u||} , we introduce the auxiliary processes
	\begin{equation} \label{eq: solution_auxiliary_full_discrete}
		\begin{aligned} 
			U^{N}_t= & e^{-A_N^2 t}U_0+\int_0^t A_N e^{-A_N^2(t-s)} F_N(	U^{N}_{s}) \mathrm{d} s+\mathscr{O}^N_t ,
		\end{aligned}
	\end{equation}  
	where  
	$ \mathscr{O}^{N}_t:= \sigma  \sum_{j=1}^{N-1}\int_0^t e^{-A_N^2(t-s)} \phi_{N, j} \mathrm{d} \beta_{j}(s). $
	Through the interpolation of $	U^{N}_t$, we can obtain
	\begin{equation} \label{eq:eq:discrete_CH_FractionalNoise_SDE_spatial} 
		\begin{aligned}
			u^N(t, x):=P_N U^N_t(x)= & \int_{\mathcal{O}} G_t^N(x, y) u_0\left(\kappa_N(y)\right) \mathrm{d} y+\int_0^t \int_{\mathcal{O}} \Delta_N G_{t-s}^N(x, y) f(u^N\left(s, \kappa_N(y)\right)) \mathrm{d} y \mathrm{d} s \\
			&+\sigma\sum_{j=1}^{N-1} \int_0^t \int_{\mathcal{O}} G_{t-s}^N(x, y)\phi_j\left(\kappa_N(y)\right) \mathrm{d} y \mathrm{d}  \beta_j(s)  .
		\end{aligned}
	\end{equation} 
	Then  we separate the error term $\left\|u(t, \cdot)-u^{\tau, N}(t, \cdot)\right\|_{L^p\left(\Omega; L^{\infty}\right)}$ as follows
	\begin{equation}  
		\begin{aligned}
			\left\|u(t, \cdot)-u^{\tau, N}(t, \cdot)\right\|_{L^p\left(\Omega; L^{\infty}\right)}
			\leq & \left\|u(t, \cdot)-u^{N}(t, \cdot)\right\|_{L^p\left(\Omega; L^{\infty}\right)}+    \left\|u^{N}(t, \cdot)-u^{\tau, N}(t, \cdot)\right\|_{L^p\left(\Omega; L^{\infty}\right)} 	.
		\end{aligned}
	\end{equation}  
	Following the methods used to prove Theorems 3.1 and 4.2 in \cite{deng2025}, we can replicate those steps with only minor modifications to establish the next two propositions. Then Theorem \ref{thm: ||u^{M,N}-u||}  is a straightforward conclusion of two auxiliary results stated below.
	\begin{proposition}\label{thm: ||u^{N}-u||}  
		Suppose that Assumptions  \ref{assu:1}, \ref{assu:3 noise}, \ref{assu:2.2} hold.
		Then for $p \geq 1$, there exists a constant $C>0$, independent of $N$,  such that  
		\begin{equation}\label{ineq: ||u^{M,N}-u||}
			\sup_{t \geq 0}\left\|u(t,\cdot)-u^{ N}(t, \cdot)\right\|_{L^p(\Omega;L^\infty)}   \leq C h^{1-\epsilon} ,
		\end{equation} 
		where $\epsilon$ is an arbitrarily small positive number.
	\end{proposition}   
	The proof is provided in Appendix D.
	
	\begin{proposition}\label{thm: ||u^{N}-u^{tau,N}||}  
		Suppose that Assumptions  \ref{assu:1}, \ref{assu:3 noise}, \ref{assu:2.2} hold and $h^{-1}\tau^9\leq1$.
		Then for $p \geq 1$, there exists a constant $C>0$, independent of $\tau$ and $h$,  such that  
		\begin{equation}\label{ineq: ||u^{M,N}-u||}
			\sup_{t \geq 0}\left\|u^N(t,\cdot)-u^{\tau,N}(t, \cdot)\right\|_{L^p(\Omega;L^\infty)}   \leq C\left(h^{1-\epsilon}+ \tau^{\frac{3}{8}-\frac{\epsilon}{2}} \right),
		\end{equation} 
		where $\epsilon$ is an arbitrarily small positive number.
	\end{proposition} 
	The proof of this proposition is similar to Theorem 5.7 in \cite{deng2025}. Therefore, these detailed steps are omitted.

	\section{Approximation of invariant measures} \label{section5}
	\subsection{Existence and uniqueness of invariant measure for the numerical solution}
	Since $u^{\tau,N}(t_m,\cdot)$ is defined as the interpolation of $	U^{\tau,N}_{t_m}$, the existence and uniqueness of the invariant measure for $\left\{u^{\tau,N}(t_m,\cdot)\right\}_{m \in \mathbb{N}}$ on $V^N$ is equivalent to that  for $\left\{U^{\tau,N}_{t_m}\right\}_{m \in \mathbb{N}}$ on $\mathbb{R}^N$, where $V^N=\text{Span}\left\{P_N\phi_{N,0},P_N\phi_{N,1},\cdots,P_N\phi_{N,N-1}\right\}$. We develop an new ergodicity framework, building on the work of \cite{mattingly2002}, to prove the existence and uniqueness of the invariant measure for $\left\{U^{\tau,N}_{t_m}\right\}_{m \in \mathbb{N}}$. This guarantees the corresponding result for $\left\{u^{\tau,N}(t_m,\cdot)\right\}_{m \in \mathbb{N}}$ .
	We are aware of several notable works, such as \cite{chen2017approximation, chen2020approximation}, which applied this theory to establish the existence and uniqueness of invariant measures for numerical approximations to SPDEs. These studies primarily focused on the numerical invariant measures for SPDEs under Dirichlet boundary conditions.
	We aim to generalize the application of this theory to the study of numerical invariant measures for SPDEs subject to Neumann boundary conditions. The key point is to establish  a minorization condition restricted to the hyperplane  $D_\alpha :=
	\left\{\vec{u}\in  \mathbb{R}^N:\ \ \frac{1}{\pi} \left\langle\vec{u}, \phi_{N, 0}  \right\rangle_{l^2_N}=  \alpha    \right\}$.  Hence we introduce the 
	minorization condition restricted to $D_\alpha$.
	
	\begin{assumption} \label{ass: 5} (Minorization condition restricted to $D_\alpha$)
		The Markov chain $\left\{\vec{v}_k\right\}_{k \in \mathbb{N}}$ on a state space $\left(\mathbb{R}^N, \mathscr{B}\big(\mathbb{R}^N\big)\right)$ with transition kernel $P_k(\vec{v}, B):=\mathbb{P}\left\{  \vec{v}_k \in B \mid \vec{v}_0=\vec{v}\right\}, k \in \mathbb{N}, \vec{v} \in \mathbb{R}^N, B \in \mathscr{B}\left(\mathbb{R}^N\right)$ satisfies the following conditions:  \\
		\textbf{(i)}   $ \vec{v}_k\in D_\alpha,\ \forall k \in \mathbb{N}$ .\\
		Moreover, for a fixed compact set $S \in \mathcal{B}\big(\mathbb{R}^N\big)$, we have:\\
		\textbf{(ii)} for some $\vec{w}^* \in \operatorname{Int}(S)\cap D_\alpha $ and any $\delta>0$, there exists  $\bar{k}=\bar{k}(\delta) \in \mathbb{N}$ such that
		
		$$
		P_{\bar{k}}\left(\vec{w}, B_\delta\left(\vec{w}^*\right)\cap D_\alpha\right)>0, \quad \forall\ \vec{w} \in S\cap D_\alpha ,
		$$
		
		where $B_\delta\left(\vec{w}^*\right)\subset \mathbb{R}^N$ denotes the open ball in   of radius $\delta$ centered at $\vec{w}^*$.\\
		\textbf{(iii)} for $k \in \mathbb{N}$ the transition kernel $P_k(\vec{v}, B )$ possesses a density $p_k(\vec{v},\vec{w})$ such that
		
		$$
		P_k(\vec{v}, B)=\int_B p_k(\vec{v}, \vec{w})   \delta_{D_\alpha}(\vec{w})\  \mathrm{d} \vec{w}, \quad \forall\ \vec{v} \in S\cap D_\alpha    , B \in \mathscr{B}\big(\mathbb{R}^N\big) \cap \mathscr{B}(S\cap D_\alpha ), 
		$$
		and $p_k(\vec{v}, \vec{w})$ is jointly continuous in $(\vec{v}, \vec{w}) \in \left(S\cap D_\alpha\right)  \times\left( S\cap D_\alpha \right)$. Here $\delta_{D_\alpha}(\cdot)$ denotes the Dirac function concentrated on the set $D_\alpha$.
	\end{assumption}
	
	\begin{lemma} \label{lem: geometric ergodicity for Markov chains 1}
		Suppose Assumption \ref{ass: 5} holds.  Then there exist $k^* \in \mathbb{N}$, $\eta>0$, and a probability measure $\nu$ satisfying 
		$\nu\left(\left(S\cap D_\alpha\right)^c\right)=0$ and $\nu(S\cap D_\alpha)=1$, such that 
		$$
		P_{k^*}(\vec{v}, B\cap D_\alpha) \geq \eta \nu(B\cap D_\alpha), \quad \forall B \in \mathscr{B}\big(\mathbb{R}^N\big),\ \vec{v} \in S\cap D_\alpha .
		$$
	\end{lemma} 
	\begin{proof}
		Based on Assumption \ref{ass: 5} (ii), it can be readily shown that for any $\delta^{\prime}>0$ and $\vec{w}^*\in \operatorname{Int}(S)\cap D_\alpha$, there exists $\bar{k}_2={\bar{k}_2}(\delta^{\prime}) \in \mathbb{N}$ such that $P_{\bar{k}_2}\left(\vec{w}^*, B_{\delta^{\prime}}\left(\vec{w}^*\right)\cap D_\alpha\right)>0$.  Furthermore, the existence of the density function  $p_{\bar{k}_2}(\cdot, \cdot)$  from Assumptions \ref{ass: 5} (iii) implies that there exists $\vec{x}^* \in B_{\delta^{\prime}}\left(\vec{w}^*\right) \cap D_\alpha$  and  $\gamma_1>0$ such that  
		$$
		p_{\bar{k}_2}\left(\vec{w}^*, \vec{x}^*\right) \geq 2 \gamma_1>0.
		$$
		Owing to the joint continuity of the density stated in Assumption \ref{ass: 5} (iii),  there exist $\epsilon_1^*,\ \epsilon_2^*>0$ such that
		$$ B_{\epsilon_2^*}\left(\vec{x}^*\right) \subset S \quad \text{and} \quad     	p_{\bar{k}_2}\left(\vec{w}, \vec{x}\right) \geq  \gamma_1 ,\quad \forall \vec{w}\in \mathscr{B}_{\epsilon_1^*}\left(\vec{w}^*\right) \cap D_\alpha,  \vec{x}\in B_{\epsilon_2^*}\left(\vec{w}^*\right) \cap D_\alpha, $$  
		which implies
		$$
		P_{\bar{k}_2}(\vec{w}, B \cap D_\alpha)=\int_{B \cap D_\alpha} p_{\bar{k}_2}(\vec{w}, \vec{x})  \delta_{D_\alpha}(\vec{x}) \mathrm{d} \vec{x} 
		\geq \int_{B   \cap B_{\epsilon_2}\left(\vec{w}^*\right)\cap D_\alpha} p_{\bar{k}_2}(\vec{w}, \vec{x})  \delta_{D_\alpha}(\vec{x}) \mathrm{d} \vec{x} 
		\geq \epsilon_1 \lambda\left(B  \cap B_{\epsilon_2}\left(\vec{w}^*\right)\cap D_\alpha\right),
		$$ 
		for all $\vec{w} \in B_{\epsilon_1^*}\left(\vec{w}^*\right) \cap D_\alpha$. Here $\lambda(\cdot)$ is 
		the  $(N-1)$-dimensional Hausdorff measure on the hyperplane $D_\alpha$.
		
		According to Assumption \ref{ass: 5}  (ii), there exists a $ \bar{k}_1>0$ such that 
		$$P_{\bar{k}_1}\left(\vec{v}, B_{\epsilon_1^*}\left(\vec{w}^*\right)\cap D_\alpha\right)>0, \quad \forall \vec{v} \in S\cap D_\alpha. $$
		The continuity of $p_{\bar{k}_1}(\cdot, \vec{w})$ on $ \left(S\cap D_\alpha\right)  \times\left( S\cap D_\alpha \right)$ , as stated in Assumption \ref{ass: 5} (iii), can be extended to $P_{\bar{k}_1}(\cdot, B)$  via the dominated convergence theorem. Therefore,
		$$	\inf _{\vec{v}\in S\cap D_\alpha} P_{\bar{k}_1}\left(\vec{v}, B_{\epsilon_1^*}\left(\vec{w}^*\right)\cap D_\alpha\right) \geq \gamma_2, $$ 
		for some $\gamma_2>0$. Now let $ k^* =\bar{k}_1+\bar{k}_2$. Then, for all $\vec{v} \in S \cap D_\alpha$,
		$$ 
		\begin{aligned}
			P_{k^*}(\vec{v}, B\cap D_\alpha)  & \geq \int_{B_{\epsilon_1}\left(\vec{w}^*\right)\cap D_\alpha} p_{\bar{k}_1}(\vec{v}, \vec{w})   P_{\bar{k}_2}(\vec{w}, B\cap D_\alpha) \mathrm{d} \vec{w}\\
			& \geq \gamma_1 \lambda\left(B \cap B_{\epsilon_2}\left(\vec{w}^*\right)\cap D_\alpha\right) \int_{B_{\epsilon_1}\left(\vec{w}^*\right)\cap D_\alpha} p_{\bar{k}_1}(\vec{v}, \vec{w})  \delta_{D_\alpha}(\vec{w})   \mathrm{d} \vec{w}\\
			& =\gamma_1 \lambda\left(B \cap B_{\epsilon_2}\left(\vec{w}^*\right)\cap D_\alpha\right) P_{\bar{k}_1}\left(\vec{v}, B_{\epsilon_1^*}\left(\vec{w}^*\right)\cap D_\alpha\right)\\
			& \geq \gamma_1\gamma_2 \lambda\left( B_{\epsilon_2}\left(\vec{w}^*\right)\cap D_\alpha\right)\nu(B),
		\end{aligned}
		$$ 
		where $\nu(\cdot):=\lambda\left(\cdot \cap B_{\epsilon_2}\left(\vec{w}^*\right)\cap D_\alpha\right)/\lambda\left( B_{\epsilon_2}\left(\vec{w}^*\right)\cap D_\alpha\right).$ Letting  $\eta=\gamma_1 \gamma_2 \lambda\left(B_{\epsilon_2}\left(\vec{w}^*\right)\cap D_\alpha\right)$, we obtain $ 	P_{k^*}(\vec{v}, B\cap D_\alpha) \geq \eta \nu(B\cap D_\alpha) $ for all $B \in \mathscr{B}\big(\mathbb{R}^N\big)$ and $\vec{v}\in S\cap D_\alpha$. Since $B_{\epsilon_2}\left(\vec{w}^*\right) \subset S$,  it follows that  $\nu(S\cap D_\alpha)=1$ and   $\nu\left(\left(S\cap D_\alpha\right)^c\right)=0$. This completes the proof.
	\end{proof}
	
	Next we introduce the Lyapunov condition.
	\begin{assumption} \label{ass: 4} (Lyapunov condition)
		There is a function $V: \mathbb{R}^N \rightarrow[1, \infty)$ with $\lim _{|\vec{v}| \rightarrow \infty} V(\vec{v})=\infty$ and real numbers $\alpha_1 \in$ $(0,1), \alpha_2 \in[0, \infty)$ such that 
		$$
		\mathbb{E}\left[V\left(\vec{v}_{k+1}\right) \mid \mathscr{F}_{t_k}\right] \leq \alpha_1 V\left(\vec{v}_k\right)+\alpha_2,
		$$ 
		where $\{ \mathscr{F}_t\}_{t\in[0,\infty)}$ is a filtration with respect to the probability space $(\Omega, \mathscr{F}, \mathbb{P})$.
	\end{assumption}

	\begin{lemma} \cite[Lemma A.6]{mattingly2002} \label{lem: geometric ergodicity for Markov chains 2}
		Suppose $\{\vec{v}_k\}_{k\in \mathbb{N}}$ and $\{\vec{w}_k\}_{k\in \mathbb{N}}$ are two Markov chains satisfying Assumption \ref{ass: 4}. Let $\gamma \in (\sqrt{\alpha_1},1)$ and define
		$$ S:=\left\{\vec{v}:\  V(\vec{v}) \leq \frac{2 \alpha_2}{\gamma-\alpha_1}\right\} \quad \text{and} \quad  \tilde{\zeta}=\inf _{k \geq 0}\left\{\left(\vec{v}^\prime_k, \vec{w}^\prime_k\right) \in  S \times S, \hat{\psi}_k=0\right\} , $$   
		where $\{\hat{\psi}_k\}_{k\in \mathbb{N}}$ is a sequence of independent and identically distributed random variables satisfying $\mathbb{P}(\hat{\psi}_k=1)=1-\eta$  and  $ \mathbb{P}(\hat{\psi}_k=0)=\eta$.
		Then there exists an $r \in(0,1)$ such that 
		$$
		\max \left\{\mathbb{E} [V\left(\vec{v}_k^{\prime}\right)] 1_{n<\tilde{\zeta}},\  \mathbb{E}[ V\left(\vec{w}_k^{\prime}\right) 1_{n<\tilde{\zeta}}]\right\} \leq C \left[\mathbb{E} [V\left(\vec{v}_0\right)+V\left(\vec{w}_0\right)] +1\right] r^n .
		$$ 
	\end{lemma} 
	
	By Lemmas \ref{lem: geometric ergodicity for Markov chains 1} and \ref{lem: geometric ergodicity for Markov chains 2}, we can establish the following theorem.
	\begin{theorem} \label{th: geometric ergodicity for Markov chains}
		If Markov chain $\left\{\vec{v}_k\right\}_{k \in \mathbb{N}}$ satisfies  Assumption \ref{ass: 4} with  the function $V$ and Assumptions \ref{ass: 5} with
		$$ S=\left\{\vec{v}:\  V(\vec{v}) \leq \frac{2 \alpha_2}{\gamma-\alpha_1}\right\}   $$ 
		for some $\gamma \in\left(\sqrt{\alpha_1}, 1\right)$.  Then $\left\{\vec{v}_k\right\}_{k \in \mathbb{N}}$ is ergodic with a unique invariant measure. 
	\end{theorem} 
	\begin{proof}
		\textbf{Step 1: }We first establish the existence of invariant measure for the Markov chain   $\left\{\vec{v}_k\right\}_{k \in \mathbb{N}}$. \par 
		By Assumption \ref{ass: 4}, we have that for any $k \in \mathbb{N}$,
		\begin{equation*}  
			\mathbb{E}\left[ V(\vec{v}_k)  \right] =	\mathbb{E}\left[ 	\mathbb{E}\left[ V(\vec{v}_k)\mid \mathscr{F}_{k-1} \right] \right]
			\leq \alpha_1 \mathbb{E}\left[ V(\vec{v}_{k-1})  \right]+\alpha_2\leq \alpha_1^k \mathbb{E}\left[ V(\vec{v}_{0})  \right]
			+\alpha_2\sum_{i=1}^{k}\alpha_1^{i-1}\leq   \alpha_1^k \mathbb{E}\left[ V(\vec{v}_{0})  \right]
			+\frac{\alpha_2}{1-\alpha_1} .
		\end{equation*}
		Since $\alpha_1 \in (0,1)$, the above estimate implies that
		\begin{equation}  \label{es: V(x)}
			\sup_{k \geq 0} \mathbb{E}\left[V\left(\vec{v}_k\right)\right]<\infty.
		\end{equation}
		We define a sequence of probability measures  by
		\begin{equation*}  
			\pi^{\tau,N}_m(B) := \frac{1}{m+1} \sum_{k=0}^m \mathbb{P}\left\{\vec{v}_k \in B\right\}, \quad   B\in  \mathscr{B}\big( \mathbb{R}^N \big),\ m \in \mathbb{N}.
		\end{equation*} 
		It follows from \eqref{es: V(x)} that the sequence $\left\{\pi^{\tau,N}_m\right\}_{m \in \mathbb{N}}$ is tight.  Consequently, by the Krylov–Bogoliubov Theorem \cite[Theorem 1.2]{Hong2019}, the Markov chain $\left\{\vec{v}_k\right\}_{k \in \mathbb{N}}$  admits an invariant measure $\pi^{\tau,N}$ .\par  
		
		\textbf{Step 2: }   
		Before establishing the uniqueness of an invariant measure, we construct an auxiliary Markov chain $ \left\{\vec{w}_k \right\}_{k \in \mathbb{N}} $, which is  equal in distribution to $\left\{\vec{v}_{k  k^*}\right\}_{k \in \mathbb{N}}$, with $k^*$  as defined in Lemma \ref{lem: geometric ergodicity for Markov chains 1}. \par

		We now introduce a new transition kernel
		$$
		\tilde{P}_{k^*}(\vec{v}, B)= \begin{cases}P_{k^*}(\vec{v}, B), & \forall \vec{v} \in S^c \cap D_\alpha, \\
			\frac{1}{1-\eta}[P_{k^*}(\vec{v},B)-\eta v(B)], & \forall \vec{v} \in S\cap D_\alpha,\end{cases}
		$$
		where $\eta$ and $\nu$ are defined in Lemma \ref{lem: geometric ergodicity for Markov chains 1}. It can be verified that this kernel is well-defined, as ensured by  Lemma \ref{lem: geometric ergodicity for Markov chains 1} and Assumption \ref{ass: 5}.
		
		The new Markov chain $\left\{\vec{w}_k \right\}_{k \in \mathbb{N}}$ can be generated by Algorithm 1.
	\begin{algorithm}
		\caption{Construction of the new Markov chain $\{\vec w_k\}$}
		\begin{algorithmic}[1]
			\State \textbf{Input:} Initial value $\vec w_0=\vec v_0\in D_\alpha$
			\State \textbf{Output:} Markov chain $\{\vec w_k\}$
			\For{$k = 0,1,2,\dots$}
			\State Sample $\hat{\psi}_k$
			\If{$\vec w_k \in S$}
			\State $\vec w_{k+1} \gets \hat{\psi}_k\,\vec w_k^* + (1-\hat{\psi}_k)\vec\xi_k$
			\Else
			\State $\vec w_{k+1} \gets \vec w_k^*$
			\EndIf
			\EndFor
		\end{algorithmic}
	\end{algorithm}

		It is straightforward to verify that the transition law of $\{\vec{w}_k \}$ coincides with $P_{k^*}$.
		
		%
		%
		%
		%

		\textbf{Step 3: }  In this step, we now establish the uniqueness of the invariant measure for  $\left\{\vec{v}_k\right\}_{k \in \mathbb{N}}$.\par 
		Assume that $\hat{\pi}^{\tau,N}$ is another invariant measure. We will show that $ {\pi}^{\tau,N}=\hat{\pi}^{\tau,N}$. 
		Using the construction in Step 2, we generate two chains  $\left\{\vec{v}^\prime_k \right\}_{k \in \mathbb{N}}$ and $\left\{\vec{w}^\prime_k \right\}_{k \in \mathbb{N}}$ ,  each of which is  equal in distribution to $\left\{\vec{v}_{k  k^*}\right\}_{k \in \mathbb{N}}$, with initial distributions $\vec{v}_{0}^\prime=\pi^{\tau,N}$ and   $\vec{w}^\prime_{0}=\hat{\pi}^{\tau,N}$, respectively. Obviously, $\vec{v}^\prime_k,\ \vec{w}^\prime_k\in D_\alpha$.
		Let $f:\ \mathbb{R}^N\to \mathbb{R}$ be a test function satisfying $|f(x)|\leq V(x)$ for all $x\in\mathbb{R}^N$, where $V$ is the Lyapunov function.
		We decompose $f$ into its positive and negative parts, denoted by $f^+$ and $f^-$. Then,
		$$
		\left|\mathbb{E} [f\left(\vec{v}^\prime_k\right) ]-\mathbb{E} [ f\left(\vec{w}^\prime_k\right) ]\right| \leq\left|\mathbb{E} [f^{+}\left(\vec{v}^\prime_k\right)]-\mathbb{E} [f^{+}\left(\vec{w}^\prime_k\right)]\right|+\left|\mathbb{E} [f^{-}\left(\vec{v}^\prime_k\right)]-\mathbb{E} [f^{-}\left(\vec{w}^\prime_k\right)]\right| .
		$$ 
		Define the coupling time
		$$
		\zeta=\inf _{k \geq 0}\left\{\left(\vec{v}^\prime_k, \vec{w}^\prime_k\right) \in  \left(S\cap D_\alpha\right) \times \left(S\cap D_\alpha\right), \hat{\psi}_k=0\right\}
		$$ 
		It follows that
		$$
		\mathbb{E}[ f^{+}\left(\vec{v}^\prime_k\right)]=\mathbb{E} [f^{+}\left(\vec{v}^\prime_k\right) \chi_{k \geq \zeta}]+\mathbb{E} [f^{+}\left(\vec{v}^\prime_k\right) \chi_{k<\zeta }].
		$$
		Since $\mathbb{E}[ f^{+} (\vec{v}^\prime_k ) \chi_{k \geq \zeta}]=\mathbb{E}[ f^{+} (\vec{w}^\prime_k ) \chi_{k \geq \zeta}] \leq \mathbb{E} [f^{+} (\vec{w}^\prime_k )]$ and $f^{+} \leq V$, we obtain 
		$$
		\mathbb{E}[ f^{+}\left(\vec{v}^\prime_k\right) ]\leq \mathbb{E} [f^{+}\left(\vec{w}^\prime_k\right)]+\mathbb{E}[ V\left(\vec{v}^\prime_k\right) \chi_{k<\zeta}] .
		$$ 
		Since $\vec{v}^\prime_k$ and $\vec{w}^\prime_k$ play symmetric roles, we obtain
		$$
		\left|\mathbb{E} [f^{ +}\left(\vec{v}^\prime_k\right)]-\mathbb{E} [f^{ +}\left(\vec{w}^\prime_k\right)]\right| \leq \max \left\{\mathbb{E}[ V\left(\vec{v}^\prime_k\right) \chi_{k<\zeta}], \mathbb{E} [V\left(\vec{w}^\prime_k\right) \chi_{k<\zeta}]\right\}.
		$$ 
		Similarly 
		$$
		\left|\mathbb{E} [f^{-}\left(\vec{v}^\prime_k\right)]-\mathbb{E} [f^{-}\left(\vec{w}^\prime_k\right)]\right| \leq \max \left\{\mathbb{E} [V\left(\vec{v}^\prime_k\right) \chi_{k<\zeta}], \mathbb{E}[ V\left(\vec{w}^\prime_k\right) \chi_{k<\zeta}]\right\}. 
		$$ 
		Hence,
		$$
		\left|\mathbb{E} [f\left(\vec{v}^\prime_k\right)]-\mathbb{E} [f\left(\vec{w}^\prime_k\right)]\right| \leq 2 \max \left\{\mathbb{E} [V\left(\vec{v}^\prime_k\right) \chi_{k<\zeta}], \mathbb{E} [V\left(\vec{w}^\prime_k\right) \chi_{k<\zeta}]\right\} .
		$$
		
		Let $\tilde{\zeta}$ be another coupling time defined by 
		$$ \tilde{\zeta}:=\inf _{k \geq 0}\left\{\left(\vec{v}^\prime_k, \vec{w}^\prime_k\right) \in  S \times S, \hat{\psi}_k=0\right\}.$$   
		It follows from Assumption \ref{ass: 5} (i) that 
		$$ \max \left\{\mathbb{E} [V\left(\vec{v}^\prime_k\right) \chi_{k<\zeta}], \mathbb{E} [V\left(\vec{w}^\prime_k\right) \chi_{k<\zeta}]\right\}=\max \left\{\mathbb{E} [V\left(\vec{v}^\prime_k\right) \chi_{k<\tilde{\zeta}}], \mathbb{E}[ V\left(\vec{w}^\prime_k\right) \chi_{k<\tilde{\zeta}}]\right\}.    $$
		Then, by Lemma \ref{lem: geometric ergodicity for Markov chains 2}, we obtain
		$$
		\begin{aligned}
			&	\left|\int_{\mathbb{R}^N\cap D_\alpha} f(\vec{v}) \pi^{\tau,N}(\mathrm{d}\vec{v} )- \int_{\mathbb{R}^N\cap D_\alpha} f(\vec{v}) \hat{\pi}^{\tau,N}(\mathrm{d}\vec{v} )   \right| \\
			= & \left|\mathbb{E}[ f \left(\vec{v}^\prime_k\right)]-\mathbb{E} [f\left(\vec{w}^\prime_k\right)]\right| \\
			\leq&  C\left[ 
			\int_{\mathbb{R}^N\cap D_\alpha} V(\vec{v}) \pi^{\tau,N}(\mathrm{d}\vec{v} )
			+\int_{\mathbb{R}^N\cap D_\alpha} V(\vec{w}) \hat{\pi}^{\tau,N}(\mathrm{d}\vec{w} )+1\right] r^k  \rightarrow 0 \text { as } k \rightarrow \infty.
		\end{aligned}
		$$ 
		This shows that $\pi^{\tau,N} = \hat{\pi}^{\tau,N}$, and hence uniqueness is proved.
	\end{proof}

	\begin{theorem} \label{th: invariant measure for the SCHE discrete}
		Suppose that Assumptions \ref{assu:1}-\ref{assu:3 noise} hold, then $\left\{ u^{\tau,N}(t_m,\cdot) \right\}_{m\in \mathbb{N}}$ admits a unique invariant measure $\pi^{\tau,N}$. 
	\end{theorem}
	
	\begin{proof}
		Since  $u^{\tau,N}(t,x)=P_NU^{\tau,N}_{t}$, it is sufficient to establish  the existence and uniqueness of an invariant measure for  $\left\{ U^{\tau,N}_{t_m} \right\}_{m\in \mathbb{N}}$. From \eqref{eq: mild_solution_full_discrete discrete},  
		it follows that  $\left\{ U^{\tau,N}_{t_m} \right\}_{m\in \mathbb{N}}$ forms a time-homogeneous Markov chain. According to Theorem \ref{th: geometric ergodicity for Markov chains},  establishing the existence and uniqueness of an invariant measure for  $\left\{ U^{\tau,N}_{t_m} \right\}_{m\in \mathbb{N}}$ necessitates verifying that this Markov chain  satisfies both the Lyapunov and minorization conditions restricted to $D_\alpha$.
		
		We first establish the Lyapunov condition. Define the Lyapunov function by
		$$V(\vec{v}):= \sum_{j=1}^{N-1}\lambda_{N, j}^{-1} \langle\vec{v},\phi_{N,j}\rangle^2_{l^{2}_N} + \langle\vec{v},\phi_{N,0}\rangle^2_{l^{2}_N}+1,$$ 
		for any $\vec{v}\in l^2_N$.  By applying \eqref{eq: mild_solution_full_discrete discrete} and using properties of conditional expectation, we obtain
		\begin{equation} \label{ineq: Lyapunov condition 1}
			\begin{aligned}
				\mathbb{E}\left[V\left(	U^{\tau,N}_{t_{m+1}}\right) \mid \mathscr{F}_{t_m}\right] 
				=&\sum_{j=1}^{N-1}\lambda_{N, j}^{-1}\mathbb{E}\left[ \langle	e^{-A_N^2\tau} U^{\tau,N}_{t_{m}}+\tau A_N e^{-A_N^2\tau}   \tilde{F}_N(		U^{\tau,N}_{t_{m}}) +\sigma e^{-A_N^2\tau} \Delta \beta_{m} ,\phi_{N,j}\rangle^2_{l^{2}_N}\mid \mathscr{F}_{t_m}\right]  \\
				& +\mathbb{E}\left[\langle e^{-A_N^2\tau} 	U^{\tau,N}_{t_{m}}+\tau A_N e^{-A_N^2\tau}  \tilde{F}_N(		U^{\tau,N}_{t_{m}}) +\sigma e^{-A_N^2\tau} \Delta \beta_{m} ,\phi_{N,0}\rangle^2_{l^{2}_N}\mid \mathscr{F}_{t_m}\right] +1\\
				=&\sum_{j=1}^{N-1}e^{-2 \lambda_{N,j} \tau}\lambda_{N, j}^{-1}  \langle	U^{\tau,N}_{t_{m}}+\tau A_N  \tilde{F}_N(		U^{\tau,N}_{t_{m}})   ,\phi_{N,j}\rangle^2_{l^{2}_N} +\sum_{j=1}^{N-1}\lambda_{N, j}^{-1}e^{-2 \lambda_{N,j} \tau}\mathbb{E}\left[ \langle \sigma   \Delta \beta_{m} ,\phi_{N,j}\rangle^2_{l^{2}_N}   \right]  \\
				&+ \langle	U^{\tau,N}_{t_{m}}+\tau A_N  \tilde{F}_N(		U^{\tau,N}_{t_{m}})  ,\phi_{N,0}\rangle^2_{l^{2}_N} + 
				\mathbb{E}\left[\langle\sigma   \Delta \beta_{m} ,\phi_{N,0}\rangle^2_{l^{2}_N} \right] +1.
			\end{aligned}
		\end{equation}  
		According to \eqref{ineq: uv to u v}, we have 
		$$  \left\| F_N(U^{\tau,N}_{t_{m}}) \right\|_{w^{1,2}_N}\leq C \left( 1+ \left\| U^{\tau,N}_{t_{m}}\right\|^3_{w^{1,2}_N}\right) , $$
		which, combined with   the H\"older inequality, yields
		\begin{equation*}  
			\begin{aligned}
				& \sum_{j=1}^{N-1}e^{-2 \lambda_{N,j} \tau}\lambda_{N, j}^{-1}  \langle	U^{\tau,N}_{t_{m}}+\tau A_N  \tilde{F}_N(		U^{\tau,N}_{t_{m}})  ,\phi_{N,j}\rangle^2_{l^{2}_N} \\
				\leq & e^{-2 \lambda_{N,1} \tau} \left( \sum_{j=1}^{N-1}\lambda_{N, j}^{-1} \langle	U^{\tau,N}_{t_{m}}  ,\phi_{N,j}\rangle^2_{l^{2}_N}+  \tau^2\left\| (-A_N)^{\frac{1}{2}}\tilde{F}_N(U^{\tau,N}_{t_{m}}) \right\|^2_{l^2_N}+2 \tau\left\| U^{\tau,N}_{t_{m}} \right\|_{l^2_N} \left\| \tilde{F}_N(U^{\tau,N}_{t_{m}}) \right\|_{l^2_N} \right)     \\
				\leq & e^{-2 \lambda_{N,1} \tau}\Bigg(\sum_{j=1}^{N-1}\lambda_{N, j}^{-1} \langle	U^{\tau,N}_{t_{m}}  ,\phi_{N,j}\rangle^2_{l^{2}_N}+ 
				\frac{C  \tau^2 \left(1+\left\| U^{\tau,N}_{t_{m}}\right\|^6_{w^{1,2}_N}\right) +C\tau\left\| U^{\tau,N}_{t_{m}}\right\|_{w^{1,2}_N}\Big(1+\left\| U^{\tau,N}_{t_{m}}\right\|^3_{w^{1,2}_N}\Big)}{1+\tau\left\| U^{\tau,N}_{t_{m}}\right\|^{12}_{w^{1,2}_N}}  \Bigg) \\
				\leq & e^{-2 \lambda_{N,1} \tau} \sum_{j=1}^{N-1}\lambda_{N, j}^{-1} \langle	U^{\tau,N}_{t_{m}}  ,\phi_{N,j}\rangle^2_{l^{2}_N}+C. \\
			\end{aligned}
		\end{equation*} 
		Observe that $  \lambda_{N, 0}=0\  \text{and}\ \left\langle U^{\tau,N}_{t_i} ,\phi_{N,0}\right\rangle_{l^{2}_N}=\left\langle u^N_0 ,\phi_{N,0}\right\rangle_{l^{2}_N}$ for all $ i\in \mathbb{N} .$ Hence, we have 
		$$   \langle	U^{\tau,N}_{t_{m}}+\tau A_N  \tilde{F}_N(		U^{\tau,N}_{t_{m}})  ,\phi_{N,0}\rangle^2_{l^{2}_N} = \left\langle u^N_0 ,\phi_{N,0}\right\rangle_{l^{2}_N}\leq C.  $$
		Applying  Itô's isometry and the orthonormality of $\left\{\phi_{N,j}\right\}_{j=0}^{N-1} $, we obtain
		\begin{equation*} 
			\begin{aligned}
				& \sum_{j=1}^{N-1}\lambda_{N, j}^{-1}e^{-2 \lambda_{N,j} \tau}\mathbb{E}\left[ \langle \sigma   \Delta \beta_m ,\phi_{N,j}\rangle^2_{l^{2}_N}   \right] 
				+ \mathbb{E}\left[\langle\sigma   \Delta \beta_m ,\phi_{N,0}\rangle^2_{l^{2}_N} \right] \\
				= &\sum_{j=1}^{N-1}\lambda_{N, j}^{-1}e^{-2 \lambda_{N,j} \tau}\sigma^2  \mathbb{E}\left[    | \Delta \beta_{j,m}|^2     \right] 
				+ \mathbb{E}\left[\langle\sigma   \Delta \beta_m ,\phi_{N,0}\rangle^2_{l^{2}_N} \right] 
				\leq C .
			\end{aligned} 
		\end{equation*} 
		Substituting these inequalities into \eqref{ineq: Lyapunov condition 1}, we conclude that
		$$  \mathbb{E}\left[V\left(	U^{\tau,N}_{t_{m+1}}\right) \mid \mathscr{F}_{t_m}\right]
		\leq  e^{-2 \lambda_{N,1} \tau}  V\left(	U^{\tau,N}_{t_{m}}\right)+ C ,$$
		which demonstrates that the Lyapunov condition is satisfied with the Lyapunov function $V$ . \par 
		
		Now we are ready to prove the minorization condition. Define the set
		$$ S:=\Big\{\vec{v}:\  V(\vec{v}) \leq \frac{2 \alpha_2}{\gamma-\alpha_1}\Big\}    $$ 		for some $\gamma \in\left(\sqrt{\alpha_1}, 1\right)$. 
		Since $\big\langle U^{\tau,N}_{t_i} ,\phi_{N,0}\big\rangle_{l^{2}_N}=\left\langle u^N_0 ,\phi_{N,0}\right\rangle_{l^{2}_N} $, we have 
		$   U^{\tau,N}_{t_m}\in D_\alpha ,\ \forall m \in \mathbb{N}$ . 
		
		Now, for any  $\vec{v} \in S \cap D_\alpha$, $\vec{v}^* \in \text{Int}(S )\cap D_\alpha $, $\delta>0$ and $\vec{w}  \in B_\delta\left(\vec{v}^* \right)\cap D_\alpha $, we apply \eqref{eq: mild_solution_full_discrete discrete} together with the fact that $\left\{\phi_{N,j}\right\}_{j=0}^{N-1}$ is an orthonormal basis of $l^2_N$ to obtain
		$$
		\left\langle  \Delta \beta_{m}, \phi_{N,j} \right\rangle_{l^2_N}
		= e^{\lambda_{N,j}^2 \tau}\left\langle \vec{w} ,\phi_{N,j} \right\rangle_{l^2_N}-\left\langle \vec{v} , \phi_{N,j} \right\rangle_{l^2_N}
		-\tau \lambda_{N,j}    \left\langle  \tilde{F}_N(\vec{v} ),\phi_{N,j} \right\rangle_{l^2_N}, \quad  j=0,1, \cdots, N-1.
		$$ 
		This demonstrates that one can choose $\Delta \beta_{m}$ such that  $U^{\tau,N}_{t_{m+1}}=\vec{w}$ starting from $U^{\tau,N}_{t_{m}}=\vec{v}$.
		Since  $\left\{ \left\langle  \Delta \beta_{m},\phi_{N,j} \right\rangle_{l^2_N}\right\}_{j=1}^{N-1}$ are Gaussian random variables, it follows that
		$	P_{1}\left(\vec{v}, B_\delta\left(\vec{v}^*\right)\cap D_\alpha\right)>0$, which verifies the second part of the minorization condition. To verify the third condition in the minorization condition, observe that each Gaussian variable  $ \left\langle  \Delta \beta_{m},\phi_{N,j} \right\rangle_{l^2_N}$ admits a $C^{\infty}$ density function for $j=1,\cdots,N-1$, and $ \left\langle  \Delta \beta_{m},\phi_{N,0} \right\rangle_{l^2_N}=0$. Hence, the one-step transition kernel $P_1(\vec{v}, \cdot)$ admits a density $p_1(\vec{v}, \vec{w})$, which is jointly continuous on $(\vec{v}, \vec{w}) \in\left(S\cap D_\alpha\right)  \times\left( S\cap D_\alpha \right)$. Finally, the time-homogeneity of the Markov chain $\left\{U^{\tau,N}_{t_{m}}\right\}_{m \in \mathbb{N}}$  implies that the density functions $\left\{p_m(\vec{v}, \vec{w}) \right\}_{m \in \mathbb{N}}$ are jointly continuous. This completes the proof.
	\end{proof}

	\subsection{Approximation of invariant measure}
	We are in a position to state the approximation result of the invariant measure.
	\begin{theorem} \label{th: error: invariant measure }
		Suppose that Assumptions  \ref{assu:1}, \ref{assu:3 noise}, \ref{assu:2.2} hold,  then for any bounded uniformly continuous function  $\phi:\ H \to \mathbb{R}$, there exists a constant $C>0$, independent of $h$ and $\tau$,  such that
		\begin{equation}\label{ineq: ||tilder U^{M,N} || 1,N}
			\left|\int_{H_\alpha} \phi(v) \pi(\mathrm{d} v)-   \int_{V^N} \phi(v) \tilde{\pi}^{\tau,N}( \mathrm{ d}v ) \right| \leq   C\left(h^{1-\epsilon}+ \tau^{\frac{3}{8}-\frac{\epsilon}{2}} \right),
		\end{equation}   
		where $\epsilon$ is an arbitrarily small positive number.
	\end{theorem}
	\begin{proof}
		By Theorem \ref{th: invariant measure for the SCHE discrete} and the definition of ergodicity, we obtain
		\begin{equation}\label{eq: ergodicity M N}
			\lim_{M \rightarrow \infty} \frac{1}{M} \sum_{m=0}^{M-1}\mathbb{E} \left[\phi (u^{\tau,N}( t_m,\cdot)  )\right]=  \int_{V^N} \phi(v)  \pi^{\tau,N}( \mathrm{ d} v ), \quad  \phi \in C_b(H),
		\end{equation}
		which, combined with Theorem \ref{thm: ||u^{M,N}-u||}, leads to
		\begin{equation*}
			\begin{aligned}
				\left|\int_{H_\alpha} \phi(v) \pi(\mathrm{d} v)-\int_{V^N} \phi(v) \pi^{\tau,N}( \mathrm{ d} v )\right|  
				\leq &  \lim _{M \rightarrow \infty} \frac{1}{M \tau} \sum_{m=0}^{M-1} \int_{t_m}^{t_{m+1}}\left|\mathbb{E} \left[\phi\left(u(t_m,\cdot)\right)\right]-\mathbb{E} \left[\phi (u^{\tau,N}( t_m,\cdot)  )\right]\right| \mathrm{d} t \\ 
				\leq &  C\left(h^{1-\epsilon}+ \tau^{\frac{3}{8}-\frac{\epsilon}{2}} \right).
			\end{aligned}
		\end{equation*} 
	\end{proof}
	\begin{remark} \label{re: 5.2 numerical ergodic limit}
		Based on Theorem \ref{th: error: invariant measure } and the definition of ergodicity, we can obtain the expectation of the time averages of the numerical solutions converges to the ergodic limit,  i.e.  
		\begin{equation*} 
			\lim _{T \rightarrow \infty} \lim _{\tau \rightarrow 0} \lim _{N \rightarrow \infty}	\frac{\tau}{T}\sum_{m=0}^{\lfloor T/\tau \rfloor}  \mathbb{E} \left[\phi (u^{\tau,N}( t_m,\cdot)  )\right]=\int_{H_\alpha} \phi(v) \pi(\mathrm{d} v).
		\end{equation*}    
	\end{remark}
	
	\subsection{Strong law of large numbers}
	In this subsection, we present the strong law of large numbers of the exact solution and the full discretization.
	\begin{theorem} \label{th:  strong law of large numbers}
		Suppose that Assumptions \ref{assu:1}-\ref{assu:3 noise} hold,  then for any bounded uniformly continuous function  $\phi:\ H \to \mathbb{R}$,   the following strong law of large numbers holds
		\begin{equation}\label{ineq:  strong law of large numbers}
			\lim _{T \rightarrow \infty} \frac{1}{T} \int_0^T \phi(u(t,\cdot)) \mathrm{d} t= \int_{H_\alpha} \phi(v) \pi(\mathrm{d} v),\quad \text {a.s.}
		\end{equation}    
	\end{theorem}
	\begin{proof}  	Without loss of generality, suppose $\int_{H_\alpha} \phi(v) \pi(\mathrm{d}v) = 0$. Otherwise, define 
		$\tilde{\phi} := \phi - \int_{H_\alpha} \phi(v) \pi(\mathrm{d}v)$ and consider $\tilde{\phi}$ in place of $\phi$.\par 
		Define 
		$$
		\mathcal{S}(t,\phi):=\int_0^t\phi(u(s,\cdot)) \mathrm{d}s,\ I_p(T):=\sup _{0 \leq t \leq T} \mathbb{E}\left[\left|\mathcal{S}(t,\phi)\right|^{2 p}\right], \ \text{and}\ 	\hat{g}\left(s_1, s_2\right)  :=	\phi\left(u(s_{1}) \right) \mathbb{E}\left[\phi\left(u(s_2,\cdot) \right)  \mid \mathscr{F}_{s_1}\right].
		$$ 
		By the tower property of conditional expectation and Hölder's inequality, we obtain
		$$
		\begin{aligned}
			\mathbb{E}\left[\left|\mathcal{S}(t,\phi)\right|^{2 p}\right]  
			=&\mathbb{E}\left[\int_{[0, t]^{2 p}} \prod_{i=1}^{2 p} \phi(u(s_i,\cdot))\  \mathrm{d} s_1 \cdots \mathrm{d} s_{2 p}\right] \\ 
			=&(2 p)!\mathbb{E}\left[\int_{0\leq s_1\leq\cdots\leq s_{2p-1} }\prod_{i=1}^{2 p-2} \phi\left(u(s_{i}) \right)\int_{0\leq   s_{2p-1}\leq  s_{2p}\leq t}
			\hat{g}\left(s_{2p-1}, s_{2p}\right)
			\mathrm{d} s_1 \cdots \mathrm{~d} s_{2 p}\right] \\ 
			\leq & \frac{(2p)!}{(2 p-2)!} \int_{0\leq   s_{2p-1}\leq  s_{2p}\leq t}
			\left(	\mathbb{E}\left[  \left|\mathcal{S}(s_{2p-1} ,\phi)\right|^{2 p} \right]\right)^{\frac{p-1}{p}}\left(\mathbb{E}\left[|
			\hat{g}\left(s_{2p-1}, s_{2p}\right)|^p\right]\right)^{\frac{1}{p}}
			\mathrm{d}s_{2p-1}  \mathrm{d} s_{2p} .
		\end{aligned}
		$$  
		It follows that
		$$
		I_p(T) \leq C \left(I_p(T)\right)^{\frac{p-1}{p}}\int_{0\leq   s_{1}\leq  s_{2}\leq T}\left(\mathbb{E}\left[\left|\hat{g}\left(s_1, s_2\right)\right|^p\right]\right)^{\frac{1}{p}} \mathrm{ d} s_1 \mathrm{ d} s_2,
		$$
		which yields 
		\begin{equation} \label{ineq: big number 1}
			\begin{aligned}
				I_p(T) \leq \left(C  \int_{0\leq   s_{1}\leq  s_{2}\leq T}\left(\mathbb{E}\left[\left|\hat{g}\left(s_1, s_2\right)\right|^p\right]\right)^{\frac{1}{p}} \mathrm{~d} s_1 \mathrm{d} s_2\right)^p.
			\end{aligned}
		\end{equation} 
		It follows from the Markov property and Proposition \ref{propo: error problem exact} that
		$$
		\begin{aligned}
			\mathbb{E}\left[\phi\left(u(s_2,\cdot) \right)  \mid \mathscr{F}_{s_1}\right]
			=&	\mathbb{E}\left[\phi\left(u(s_2,\cdot) \right) - \int_{H_\alpha} \phi(v) \pi(\mathrm{d}v)\ \Big|\mathscr{F}_{s_1}\right]\\
			=&   \int_{H_\alpha}	\mathbb{E}\left[  \phi\left(u(s_2,\cdot) \right)-  P_t \phi\left(v\right) |\mathscr{F}_{s_1}  \right]\pi(\mathrm{d} v)\   \\
			\leq &Ce^{-\frac{  \epsilon_0 }{2}(s_2-s_1)}\int_{H_\alpha}  \left\|u(u_0;s_1,\cdot)-u(v;s_1,\cdot)\right\|_{L^2}   \pi\left(d v\right) .  
		\end{aligned}
		$$   
		Furthermore, by Theorem \ref{th: u_regularity}, we conclude
		\begin{equation} \label{ineq: big number 2}
			\begin{aligned}
				\begin{aligned}
					\left(\mathbb{E}\left[\left|\hat{g}\left(s_1, s_2\right)\right|^p\right]\right)^{\frac{1}{p}}\leq &Ce^{-\frac{  \epsilon_0 }{2} (s_2-s_1) }.
				\end{aligned}
			\end{aligned}
		\end{equation} 
		Combining \eqref{ineq: big number 1} and \eqref{ineq: big number 2}, we obtain
		$$
		\mathbb{E}\left[\left|\frac{1}{T} \int_0^T\phi(u(t,\cdot))  \mathrm{d} t\right|^{2 p}\right] \leq
		\frac{I_p(T)}{T^{2p}} \leq 
		\left(\frac{C}{T^2} \int_{0\leq   s_{1}\leq  s_{2}\leq t}e^{-\frac{  \epsilon_0 }{2} (s_2-s_1) } \mathrm{~d} s_1 \mathrm{ d} s_2\right)^p\leq 
		CT^{-p}.$$ 
		It follows from the Borel–Cantelli Lemma that for any $\epsilon > 0$, 
		$$
		\left|\frac{1}{T} \int_0^T\phi(u(t,\cdot))  \mathrm{d} t-\int_{H_\alpha} \phi(v) \pi(\mathrm{d} v)\right| \leq C t^{-\frac{1}{2}+\epsilon},\quad \text {a.s.}
		$$
		Taking the limit as $T \to \infty$, this completes the proof.
	\end{proof}
	
	\begin{theorem} \label{th: err  strong law of large numbers}
		Suppose that Assumptions  \ref{assu:1}, \ref{assu:3 noise}, \ref{assu:2.2} hold,  then  for any bounded uniformly continuous function  $\phi:\ H \to \mathbb{R}$, the following strong law of large numbers holds
		\begin{equation}\label{ineq:  err strong law of large numbers}
			\lim _{T \rightarrow \infty} \lim _{\tau \rightarrow 0} \lim _{N \rightarrow \infty} \frac{\tau}{T} \sum_{m=0}^{\lfloor T/\tau\rfloor} \phi(u^{\tau,N}({t_m},\cdot))=\int_{H_\alpha} \phi(v) \pi(\mathrm{d} v),\quad \text {a.s.}
		\end{equation}    
	\end{theorem}
	\begin{proof}
		We now establish that
		\begin{equation}
			\begin{aligned}
				&\frac{\tau}{T} \sum_{m=0}^{\lfloor T/\tau\rfloor} \phi(u^{\tau,N}({t_m},\cdot))-\int_{H_\alpha} \phi(v) \pi(\mathrm{d} v)\\
				=&\left(\frac{\tau}{T} \sum_{m=0}^{\lfloor T/\tau\rfloor} \phi(u^{\tau,N}({t_m},\cdot))-\frac{1}{T} \int_0^T \phi(u(t,\cdot)) \mathrm{d} t\right) 
				+\left(\frac{1}{T} \int_0^T \phi(u(t,\cdot)) \mathrm{d} t -\int_{H_\alpha} \phi(v) \pi(\mathrm{d} v)\right)\\
				=&:\mathcal{K}_1^{\tau,N,T}+\mathcal{K}_2^{T} .
			\end{aligned}
		\end{equation} 
		Using Theorem \ref{thm: ||u^{M,N}-u||}, one obtains
		\begin{equation*}
			\begin{aligned}
				\mathbb{E}\left[ |\mathcal{K}_1^{\tau,N,T} |^2 \right]\leq &
				\mathbb{E}\Big[	\Big|\frac{\tau}{T} \sum_{m=0}^{\lfloor T/\tau\rfloor} \phi(u^{\tau,N}({t_m},\cdot))-\frac{\tau}{T} \sum_{m=0}^{\lfloor T/\tau\rfloor}  \phi\left(u(t_m,\cdot)\right)\Big|^2\Big]
				+	\mathbb{E}\Big[\Big|\frac{\tau}{T} \sum_{m=0}^{\lfloor T/\tau\rfloor}  \phi\left(u ({t_m},\cdot)\right)-
				\frac{1}{T} \int_0^T \phi(u(t,\cdot)) \mathrm{d} t\Big|^2\Big]\\
				\leq&
				C\left(h^{1-\epsilon}+ \tau^{\frac{3}{8}-\frac{\epsilon}{2}} \right).
			\end{aligned}
		\end{equation*} 
		By  the Borel--Cantelli Lemma,  it follows that
		\begin{equation*}
			\begin{aligned}
				\lim _{T \rightarrow \infty} \lim _{\tau \rightarrow 0} \lim _{N \rightarrow \infty} \mathcal{K}_1^{\tau,N,T}=0,\quad \text {a.s.}
			\end{aligned}
		\end{equation*}  
		According to Theorem \ref{th:  strong law of large numbers}, it is straightforward that 
		\begin{equation*}
			\begin{aligned}
				\lim _{T\rightarrow \infty}   \mathcal{K}_2^{T} =0,\quad \text {a.s.}
			\end{aligned}
		\end{equation*}
		This completes the proof.
	\end{proof}

	\section{Numerical example}  \label{section6}
	In this section, we provide   numerical experiments to validate the theoretical results  and to illustrate them from three perspectives: 
	(i) verification of the strong convergence rate of the proposed numerical scheme;  
	(ii)  demonstration of the numerical approximation of the ergodic limit;   
	(iii) investigation of the dependence of ergodic limits on initial conditions.\par 
	
	\begin{example}\label{ex: 6.1}
		We consider the SCHE driven by white noise as follows 
		\begin{equation} \label{eq: exm 1}
			\begin{cases}
				\dfrac{\partial u(t,x)}{\partial t}+\Delta^2 u(t,x)
				= \frac{1}{2}\Delta\!\left(u^3(t,x)-u^2(t,x)+2u(t,x)-2\right)
				+ \dfrac{\partial W(t,x)}{\partial t}, & t\in(0,T],\ x\in(0,\pi), \\[6pt]
				\frac{\partial u(t,0)}{\partial x}=\frac{\partial u(t,\pi)}{\partial x}=\frac{\partial^3 u(t,0)}{\partial x^3} =\frac{\partial^3 u(t,\pi)}{\partial x^3}=0,& t\in(0,T], \\  
				u(0, x)=u_0(x) ,  & x\in(0,\pi).
			\end{cases}
		\end{equation}   
	\end{example}

	\noindent\textbf{Part 1: Strong convergence rates of the numerical scheme.}  \par 
	To measure approximation errors in the mean-square sense, we employ the Monte--Carlo method with $L$  independent sample trajectories. The error is calculated by 
	$$
	E(\tau, N)=\max _{0 \leq i \leq {T}/{\tau}}\left(\frac{1}{L} \sum_{k=1}^{L} \max _{0 \leq j \leq N}\left|u^{\tau, N}\left(t_i, x_j, \omega_k\right)-u^{\tau_{\mathrm{ref}}, N_{\mathrm{ref}}}\left(t_i, x_j, \omega_k\right)\right|^2\right)^{\frac{1}{2}},
	$$ 
	where $\left\{t_i, x_j\right\}_{ 0\leq i\leq T/\tau, 0 \leq j \leq N}$ represents the grid points of $[0, T] \times[0, \pi]$ with temporal step-size $\tau$ and spatial step-size $\pi / N$. Here $u^{\tau, N}\left(t_i, x_j, \omega_k\right)$ is the numerical solution at the $k$-th trajectory. In the following, we take  $L=1000$. To verify the strong convergence rate, we simulate the reference solution by taking $\tau_\mathrm{ref}=T\times2^{-13}$ and $N_\mathrm{ref}=2^{10}$. 
	
	We first examine the strong convergence rates of the fully discrete scheme \eqref{eq: mild_solution_full_discrete}. The mean-square errors  and strong convergence rates are reported in Table \ref{tbl:1}, which aligns with the theoretical predictions stated in Theorem  \ref{thm: ||u^{M,N}-u||}. These results demonstrate that the mean-square temporal convergence rate of the full discretization is about $\frac{3}{8}$, and the spatial convergence rate is about $1$.  \par 
	
	\begin{table}[htbp]
		\centering  
		\caption{ Mean-square errors and  strong convergence rates ( $T=100,$ $u_0=\frac{1}{3}\cos(x)+\frac{1}{3} $).}
		\begin{tabular}{ c c  c c  c  c }
			\specialrule{1pt}{10pt}{1pt} \specialrule{0.5pt}{1pt}{2pt} 
			$(\tau,N)$   &   $E(\tau,N)$        & rate            &  	  $(\tau,N)$   &   $E(\tau,N)$        & rate           \\ \specialrule{1pt}{2pt}{2pt}
			$(T\times2^{-4}, 2^9)$  &  $5.19 \times 10^{-1}$     &       &   	$(T\times2^{-13}, 2^3)$  &  $1.04 \times 10^{-1}$     &                 \\ [1.5mm]
			$(T\times2^{-5}, 2^9)$   & $4.32 \times 10^{-1}$      &0.264  & 		$(T\times2^{-13}, 2^4)$ & $5.05 \times 10^{-2}$      & 1.041  \\ [1.5mm]
			$(T\times2^{-6}, 2^9)$   & $3.61 \times 10^{-1}$      &0.262  &    	$(T\times2^{-13}, 2^5)$ & $2.21 \times 10^{-2}$      & 1.189  \\ [1.5mm]
			$(T\times2^{-7}, 2^9)$   & $2.61 \times 10^{-1}$      &0.634  &     	$(T\times2^{-13}, 2^6)$ & $1.01\times 10^{-2}$       & 1.134  \\[1.5mm]
			$(T\times2^{-8}, 2^9)$   &$2.07\times 10^{-1}$        &0.338  &  		$(T\times2^{-13}, 2^7)$ & $5.18 \times 10^{-3}$      & 0.962  \\  
			\specialrule{0.5pt}{2pt}{1pt} \specialrule{1pt}{1pt}{0pt}
		\end{tabular}\label{tbl:1} 
	\end{table}

	\noindent\textbf{Part 2: Numerical approximation of the ergodic limit.}  \par 
	Next, we investigate the long-time behavior of the solution to \eqref{eq: exm 1}. According to Remark \ref{re: 5.2 numerical ergodic limit} and Theorem \ref{th: err  strong law of large numbers}, the time average of $\phi\big(u^{\tau,N}(u_0;t,\cdot)\big)$    converges almost surely to the ergodic limit $\int_{H_\alpha} \phi(v) \pi(\mathrm{d} v)$, and t
	This  motivates us to approximate the ergodic limit using two approaches:
	\begin{itemize}
		\item \textbf{Approximation I}:\     $  \frac{\tau}{LT} \sum_{m=0}^{\lfloor T/\tau\rfloor}\sum_{k=1}^{L} \phi\left( u^{\tau, N}\left(t_m, \cdot\ , \omega_k\right)\right)  \approx \int_{H_\alpha} \phi(v) \pi(\mathrm{d} v);$ 
		\item \textbf{Approximation II}:  \   $  \frac{\tau}{T} \sum_{m=0}^{\lfloor T/\tau\rfloor}  \phi\left( u^{\tau, N}\left(t_m, \cdot\ , \omega\right)\right)  \approx \int_{H_\alpha} \phi(v) \pi(\mathrm{d} v).$ 
	\end{itemize}
	
	We define the function $g: H\times H\times\mathbb{R}\to \mathbb{R}$ by
	$$   g(v ,u,\alpha_1):=\int_{\mathcal{O}}v(x) u(x) \mathrm{d} x-\alpha_1\int_{\mathcal{O}}   v(x)  \mathrm{d} x\int_{\mathcal{O}}   u(x)  \mathrm{d} x,\quad \forall v,u\in H,\ \alpha_1\in \mathbb{R}.   $$ 
	Furthermore, we introduce the test function  $\phi_{v,\alpha_1,\alpha_2} :  H\to \mathbb{R}$ for any $v\in H$ and $\alpha_1,\ \alpha_2\in\mathbb{R}$ defined by
	$$  \phi_{v,\alpha_1,\alpha_2}(u):=   \frac{\alpha_2  g(v,u,\alpha_1)}{ 1+  \left| \frac{ 1}{\alpha_2}g(v,u,\alpha_1)  \right|^2},\quad  \forall u\in H.   $$
	This test function offers flexibility through the choice of $v$ and $\alpha_2$, thereby yielding a family of bounded and uniformly continuous functions whose ergodic limit is zero when $\alpha_1 = 1$ and $\alpha = -\frac{a_1}{3a_0}$, where $a_1$ and $a_3$ are the coefficients from the nonlinear term $f$, and  $\alpha$ denotes the parameter appearing in the index of $H_\alpha$ (the proof is provided in Appendix E).
	Specifically, we select three initial conditions $u_0 = \frac{1}{3}$, $\frac{1}{3} \cos(x) + \frac{1}{3}$, and $2 \cos(2x) + \cos(x) + \frac{1}{3}$, together with parameter  triples $(v, \alpha_1,\alpha_2) = (e^x,1, 2)$ and $(e^{-x},1, 3)$. These choices allow us to construct the  test functions $\phi_{v,\alpha_1,\alpha_2}$ and validate the two  approximations.\par 
	
	As shown in Figures \ref{fig:combined 1} and \ref{fig:combined 2}, both Approximations I and II converge to the same ergodic limit, regardless of the initial condition. Table \ref{tbl:2} compares the computational cost and error of both approximations. Notably, the CPU time required for Approximation  II  is  significantly less than that for Approximation I. This efficiency improvement comes from the fact  that Approximation  II uses only a single sample trajectory, thereby avoiding the simulation of a large number of samples and offering an advantage in computational efficiency.  \par 
	
	\begin{figure}[htbp]  
		\centering 
		\subfigure[ $\alpha_1=1$, $\alpha_2=2$, $v=e^x$ and $T=100$ .]{
			\includegraphics[width=0.45\textwidth]{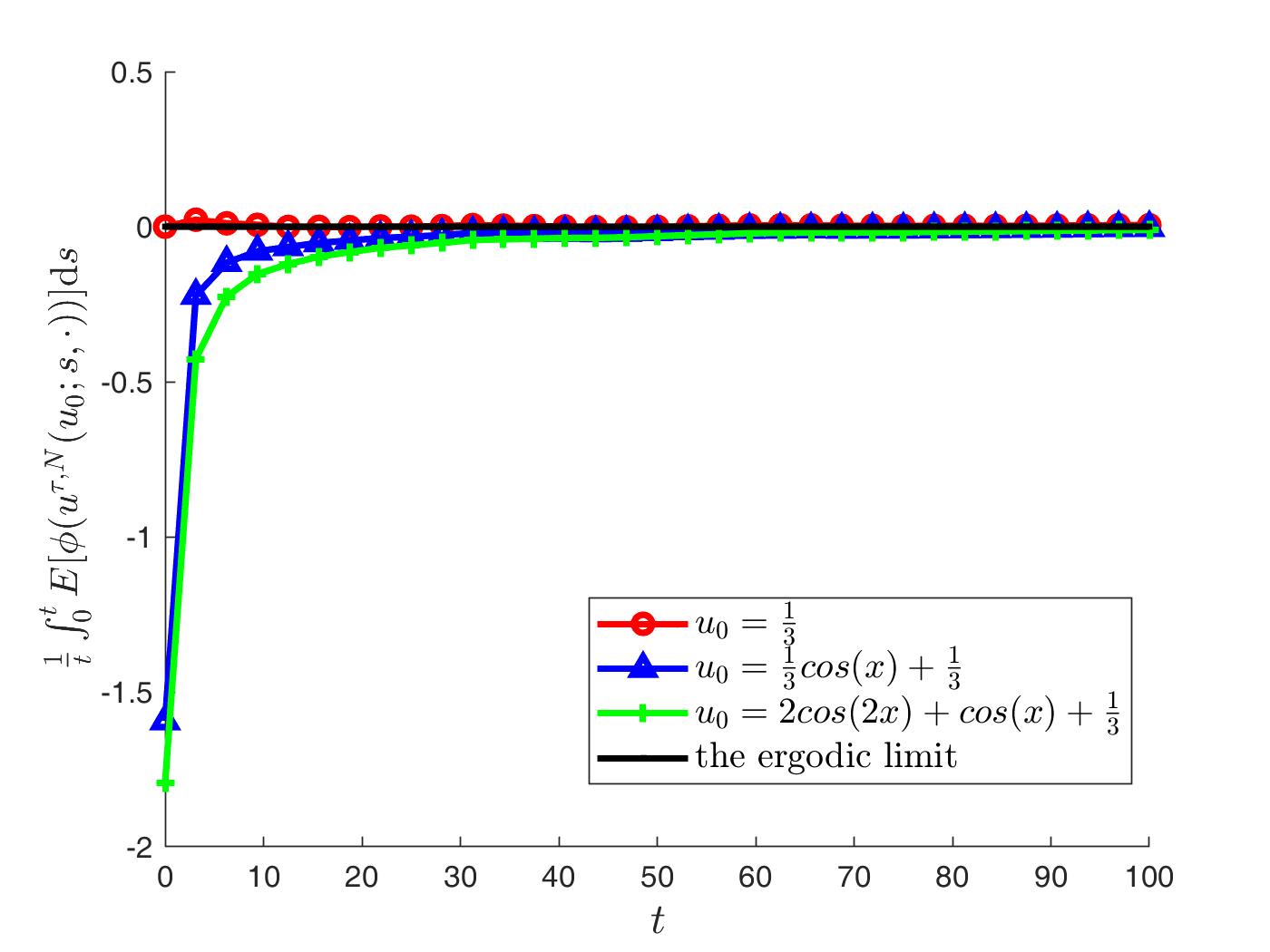}
		}
		\hspace{0.04\textwidth}
		\subfigure[ $\alpha_1=1$, $\alpha_2=3$, $v=e^{-x}$ and $T=100$.]{
			\includegraphics[width=0.45\textwidth]{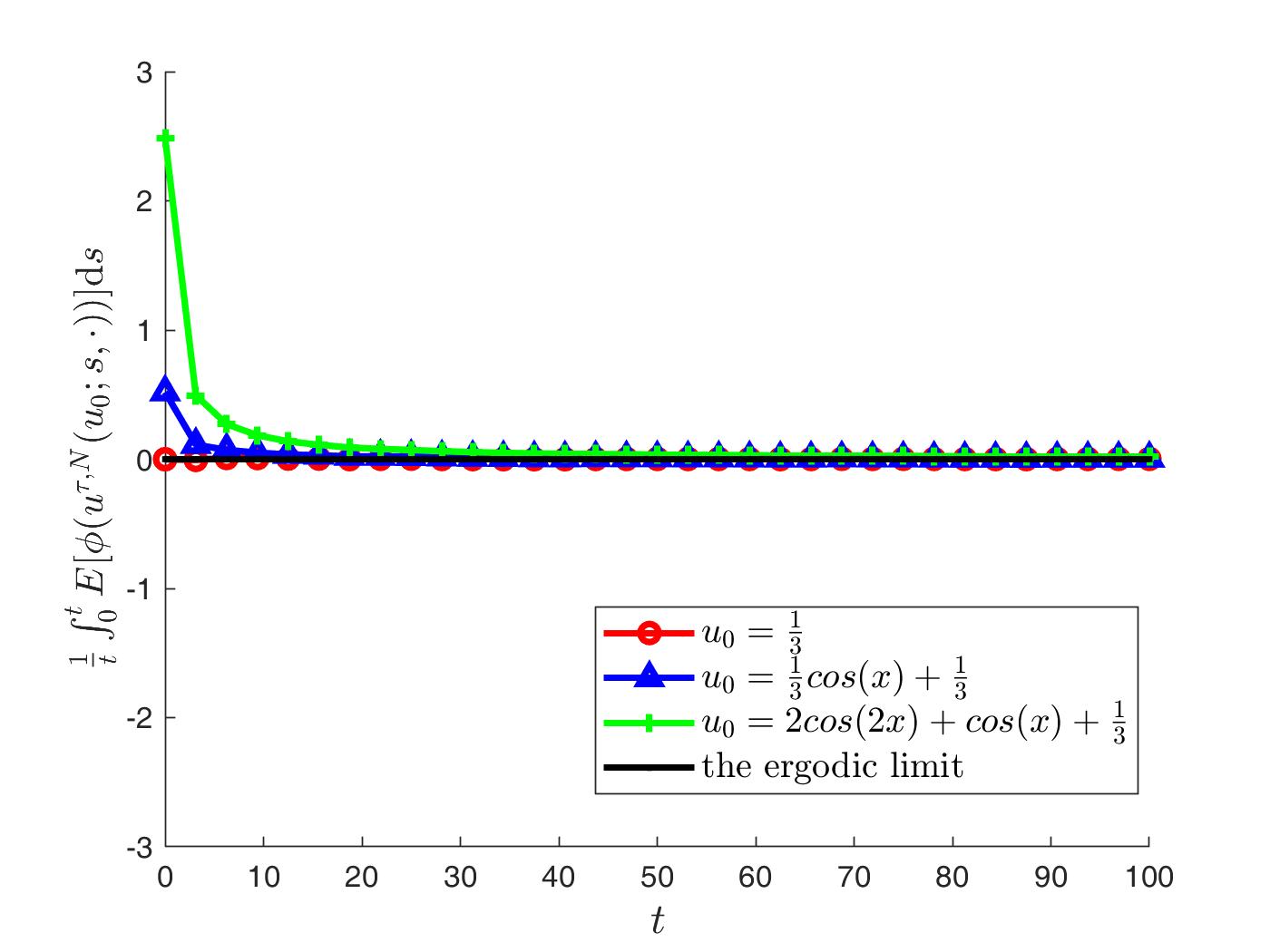}
		}
		
		\caption{The time average of $\mathbb{E}\left[\phi\big(u^{\tau,N}(u_0;t,\cdot)\big)\right]$ started from different initial values.}
		\label{fig:combined 1}
	\end{figure}

	\begin{figure}[htbp]  
		\centering 
		\subfigure[ $\alpha_1=1$, $\alpha_2=2$, $v=e^x$ and $T=100$ .]{
			\includegraphics[width=0.45\textwidth]{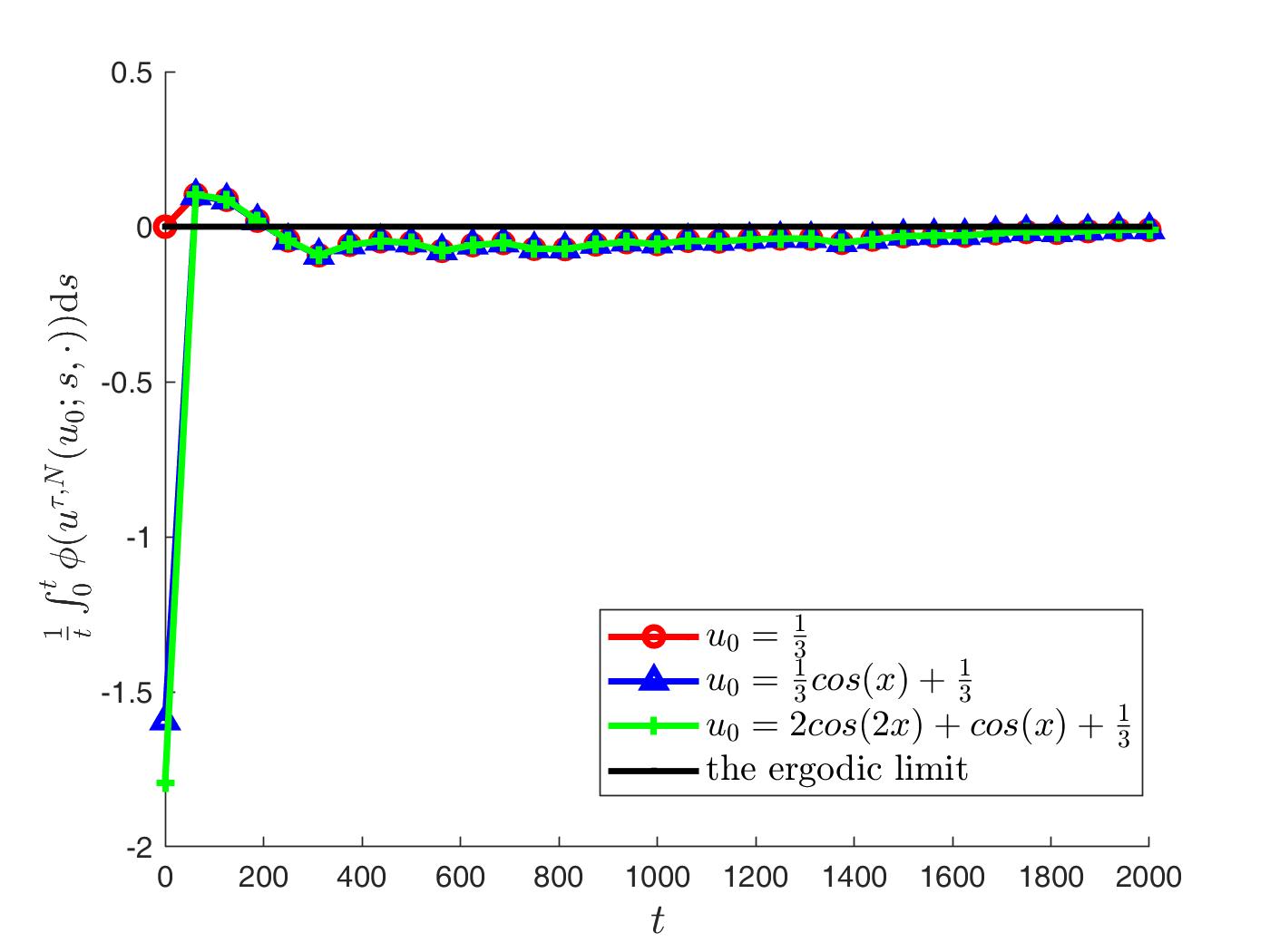}
		}
		\hspace{0.04\textwidth}
		\subfigure[ $\alpha_1=1$, $\alpha_2=3$, $v=e^{-x}$ and $T=100$.]{
			\includegraphics[width=0.45\textwidth]{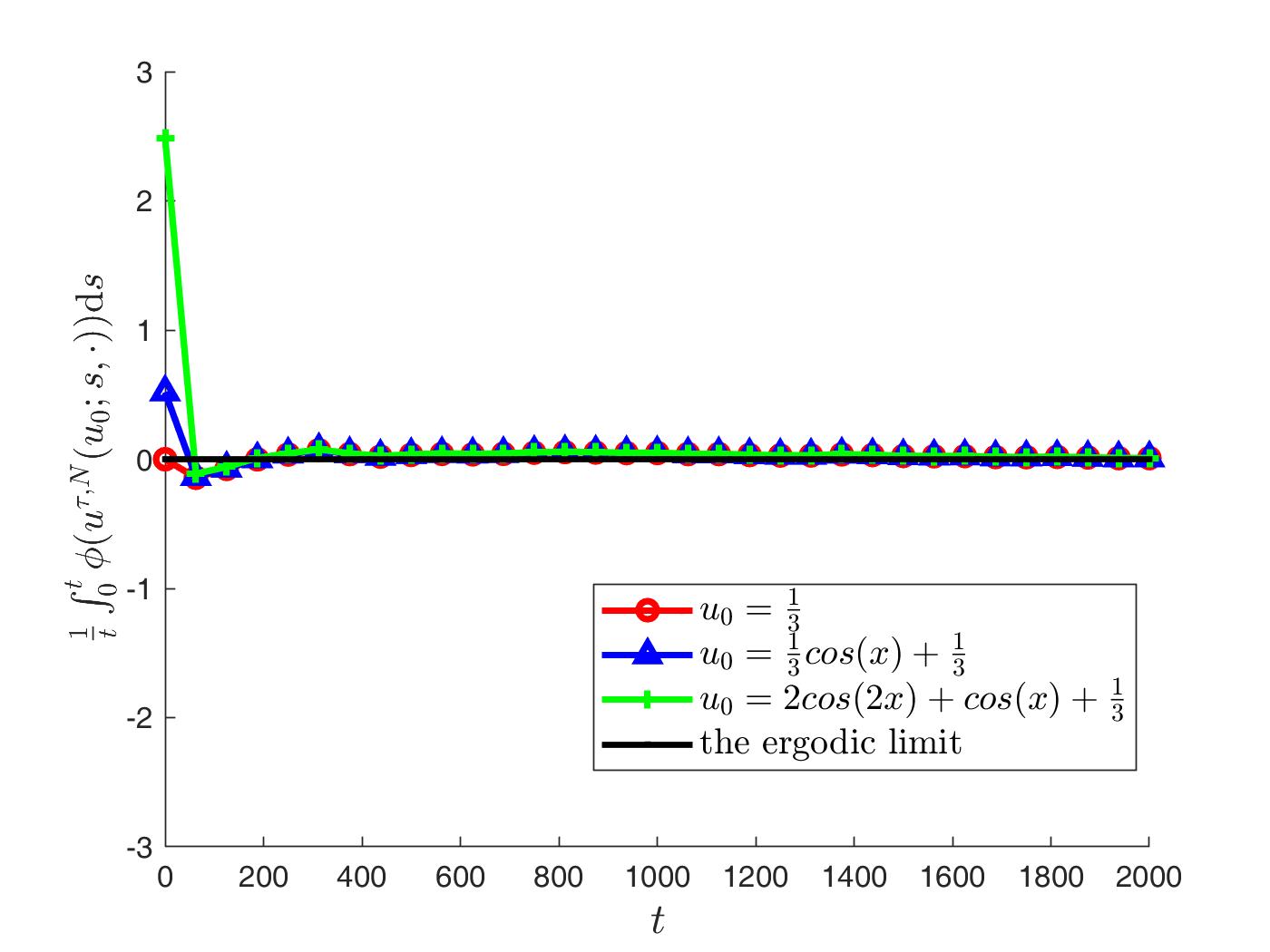}
		}
		
		\caption{The time average of $ \phi\big(u^{\tau,N}(u_0;t,\cdot)\big) $ started from different initial values.}
		\label{fig:combined 2}
	\end{figure}

	\begin{table}[htbp]
		\centering  
		\caption{ Errors and CPU time for approximating the ergodic limit.}
		\resizebox{\textwidth}{!}{ 
			\begin{tabular}{c c c  c c c c c  }
				\specialrule{1pt}{10pt}{1pt} \specialrule{0.5pt}{1pt}{2pt} 
				&		& \multicolumn{2}{c}{$u_0=\frac{1}{3}$} & \multicolumn{2}{c}{$u_0=  \frac{1}{3}\cos(x)+\frac{1}{3}$}  & \multicolumn{2}{c}{$u_0= \sum_{k=1}^{2} k\cos(kx)+\frac{1}{3}$} \\
				\cmidrule(r){3-4} \cmidrule(r){5-6}   \cmidrule(r){7-8} 
				$(v,\alpha_1,\alpha_2)$	& Approximation  &  error        & CPU time(s)       &  error        & time(s)    &  error        & CPU time(s)         \\ \specialrule{1pt}{2pt}{2pt} 
				$(e^x,1,2)$ & 	I $(T=100)$ &
				$4.92 \times 10^{-3}$     &   1795    &   $3.25 \times 10^{-3}$   &  1867    & $1.03 \times 10^{-2}$ & 1884   \\ [1.5mm] 
				&   II $(T=2000)$ &$9.54\times 10^{-3}$     &  214
				&$9.91\times 10^{-3}$     &  268     &$1.01\times 10^{-2}$     &  273   \\   \specialrule{0.5pt}{2pt}{1pt} 
				$(e^{-x},1,3)$ &  	I $(T=100)$ &
				$2.91 \times 10^{-3}$     &   1816    &   $6.88 \times 10^{-3}$   &  1768   & $1.97 \times 10^{-2}$ & 1988  \\ [1.5mm] 
				&   II $(T=2000)$ &$8.72\times 10^{-3}$     &  236  
				&$9.00\times 10^{-3}$     &  244    &$1.01\times 10^{-2}$     &  229     \\   
				\specialrule{0.5pt}{2pt}{1pt} \specialrule{1pt}{1pt}{0pt}
			\end{tabular}\label{tbl:2}  }
	\end{table}

	\noindent\textbf{Part 3: Dependence of ergodic limits on initial conditions.}  \par
	Finally, we examine the dependence of the ergodic limit on   initial conditions. We consider three distinct initial conditions, each belonging to a different $H_\alpha$ space:  
	$u_0 = \frac{1}{3}$ ($\alpha=\frac{1}{3}$),  
	$\frac{1}{3}\cos(x) + 1$ ($\alpha=1$), and  
	$2 \cos(2x) + \cos(x) + 2$ ($\alpha=2$),  
	together with the parameter triple $(v, \alpha_1, \alpha_2) = (e^x, 0, 3)$.  
	As shown in Figure \ref{fig: longtime b}, the ergodic limits will be different depending on the initial condition, which aligns with Remark \ref{rem: no unique} stating that the SCHE does not admit a unique invariant measure on $H$. 
	\begin{figure}[H]  
		\centering
		\includegraphics[width=0.5\textwidth]{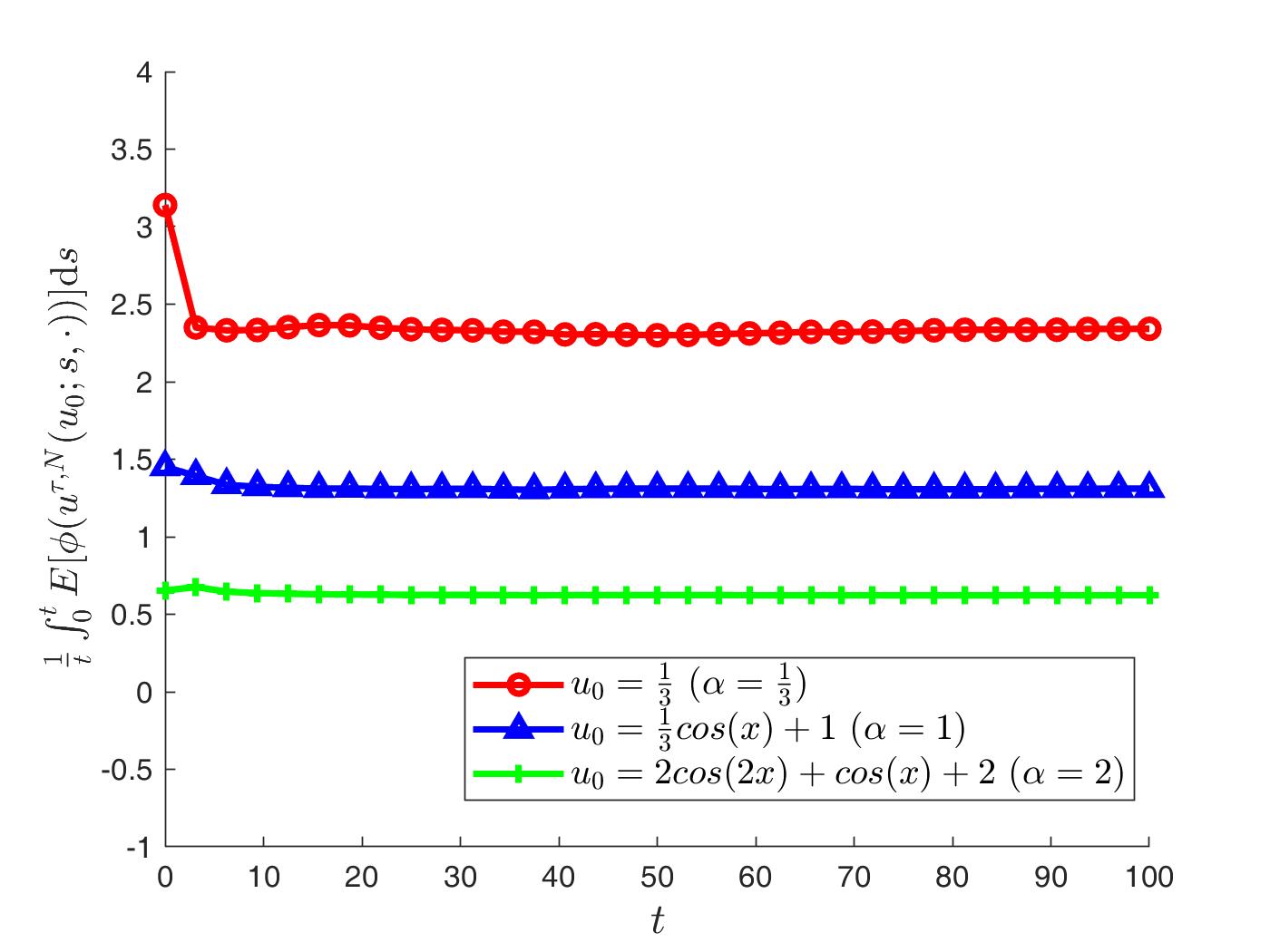}  
		\caption{Ergodic limits for   initial conditions  in different spaces.}
		\label{fig: longtime b}%
	\end{figure}

	\section*{Conclusion}
	
	In this work, we extend the classical ergodicity theory of \cite{da1996stochastic} for the stochastic Cahn--Hilliard equation (SCHE) from the negative Sobolev space $\dot{H}^{-1}_\alpha$ to the physically more relevant Hilbert space $H_\alpha$, and we show that uniqueness of invariant measures fails on $H=L^2(\mathcal{O})$.  We further introduce a fully discrete scheme—finite differences in space combined with a tamed exponential Euler method in time—for which uniform-in-time moment bounds are established.  This enables us to prove uniform strong convergence with rates $1-\varepsilon$ in space and $\frac{3}{8}-\frac{\varepsilon}{2}$ in time.  
	
	On the long-time side, we develop a mass-preserving minorization framework adapted to Neumann boundary conditions and prove that the numerical scheme is geometrically ergodic with a unique invariant measure.  A quantitative approximation bound is obtained for the numerical invariant measure, and strong laws of large numbers are derived for both the continuous and discrete systems, ensuring almost-sure convergence of temporal averages to their ergodic expectations.
	
	The numerical experiments corroborate our theoretical findings.  
	(i) The observed spatial and temporal convergence rates agree with the theoretical predictions.  
	(ii) Numerical invariant measures converge to the exact one, and computing ergodic limits via a single long trajectory is significantly more efficient than ensemble averaging.  
	(iii) In the space $H$, invariant measures are not unique; this is illustrated by constructing a family of test functions $\phi_{v,\alpha_1,\alpha_2}$ whose ergodic limits can be tuned, for instance, to zero by appropriate choices of parameters.
	
	Several directions for future research arise naturally.  
	Extending the present analysis to higher-dimensional SCHEs or to SPDEs with multiplicative noise remains an important challenge.  
	Another promising direction is the design of more efficient explicit schemes that preserve long-time statistical structures.  
	Finally, establishing uniform weak convergence for the proposed method is a compelling problem for further study.
	
	More broadly, the techniques developed in this work—perturbed-equation moment estimates, the strengthened taming strategy, and the mass-preserving minorization argument—are expected to extend to other fourth-order SPDEs and phase-field models.  
	This underlines the wider potential of the present approach beyond the SCHE itself.

\section*{Appendix A. Proof of Lemma 2.1}
\begin{proof} 
	Observe $\lambda_{0}=0$ and $\lambda_j\geq 1$ for $j\geq1$. It follows from expansion in terms of the eigenbasis of $A$, Parseval’s identity and
	\eqref{estimate: e^{-x} 1} that 
	\begin{equation*}
		\begin{aligned}
			\left\|(-A)^\gamma e^{-A^2t} \hat{u}\right\|_{L^{2}}^2=	& \sum_{j=0}^{\infty} \lambda_j^{2\gamma}e^{-2\lambda_{j}^2t}\left|\left\langle \hat{u}, \phi_j\right\rangle_{L^2}\right|^2  
			\leq    e^{-\lambda_{1}t} \sum_{j=1}^{\infty} \lambda_j^{2\gamma}e^{- \lambda_{j}^2t}\left|\left\langle \hat{u}, \phi_j\right\rangle_{L^2}\right|^2 + g(\gamma)\left|\left\langle \hat{u}, \phi_0\right\rangle_{L^2}\right|^2\\
			\leq & C e^{-\lambda_{1}t}t^{-\gamma}\sum_{j=1}^{\infty}  \left|\left\langle \hat{u}, \phi_j\right\rangle_{L^2}\right|^2 + g(\gamma)\left|\left\langle \hat{u}, \phi_0\right\rangle_{L^2}\right|^2
			\leq C e^{-\lambda_{1}t} t^{-\gamma} \left\|\hat{u}\right\|_{L^2}^2+ g(\gamma)\left|\left\langle \hat{u}, \phi_0\right\rangle_{L^2}\right|^2,
		\end{aligned}
	\end{equation*}
	which yields \eqref{ineq: e^{-A^2t}(-A)^r to l2}.
	Similarly, we can obtain
	\begin{equation*}
		\begin{aligned}
			\left\|(-A)^{\gamma_1} e^{-A^2t}\left( I- e^{-A^2s}\right)\hat{u}\right\|_{L^{2}} ^2 
			\leq &  e^{-\lambda_{1}t} \sum_{j=1}^{\infty} \lambda_j^{2\gamma_1}e^{- \lambda_{j}^2t}(1-e^{- \lambda_{j}^2s})^2\left|\left\langle \hat{u}, \phi_j\right\rangle_{L^2}\right|^2 + g(\gamma_1)\left|\left\langle \hat{u}, \phi_0\right\rangle_{L^2}\right|^2 ,
		\end{aligned}
	\end{equation*} 
	which together with \eqref{estimate: e^{-x} 1} yields
	\begin{equation*}
		\begin{aligned}
			\left\|(-A)^{\gamma_1} e^{-A^2t}\left( I- e^{-A^2s}\right)\hat{u}\right\|_{L^{2}} ^2 	 
			\leq & C e^{-\lambda_{1}t} t^{ -(\gamma_1-\gamma_2+2\gamma_3) }s^{2\gamma_3}\sum_{j=1}^{\infty}\lambda_{ j}^{2\gamma_2}  \left|\left\langle \hat{u}, \phi_j\right\rangle_{L^2}\right|^2 + g(\gamma_1) \left|\left\langle \hat{u}, \phi_0\right\rangle_{L^2}\right|^2\\
			= &C e^{-\lambda_{1}t} t^{ -(\gamma_1-\gamma_2+2\gamma_3) }s^{2\gamma_3}\left\|(-A)^{\gamma_2}\hat{u}\right\|_{L^2}^2+g(\gamma_1) \left|\left\langle \hat{u}, \phi_0\right\rangle_{L^2}\right|^2.
		\end{aligned}
	\end{equation*}

\end{proof}

\section*{Appendix B. Proof of Lemma 3.1}
\begin{proof}  
	Utilizing  the orthonormality of $\left\{\phi_j\right\}_{j\geq 0} $,   we obtain 
	\begin{equation*} 
		\begin{aligned}
			\int_\mathcal{O}|  G_{t}(x, y)- G_{t}(z, y)|^2 \mathrm{d} y =&\int_\mathcal{O}\left|\sum_{j=0}^{\infty} e^{-\lambda_j^2 t} \left(\phi_j(x)-\phi_j(z)\right) \phi_j(y)\right|^2\mathrm{d} y\\
			\leq& \sum_{j=0}^{\infty} e^{-2\lambda_j^2 t} \left|\phi_j(x)-\phi_j(z)\right|^2\\
			\leq & e^{-\lambda_{1}t}\sum_{j=0}^{\infty}\lambda_j e^{-\lambda_j^2 t} \left|x-z\right|^2, 
		\end{aligned}
	\end{equation*}
	where we use $ \left|\phi_j(x)-\phi_j(z)\right|\leq \lambda_{j}^\frac{\beta}{2}|x-y|^\beta$ for any $\beta\in [0,1]$.
	Similarly, using Hölder’s inequality, one can have
	\begin{equation*} 
		\begin{aligned}
			\int_\mathcal{O}\left|\Delta G_{t}(x, y)-\Delta G_{t}(z, y)\right|\mathrm{d} y  
			=&\int_\mathcal{O}\left|\sum_{j=0}^{\infty} \lambda_{j} e^{-\lambda_j^2 t} \left(\phi_j(x)-\phi_j(z)\right) \phi_j(y)\right| \mathrm{d} y\\
			\leq& C\sqrt{\sum_{j=0}^{\infty} \lambda_{j}^2 e^{-2\lambda_j^2 t}\left|  \phi_j(x)-\phi_j(z) \right|^2} \\
			\leq & Ce^{-\frac{1}{2}t}\left|x-z\right|^\alpha \sqrt{ \sum_{j=0}^{\infty}\lambda_j^{2+ \alpha }e^{-\lambda_j^2 t}}.
		\end{aligned}
	\end{equation*}
\end{proof}

\section*{Appendix C. Proof of Lemma 4.2}
\begin{proof}  
	According to the definitions of $G^N_t(x,y)$ and  $G_t(x,y)$, we  can separate $\Delta_N G_t^N(x, y)-\Delta G_t(x, y)$ into the following four terms:
	\begin{equation*} 
		\begin{aligned}
			\Delta_N G_t^N(x, y)-\Delta G_t(x, y)=\sum_{k=1}^5  \mathcal{L}_k^N(t,x,y),
		\end{aligned}
	\end{equation*}
	where 
	$$
	\begin{aligned}
		&  \mathcal{L}_1^N(t,x,y)=  \sum_{j=N}^{\infty}\lambda_{j} e^{-\lambda_j^2 t} \phi_j(x) \phi_j(y), \ 
		\mathcal{L}_2^N(t,x,y)=\sum_{j=0}^{N-1} \lambda_{N, j} e^{-\lambda_{N, j}^2 t} \bar{\phi}_{N, j}(x) \left(\phi_j(y)-\phi_j\left(\kappa_N(y)\right)\right), \\
		&\mathcal{L}_3^N(t,x,y)=\sum_{j=0}^{N-1} \lambda_{N, j}e^{-\lambda_{N, j}^2 t} \left(\phi_j(x) -\bar{\phi}_{N, j}(x)\right) \phi_j(y), \
		\mathcal{L}_4^N(t,x,y)=\sum_{j=0}^{N-1}\lambda_{N, j} \left( e^{-\lambda_j^2 t}- e^{-\lambda_{N, j}^2 t}\right) \phi_j(x) \phi_j(y),  \\
		&\mathcal{L}_5^N(t,x,y)=\sum_{j=0}^{N-1}\left(\lambda_{j} - \lambda_{N, j} \right) e^{-\lambda_j^2 t} \phi_j(x) \phi_j(y)  .
	\end{aligned}
	$$ 
	It follows from  the orthonormality of $\left\{\phi_j\right\}_{j\geq 0} $ and \eqref{estimate: e^{-x} 1} with $\alpha=\frac{ 2\alpha_1+5}{4}$ that 
	\begin{equation} \label{es: L1}
		\begin{aligned}
			\sup _{x \in \mathcal{O}} \int_\mathcal{O}\left| \mathcal{L}_1^N(t,x,y)\right|^2 \mathrm{d} y \leq \sup _{x \in \mathcal{O}}\sum_{j=N}^{\infty} \lambda_{j}^2e^{-2\lambda_j^2 t} |\phi_j(x)|^2  \leq C e^{-\lambda_1^2 t}t^\frac{ 2\alpha_1+5}{4}  \sum_{j=N}^{\infty} \lambda_{j}^{-\frac{ 2\alpha_1+1 }{2}}
			\leq  C e^{-\lambda_1^2 t}t^\frac{ 2\alpha_1+5}{4}  N^{-2\alpha_1}.
		\end{aligned}
	\end{equation} 
	We recall Eq. (3.10) from Lemma 3.2 in \cite{Cardon2000}, which states
	\begin{equation*} \label{eq: i neq j=0}
		\int_\mathcal{O}    \phi_i\left(\kappa_N(y)\right) \phi_j(y) \mathrm{d} y=0,\quad  i \neq j.
	\end{equation*}   
	This identity, together with \eqref{estimate: e^{-x} 1} with $\alpha=\frac{2 \alpha_1+5}{4}$, implies
	\begin{equation}  \label{es: L2}
		\begin{aligned}
			\sup _{x \in \mathcal{O}} \int_\mathcal{O}\left| \mathcal{L}_2^N(t,x,y)\right|^2 \mathrm{d} y 
			=& 	\sup _{x \in \mathcal{O}} \int_\mathcal{O} \sum_{j=1}^{N-1}\lambda_{N,j}^2   e^{-2\lambda_{N,j}^2 t} \left|\bar{\phi}_{N, j}(x)\right|^2   \left| \left(\phi_j\left(\kappa_N(y)\right)-\phi_j(y)\right)\right|^2\mathrm{d} y  \\
			\leq & C e^{- \lambda_{N,1}^2 t} \sum_{j=1}^{N-1} \lambda_{N,j}^2  e^{- \lambda_{N,j}^2 t} \int_\mathcal{O} \left| 
			\int_{\kappa_N(y)}^{y} j\sin(jz)\ \mathrm{d} z \right|^2\mathrm{d} y\\
			\leq & C  e^{- \lambda_{N,1}^2 t}N^{-2 } \sum_{j=1}^{N-1} \lambda_{N,j}^{3}e^{- \lambda_{N,j}^2 t}   \leq C  e^{- \lambda_{N,1}^2 t}  t^{-\frac{2 \alpha_1+5}{4}}N^{-2 \alpha_1}.
		\end{aligned}
	\end{equation} 
	In a similar manner, we can derive
	\begin{equation}  \label{es: L3}
		\begin{aligned}
			\sup _{x \in \mathcal{O}} \int_\mathcal{O}\left| \mathcal{L}_3^N(t,x,y)\right|^2 \mathrm{d} y 
			=&  \sum_{j=1}^{N-1}\lambda_{N,j}^2   e^{-2\lambda_{N,j}^2 t}	\sup _{x \in \mathcal{O}}   \left|\phi_j(x) -\bar{\phi}_{N, j}(x)\right|^2  \\ 
			\leq & C  e^{- \lambda_{N,1}^2 t}N^{-2 } \sum_{j=1}^{N-1} \lambda_{N,j}^{3}e^{- \lambda_{N,j}^2 t}   \leq C  e^{- \lambda_{N,1}^2 t}  t^{-\frac{2 \alpha_1+5}{4}}N^{-2 \alpha_1}.
		\end{aligned}
	\end{equation} 
	Given that
	$$\left|\lambda_{N,j}-\lambda_j \right|\leq Cj^4N^{-2}\  \ \text{and} \ \ \lambda_{N,j}\leq \lambda_j,\quad   1\leq j \leq N-1,$$
	and applying  \eqref{estimate: e^{-x} 1} with $\alpha=\frac{6 \alpha_1+5}{4}$ and   \eqref{estimate: e^{-x} 1} with $\hat{\alpha}=\frac{  \alpha_1 }{2}$, we obtain 
	\begin{equation}  \label{es: L4}
		\begin{aligned}
			\sup _{x \in \mathcal{O}} \int_\mathcal{O}\left| \mathcal{L}_4^N(t,x,y)\right|^2 \mathrm{d} y 
			=&  \sum_{j=1}^{N-1}\lambda_{N,j}^2   e^{-2\lambda_{N,j}^2 t}\left(1- e^{-2\left(\lambda_{j}^2-\lambda_{N,j}^2\right)  t}\right)^2  
			\leq C  e^{- \lambda_{N,1}^2 t}  t^{-\frac{2 \alpha_1+5}{4}}N^{-2 \alpha_1}.
		\end{aligned}
	\end{equation} 
	Similarly we can obtain 
	\begin{equation}  \label{es: L5}
		\begin{aligned}
			\sup _{x \in \mathcal{O}} \int_\mathcal{O}\left| \mathcal{L}_5^N(t,x,y)\right|^2 \mathrm{d} y 
			=&  \sum_{j=1}^{N-1}\left(\lambda_{j} - \lambda_{N, j} \right)^2   e^{-2\lambda_{j}^2 t}    \leq C  e^{- \lambda_{N,1}^2 t}  t^{-\frac{2 \alpha_1+5}{4}}N^{-2 \alpha_1}.
		\end{aligned}
	\end{equation} 
	\par By  H{\"o}lder's inequality and \eqref{es: L1}--\eqref{es: L5} one can get
	\begin{equation*}   
		\begin{aligned}
			\sup _{x \in \mathcal{O}}	\int_O\left|\Delta_N G_t^N(x, y)-\Delta G_t(x, y)\right|  \mathrm{ d} y \leq&C  \left(\int_O \sup _{x \in \mathcal{O}}\left|\Delta_N G_s^N(x, y)-\Delta G_s(x, y)\right|^2  \mathrm{ d} y \right)^\frac{1}{2}\\ 
			\leq&C  \left(\int_O \sup _{x \in \mathcal{O}} \sum_{k=1}^{5}	\left| \mathcal{L}_k^N(t,x,y)\right|^2   \mathrm{ d} y \right)^\frac{1}{2}\\
			\leq& C e^{-\lambda_{N, 1} t} t^{-\frac{2 \alpha_1+5}{8}} N^{-\alpha_1}.
		\end{aligned}
	\end{equation*}    
	This completes the proof.  
\end{proof}

\section*{Appendix D. Proof of Proposition 4.1}
\begin{proof}
	To estimate $\left\|u(t, \cdot)-u^N(t, \cdot)\right\|_{L^p(\Omega;L^\infty)}$, we introduce the auxiliary process $  \{\tilde{u}^N(t, x) \}_{(t, x) \in [0, T] \times \mathcal{O}}$ by
	\begin{equation}\label{eq:auxiliary_CH_FractionalNoise_solution_mild}
		\begin{aligned}
			\tilde{u}^N(t, x)= & \int_{\mathcal{O}} G_t^N(x, y) u_0\left(\kappa_N(y)\right) \mathrm{~d} y+\int_0^t \int_{\mathcal{O}} \Delta_N G_{t-s}^N(x, y) f\left(u\left(s, \kappa_N(y)\right)\right) \mathrm{~d} y \mathrm{d} s \\
			& +\sigma\sum_{j=1}^{N-1} \int_0^t \int_{\mathcal{O}} G_{t-s}^N(x, y)\phi_j\left(\kappa_N(y)\right) \mathrm{d} y \mathrm{d}  \beta_j(s)  , \quad (t, x) \in[0,\infty) \times \mathcal{O}.
		\end{aligned}
	\end{equation}
	Let $\tilde{U}^N_t=\big(\tilde{U}^N_{1,t} ,\cdots,\tilde{U}^N_{N,t}\big)^\top:=\big(\tilde{u}^N(t, x_1) ,\cdots,\tilde{u}^N(t, x_{N})\big)^\top $, which satisfies
	\begin{equation}\label{eq:eq:discrete_CH_FractionalNoise_SDE_matrix auxiliary process}
		\left\{  
		\begin{aligned}
			&\mathrm{d}\tilde{U}^N_t+A_N^2 \tilde{U}^N_t \mathrm{d} t=A_N F_N(\mathbb{U}_t) \mathrm{d} t+\sigma  \sum_{j=1}^{N-1}    \phi_{N,j}   \mathrm{d}  \beta_j(t)  , \quad t>0,\\
			&\tilde{U}_0=U_0.
		\end{aligned}
		\right.
	\end{equation}
	Here $\mathbb{U}_t:=\big(u(t, x_1),u(t, x_2),\cdots,u(t, x_{N})\big)^\top .$ 
	Then we have
	$$\left\|u(t, \cdot)-u^N(t, \cdot)\right\|_{L^p(\Omega;L^\infty)}\leq\left\|u(t, \cdot)-\tilde{u}^N(t,\cdot)\right\|_{L^p(\Omega;L^\infty)}+\left\|\tilde{u}^N(t, \cdot)-u^N(t,\cdot)\right\|_{L^p(\Omega;L^\infty)}.$$ \\
	\textbf{Step 1: } First, we need to verify
	\begin{equation}\label{es: u-tilder{u}^N}
		\sup_{t \geq 0}\left\|u(t,\cdot)-\tilde{u}^{N}(t, \cdot)\right\|_{L^p(\Omega;L^\infty)}   \leq C h^{1-\epsilon}.
	\end{equation}
	
	Following the proof of \cite{deng2025}[Proposition 3.3], we can directly verify that  
	\begin{equation*} 
		\begin{aligned}
			&\int_{\mathcal{O}} G_t(x, y) u_0\left(y\right) \mathrm{d}y=u_0(x)-\int_0^t \int_{\mathcal{O}} \Delta G_{t-s}(x, y) u_0^{\prime\prime}(y) \mathrm{d} y\mathrm{d}s ,\\
			&\int_{\mathcal{O}} G_t^N(x, y) u_0\left(\kappa_N(y)\right) \mathrm{d}y=\tilde{u}^N_0(x)-\int_{0}^{t}\int_{\mathcal{O}}\Delta_NG_s^N(x, y)\Delta_Nu_0((\kappa_N(y))\mathrm{d}y\mathrm{d}s,
		\end{aligned}
	\end{equation*}  
	where
	$  \tilde{u}^N_0(x):=(P_N\circ\mathbb{I}_4\circ u_0)(x).$\par 
	Then  we can obtain	  $ \tilde{u}^{N}(t,x)-u(t, x)=\sum_{j=1}^{N} \mathcal{I}_j(t,x)$, where
	\begin{equation*}
		\begin{aligned}
			\mathcal{I}_1(t,x) & :=\big(	\tilde{u}^N_0(x)-u_0(x)\big), \\
			\mathcal{I}_2(t,x) & :=	\int_0^t \int_{\mathcal{O}} \left(\Delta_NG_{t-s}^N(x, y)-\Delta G_{t-s}(x, y)\right)\left(f\left(u\left(s, \kappa_N(y)\right)\right)-\Delta_N u_0((\kappa_N(y)) \right)\mathrm{d}y\mathrm{d}s , \\
			\mathcal{I}_3(t,x) & :=\int_0^t \int_{\mathcal{O}}
			\Delta G_{t-s}(x, y)\Big(\left( f\left(u\left(s, \kappa_N(y)\right)\right)-f(u(s,y)) \right) -\left( \Delta_Nu_0(\kappa_N(y))-u^{\prime\prime}_0(y)\right)\Big) \mathrm{d} y \mathrm{d} s  , \\
			\mathcal{I}_4(t,x) & :=\sigma\sum_{j=1}^{N-1} \int_0^t \int_{\mathcal{O}} \left(G_{t-s}^N(x, y)\phi_j\left(\kappa_N(y)\right)- G_{t-s}(x, y)\phi_j\left(y\right) \right) \mathrm{d} y \mathrm{d}  \beta_j(s)\\
			& \quad-\sigma\sum_{j=N}^{\infty}\int_0^t \int_{\mathcal{O}} G_{t-s}(x, y)\phi_j\left(y\right) \mathrm{d} y \mathrm{d}  \beta_j(s).
		\end{aligned}
	\end{equation*}     
	For any $x\in\mathcal{O}$, according to  the definitions of $ \tilde{u}^N_0$ and $\Delta_N$ and Assumption \ref{assu:2},    we have
	\begin{equation} \label{ineq: tilde{u}^N_0-u_0}
		|\tilde{u}^N_0(x)-u_0(x)|\leq Ch^2 \  \text{and} \   |\Delta_Nu_0(\kappa_N(x))-u^{\prime\prime}_0(x)| \leq Ch^2. 
	\end{equation}
	It follows from \eqref{ineq: tilde{u}^N_0-u_0}, \eqref{ineq: u_regularity} and  \eqref{es: Delta_N G_s^N-Delta G_s(x, y) } with $\alpha_1=1-\epsilon$ that
	\begin{equation} \label{ineq: ||tilde{u^n}-u || 1}
		\begin{aligned}
			&\left\|\mathcal{I}_1(t,\cdot)\right\|_{L^p(\Omega;L^\infty)} 
			\leq 	\sup_{x\in\mathcal{O}}\big|\tilde{u}^N_0(x)-u_0(x)\big| 
			\leq Ch^2 ,
		\end{aligned}
	\end{equation}
	and
	\begin{equation} \label{ineq: ||tilde{u^n}-u || 1.1}
		\begin{aligned}
			\left\|\mathcal{I}_2(t,\cdot)\right\|_{L^p(\Omega;L^\infty)} 
			\leq& \int_{0}^{t}\sup_{x \in\mathcal{O}}\int_{\mathcal{O}}
			\left|\Delta_NG_{t-s}^N(x, y)-\Delta G_{t-s}(x, y)\right|\left(\|f\left(u\left(s, \kappa_N(y)\right)\right)\|_{L^p(\Omega;\mathbb{R})}+\Delta_Nu_0(y)   \right)\mathrm{d}y\mathrm{d}s  \\  
			\leq&	 \int_{0}^{t}\sup_{x \in\mathcal{O}}\int_{\mathcal{O}}
			\left|\Delta_NG_{t-s}^N(x, y)-\Delta G_{t-s}(x, y)\right|\left(\| u\left(s, \kappa_N(y) \right)\|_{L^{3p}(\Omega;\mathbb{R})}^3+1 \right)\mathrm{d}y\mathrm{d}s  \\  
			\leq& C\int_{0}^{t}\sup_{x \in\mathcal{O}}\int_{\mathcal{O}}
			\left|\Delta_NG_{t-s}^N(x, y)-\Delta G_{t-s}(x, y)\right| \mathrm{d}y\mathrm{d}s \\
			\leq& C\int_{0}^{t}e^{-\frac{\lambda_{N,1}}{2}(t-s)}  (t-s)^{-\frac{7+2\epsilon}{8}} h^{1-\epsilon}\mathrm{d}s  
			\leq Ch^{1-\epsilon}.
		\end{aligned}
	\end{equation}

	To estimate of  $\mathcal{I}_3(t,x)$, we need to investigate $u\left(t, \kappa_N(x)\right)-u(t,x)$. Taking an arbitrarily small positive number $\epsilon$, we can choose a sufficiently large number $q$ such that $q>\max\left\{ p+2,3/\epsilon \right\}$. For any $ t>0$ and $x\in \mathcal{O}$, utilizing Assumption \ref{assu:2}, \eqref{ineq: futher property of O_t}, H{\"o}lder's inequality, \eqref{es: Delta_N G_s^N-Delta G_s(x, y) } and \eqref{estimate: e^{-x} 1} with $\alpha=\frac{15}{8}$, we obtain 
	\begin{align*}
		& \left\|u\left(t, \kappa_N(x)\right)-u(t,x)\right\|_{L^{2q}(\Omega;\mathbb{R})}\\
		\leq&|u_0(\kappa_N(x))-u_0(x)|+ \left\|o(t,\kappa_N(x))-o(t,x)\right\|_{L^{2q}(\Omega;\mathbb{R})}\\
		&+\int_0^t \int_{\mathcal{O}}   \left|\Delta G_{t-s}(\kappa_N(x), y)- \Delta G_{t-s}(x, y)\right| \left(|u^{\prime\prime}_0\left(y\right)|+\left\|f\left(u\left(s, y\right)\right)\right\|_{L^{2q}(\Omega;\mathbb{R})}\right)  \mathrm{d} y \mathrm{d} s  \\
		\leq&Ch+   C\int_0^t \int_{\mathcal{O}}   \left|\Delta G_{t-s}(\kappa_N(x), y)- \Delta G_{t-s}(x, y)\right| \left(1+\left\|u\left(s, y\right)\right\|_{L^{6q}(\Omega;\mathbb{R})}^3\right)  \mathrm{d} y \mathrm{d} s\\
		\leq&Ch+ C  \int_0^t \Big(\int_{\mathcal{O}}   \left|\Delta G_{t-s}(\kappa_N(x), y)- \Delta G_{t-s}(x, y)\right|^2   \mathrm{d} y\Big)^\frac{1}{2} \mathrm{d} s 
		\leq Ch. 
	\end{align*}    
	It follows from the aforementioned estimate and the definition of $ G_t(x, y)$ that
	
	\begin{equation} \label{ineq: ||tilde{u^n}-u || 1.1}
		\begin{aligned}
			&\left\|\mathcal{I}_3(t,x)\right\|_{L^{q}(\Omega;\mathbb{R})}\\
			\leq& \int_0^t \int_{\mathcal{O}} \left|\Delta G_{t-s}(x, y)\right| 
			\left( 	|\Delta_Nu_0(y)-u^{\prime\prime}_0(y)|+	\|f\left(u\left(s, \kappa_N(y)\right)\right)-f(u(s,y)) \|_{L^{q}(\Omega;\mathbb{R})} \right)   \mathrm{d}y\mathrm{d}s   \\
			\leq&  C\int_0^t \int_{\mathcal{O}} 
			\Big| \sum_{j=0}^{\infty}\lambda_{ j} e^{-\lambda_{ j}^2 (t-s)}     \Big|
			\left(h^2+	\Big(1+\sup_{x\in\mathcal{O}}\|u(s,x) \|_{L^{4q}(\Omega;\mathbb{R})}^2\Big)\|u\left(s, \kappa_N(y)\right)-u(s,y) \|_{L^{2q}(\Omega;\mathbb{R})}\right)\mathrm{d}y\mathrm{d}s\\
			\leq& Ch\int_0^t  e^{-\frac{\lambda_{ 1}^2}{2} (t-s)}(t-s)^{-\frac{3}{4}} \mathrm{d}s\leq Ch .
		\end{aligned}
	\end{equation} 
	Moreover, for any $x,\ z\in \mathcal{O}$, by \eqref{es:  DeltaG_t(x,y)-DeltaG_t(z,y)} with $\alpha=\frac{3}{4}$ we obtain
	\begin{equation*} 
		\begin{aligned}
			\left\|\mathcal{I}_3(t,x)-\mathcal{I}_3(t,z)\right\|_{L^{q}(\Omega;\mathbb{R})}
			&\leq  Ch\int_0^t \int_{\mathcal{O}} \left|\Delta G_s(x, y)-\Delta G_s(z, y)\right|   \mathrm{d}y\mathrm{d}s  \\
			&\leq  Ch\int_0^t e^{-\frac{\lambda_{ 1}^2}{2}(t-s)}\left|x-z\right|^\frac{3}{4}\mathrm{d}s   
			\leq Ch|x-z|^\frac{3}{4}.  
		\end{aligned}
	\end{equation*} 
	Thus we employ \cite[Theorem 2.1]{revuz2013continuous}  to obtain 
	\begin{equation*}   
		\begin{aligned}
			\mathbb{E}\left[\sup_{x\neq z} \left( 
			\frac{\frac{1}{h}|\mathcal{I}_3(t,x)-\mathcal{I}_3(t,z)|}{|x-z|^{\frac{3}{4}-\frac{2}{q}}} \right)^q\right] \leq C,
		\end{aligned}
	\end{equation*}  
	which combined with \eqref{ineq: ||tilde{u^n}-u || 1.1} yields
	\begin{equation}  \label{ineq: ||tilde{u^n}-u || 2}
		\begin{aligned}
			\left\|\mathcal{I}_3(t,\cdot)\right\|_{L^p(\Omega;L^\infty)}\leq	\left\|\mathcal{I}_3(t,\cdot)\right\|_{L^q(\Omega;L^\infty)}
			\leq	\left\|\mathcal{I}_3(t,\cdot)-\mathcal{I}_3(t,1)\right\|_{L^q(\Omega;L^\infty)}+	\left\|\mathcal{I}_3(t,1)\right\|_{L^{q}(\Omega;\mathbb{R})}\leq Ch.
		\end{aligned}
	\end{equation} 
	Combining the orthonormality of $\left\{\phi_j\right\}_{j\geq 0} $ and $\left\{\phi_{N,j}\right\}_{j=0}^{N-1} $  with  Itô's isometry, we can obtain
	\begin{equation} \label{ineq: ||tilde{u^n}-u || 2.1}
		\begin{aligned} 
			\left\|\mathcal{I}_4(t,x)\right\|_{L^{q}(\Omega;\mathbb{R})}^2\leq C \sum_{k=1}^{3} 	\mathcal{I}_{4,k}(t,x),
		\end{aligned}
	\end{equation}
	where
	$$
	\begin{aligned}
		&	\mathcal{I}_{4,1}(t,x)=  \int_{0}^{t} \sum_{j=N}^{\infty}\lambda_{j}^2 e^{-2\lambda_j^2 (t-s)} \mathrm{d}s   , \  
		\mathcal{I}_{4,2}(t,x)= \int_{0}^{t}\sum_{j=0}^{N-1}  e^{-2\lambda_{N, j}^2 t} \left(\phi_j(x) -\bar{\phi}_{N, j}(x)\right)^2  \mathrm{d}s ,\\ 
		&		\mathcal{I}_{4,3}(t,x)=\int_{0}^{t}\sum_{j=0}^{N-1} \left( e^{-\lambda_j^2 t}- e^{-\lambda_{N, j}^2 t}\right) |\phi_j(x)|^2\mathrm{d}s .
	\end{aligned}
	$$ 
	A proof similar to the one used in \eqref{es: L1}, \eqref{es: L3}, and \eqref{es: L4} shows that
	\begin{equation} \label{ineq: ||tilde{u^n}-u || 2.1}
		\begin{aligned}
			\left\|\mathcal{I}_4(t,x)\right\|_{L^{q}(\Omega;\mathbb{R})}\leq C \sqrt{\sum_{k=1}^{3} 	\mathcal{I}_{4,k}(t,x)}\leq Ch^{1-\frac{\epsilon}{2}}.
		\end{aligned}
	\end{equation}
	According to \eqref{es:  G_t(x,y)-G_t(z,y)} and \eqref{ineq: futher property of O_t}, one has 
	\begin{equation*} 
		\begin{aligned}
			\left\|\mathcal{I}_4(t,x)-\mathcal{I}_4(t,z)\right\|_{L^{q}(\Omega;\mathbb{R})} 
			\leq C|x-z|.
		\end{aligned}
	\end{equation*} 
	Combining this with  \cite[Theorem 2.1]{revuz2013continuous}  leads to
	\begin{equation}   \label{ineq: ||tilde{u^n}-u || 2.2}
		\begin{aligned}
			\mathbb{E}\left[\sup_{x\neq z} \left(\frac{|\mathcal{I}_4(t,x)-\mathcal{I}_4(t,z)|}{|x-z|^{1-\frac{2}{q}}} \right)^q\right] \leq C.
		\end{aligned}
	\end{equation}  
	Then it can follow from \eqref{ineq: ||tilde{u^n}-u || 2.1} and \eqref{ineq: ||tilde{u^n}-u || 2.2} that
	\begin{equation}    \label{ineq: ||tilde{u^n}-u || 3}
		\begin{aligned}
			\left\|\mathcal{I}_4(t,\cdot)\right\|_{L^p(\Omega;L^\infty)} 
			\leq&\left\|\mathcal{I}_4(t,\cdot)\right\|_{L^q(\Omega;L^\infty)}\\ 
			\leq &
			C\left(\mathbb{E}\left[\sup_{1\leq i\leq N}\left| \mathcal{I}_4(t,x_i)\right|^q\right]\right)^\frac{1}{q} +C\left(\mathbb{E}\left[\sup_{1\leq i\leq N}\sup_{x \in[(i-1)h, ih]}\left|\mathcal{I}_4(t,x_i)-\mathcal{I}_4(t,x) \right|^q\right]\right)^\frac{1}{q}\\
			\leq &
			C\left(\sum_{i=1}^{N}\mathbb{E}\left[\left| \mathcal{I}_4(t,x_i)\right|^q\right]\right)^\frac{1}{q} +C\left(\sum_{i=1}^{N}\mathbb{E}\left[ \sup_{x \in[(i-1)h, ih]}\left|\mathcal{I}_4(t,x_i)-\mathcal{I}_4(t,x)\right|^q\right]\right)^\frac{1}{q}\\
			\leq &
			Ch^{1-\epsilon}.
		\end{aligned}
	\end{equation} 
	%
	
	\textbf{Step 2: } We next prove
	\begin{equation}\label{es: tilder{u}^N-u^N}
		\sup_{t \geq 0}\left\|\tilde{u}^{N}(t, \cdot)-u^N(t,\cdot)\right\|_{L^p(\Omega;L^\infty)}   \leq C h^{1-\epsilon}.
	\end{equation}
	We denote $E^N_t:= \tilde{U}^N_t-U^N_t$. This allows us to derive the error equation 
	\begin{equation} \label{eq: U-tilder(U)^N_matrix}
		\left\{\begin{aligned} 
			&\mathrm{d}  E_t^N+A_N^2 E_t^N \mathrm{d}  t=A_N\left[F_N (U_t^N+E_t^N+(\mathbb{U}_t-\tilde{U}^N_t) )-F_N (U_t^N)\right]\mathrm{d}  t,\ t \in(0,\infty), \\
			&E_0^N=0.
		\end{aligned}\right.
	\end{equation}
	For any $t\in[0,\infty)$,  an argument similar to the one used in Proposition \ref{propo: error problem  exact} show 
	\begin{equation}\label{es: solution_perturbed problem APPXI}
		\begin{aligned}
			\| E_t^N  \|_{L^p(\Omega;l^2_N)}  \leq &  C\sqrt{K_5(t)}
			\left(\mathbb{I}_1\big(t,\frac{1}{2},\frac{\lambda_{N,1}^2}{2},\big\|\mathbb{U} -\tilde{U}^N \big\|_{L^{2p}(\Omega;l^2_N))}^2    \big)  \right)^\frac{1}{2}  \\
			&+	C \sqrt{K_5(t)K_6(t)}   \left(
			\mathbb{I}_1\big(t,0,\epsilon_0,   \left\|\mathbb{U} -\tilde{U}^N \right\|_{L^{2p}(\Omega;l^4_N)}^4   \big) 	\right)^\frac{1}{4} , \quad t\in[0,\infty),
		\end{aligned}
	\end{equation} 
	where 
	\begin{equation*} 
		\begin{aligned}
			&K_5(t):=\mathbb{I}_1\big(t,0,\epsilon_0,1+
			\left\|\tilde{U}^N \right\|_{L^{2p}(\Omega;l^8_N)}^8+\left\|\mathbb{U} -\tilde{U}^N\right\|_{L^{2p}(\Omega;l^8_N)}^8\big) ,\\ 
			&K_6(t):=\mathbb{I}_1\big(t,\frac{1}{2},\frac{\epsilon_0}{2},1+	\big\|\mathbb{U}_s\big\|_{L^{2p}(\Omega; w^{1,2}_N)}^4+\big\|\tilde{U}^N_s\big\|_{L^{2p}(\Omega; w^{1,2}_N)}^4+	 U^N_s\|_{L^{2p}(\Omega; w^{1,2}_N)}^4\big) .  
		\end{aligned}
	\end{equation*}  
	It follows \eqref{es: U  uniform W1,2}  that 
	$$  \big\|\mathbb{U}_t\big\|_{L^{p}(\Omega; w^{1,2}_N)}\leq C.$$
	By a similar proof, we can obtain 
	$$  \big\|U^N_t\big\|_{L^{p}(\Omega; w^{1,2}_N)}\leq C.$$
	By utilizing \eqref{ineq: e^{-A_N^2t}(-A_N)^r to l2}   and Lemma \ref{lem: O^{M,N}_continuous} one can get that 
	\begin{equation}
		\begin{aligned}
			\left\|  \tilde{U}_t^N\right\|_{w^{1,2}_N} & \leq\left\|e^{-A_N^2 t} u^N_0\right\|_{w^{1,2}_N}+\int_0^t\left\|\left(-A_N\right)^{\frac{3}{2}} e^{-A_N^2(t-s)} F_N\left(\mathbb{U}_s\right)\right\|_{l_N^2} \mathrm{ d} s+\left\|  \mathscr{O}_t^N\right\|_{w^{1,2}_N} \\
			& \leq C\left(\left\|  u^N_0\right\|_{w^{1,2}_N}+\int_0^t(t-s)^{-\frac{3}{4}}e^{-\frac{\lambda_{N,1}}{2}}\left\|  F_N\left(\mathbb{U}_s\right)\right\|_{l_N^2} \mathrm{~d} s+\left\|  \mathscr{O}_t^N\right\|_{w^{1,2}_N} \right)   \leq C .
		\end{aligned}
	\end{equation} 
	Then it follows from the above estimates that
	$$
	\mathbb{E}\left[\left\|E_t^N\right\|_{l_N^2}^q\right] \leq C h^{q(1-\epsilon)}, \quad   q \geq 1.
	$$
	Furthermore,	we can obtain
	\begin{align*}
		&\left\|u^N(t, \cdot)-\tilde{u}^N(t,\cdot)\right\|_{L^p(\Omega;L^\infty)}\\
		=& \left\|E^N_t\right\|_{L^p(\Omega;l^\infty_N)}
		\leq C \left(	\mathbb{E}\left[	\left\| E^N_t \right\|_{w^{1,2}_N}^p\right]
		\right)^\frac{1}{p} \\ 
		& \leq C\left\|\int_0^t\Big\|\left(-A_N\right)^{\frac{3}{2}} e^{-A_N^2(t-s)}\big[F_N\left(\mathbb{U}_s\right)-F_N (U_s^N )\big]\Big\|_{l_N^2} \mathrm{ d} s\right\|_{L^p(\Omega; \mathbb{R})}+C h^{1-\epsilon} \\
		& \leq C \int_0^t(t-s)^{-\frac{3}{4}}e^{-\frac{\lambda_{N,1}}{2}}\Big\|1+\| \mathbb{U}_s \|_{l_N^{\infty}}^2+ \| U_s^N \|_{l_N^{\infty}}^2\Big\|_{L^{2 p}(\Omega; \mathbb{R})}\left\|\mathbb{U}_s-U_s^N\right\|_{L^{2 p}\left(\Omega; l_N^2\right)} \mathrm{d} s+C h^{1-\epsilon} \\
		& \leq C \int_0^t(t-s)^{-\frac{3}{4}}e^{-\frac{\lambda_{N,1}}{2}}\left(\left\|E_s^N\right\|_{L^{2 p}\left(\Omega; l_N^2\right)}+\left\|\mathbb{U}_t-\tilde{U}_t^N\right\|_{L^{2 p}\left(\Omega; l_N^2\right)}\right) \mathrm{d} s+C h^{1-\epsilon} \\
		& \leq C h^{1-\epsilon} 
	\end{align*} 
	This proof is completed. 
\end{proof}

	\section*{Appendix E.}
	\begin{proof}
		When $\alpha_1 = 1$,	 each function $\phi_{v,1,\alpha_2}$ is readily verified to be bounded and uniformly continuous. It remains to verify that the ergodic limit equals zero.\par 
		
		Suppose $u_0 = -\frac{a_1}{3a_0}$. Then, for any $\tau \in \mathbb{R}^+$ and $N \in \mathbb{N}^+$, we have  $U^{\tau,N}_{t_0}=-\frac{a_1}{3a_0}\phi_{N,0}$.
		We proceed by mathematical induction to demonstrate that, for any $m\in\mathbb{N} $,  
		\begin{equation}\label{math induction}
			U^{\tau,N}_{t_m}= -\frac{a_1}{3a_0}\phi_{N,0}+\sum_{j=1}^{N-1} \tilde{X}_{m,j}\phi_{N,j} ,
		\end{equation}
		where $ \{\tilde{X}_{m,j}\}_{j=1}^{N-1}$ denotes real-valued random variables symmetric about zero.

		For the base case $m=1$, it follows from \eqref{eq: mild_solution_full_discrete discrete} and  the orthonormality of $\left\{\phi_{N,j}\right\}_{j=0}^{N-1} $ that
		\begin{equation*} 
			\begin{aligned}
				\langle  	U^{\tau,N}_{t_{1}} ,\phi_{N,j}\rangle_{l^{2}_N}
				=& e^{-\lambda_{N,j}^2 \tau}  \langle  	U^{\tau,N}_{t_0} ,\phi_{N,j}\rangle_{l^{2}_N}
				+\tau \lambda_{N,j} e^{-\lambda_{N,j}^2 \tau}  \langle  \tilde{F} (U^{\tau,N}_{t_0}) ,\phi_{N,j}\rangle_{l^{2}_N}
				+\sigma e^{-\lambda_{N,j}^2 \tau}  \langle  \Delta \beta_0 ,\phi_{N,j}\rangle_{l^{2}_N}\\
				=&   \sigma e^{-\lambda_{N,j}^2 \tau}     \Delta \beta_{0,j} ,\quad j=1,2,\cdots,N-1.
			\end{aligned}
		\end{equation*} 
		Since $ \langle  	U^{\tau,N}_{t_{1}} ,\phi_{N,j}\rangle_{l^{2}_N}=-\frac{a_1}{3a_0} $, we have
		\begin{equation*} 
			\begin{aligned}
				U^{\tau,N}_{t_{1}}=\sum_{j=0}^{N-1} \langle  	U^{\tau,N}_{t_{1}} ,\phi_{N,j}\rangle_{l^{2}_N}\phi_{N,j}
				=-\frac{a_1}{3a_0}\phi_{N,0}+\sigma \sum_{j=1}^{N-1}  \Delta \beta_{0,j}\phi_{N,j},
			\end{aligned}
		\end{equation*}
		which, combined with the fact that $\{\beta_{0,j}\}_{j=1}^{N-1}$ are Gaussian random variables with mean zero, confirms that
		\eqref{math induction} holds.
		
		Assume now that \eqref{math induction} holds for  $m=k$. For $m=k+1$, it again follows from \eqref{eq: mild_solution_full_discrete discrete} and orthonormality that
		\begin{equation}  \label{math induction 1}
			\begin{aligned}
				\langle  	U^{\tau,N}_{t_{k+1}} ,\phi_{N,j}\rangle_{l^{2}_N}
				=& e^{-\lambda_{N,j}^2 \tau}  \langle  	U^{\tau,N}_{t_{k}} ,\phi_{N,j}\rangle_{l^{2}_N}
				+\tau \lambda_{N,j} e^{-\lambda_{N,j}^2 \tau}  \langle  \tilde{F} (U^{\tau,N}_{t_{k}}) ,\phi_{N,j}\rangle_{l^{2}_N}
				+\sigma e^{-\lambda_{N,j}^2 \tau}  \langle  \Delta \beta_{k} ,\phi_{N,j}\rangle_{l^{2}_N}\\
				=&   \tau \lambda_{N,j} e^{-\lambda_{N,j}^2 \tau}  \langle  \tilde{F} (U^{\tau,N}_{t_{k}}) ,\phi_{N,j}\rangle_{l^{2}_N}+\sigma e^{-\lambda_{N,j}^2 \tau}     \Delta \beta_{k,j} ,\quad j=1,2,\cdots,N-1.
			\end{aligned}
		\end{equation} 
		According the definition of $\tilde{F}$ and the orthonormality of $\left\{\phi_{N,j}\right\}_{j=0}^{N-1} $, we have
		\begin{equation} \label{math induction 2}
			\begin{aligned}
				\langle\tilde{F} (U^{\tau,N}_{t_{k}}) ,\phi_{N,j}\rangle_{l^{2}_N}
				=&  \frac{1}{1+\tau  \big\| U^{\tau,N}_{t_{k}} \big\|_{w^{1,2}_N}^{12}}  \Big\langle a_0  \Big(U^{\tau,N}_{t_{k}}\Big)^3+a_1 \Big(U^{\tau,N}_{t_{k}}\Big)^2+
				a_2  U^{\tau,N}_{t_{k}} +a_3  ,\phi_{N,j} \Big\rangle_{l^{2}_N}\\
				=&  \frac{ \Big\langle a_0  \Big(\sum_{j=1}^{N-1} \tilde{X}_{m,j}\phi_{N,j}\Big)^2+ 
					\Big(	-\frac{a_1^2}{3a_0}+a_2 \Big) \sum_{j=1}^{N-1} \tilde{X}_{m,j}\phi_{N,j} ,\phi_{N,j} \Big\rangle_{l^{2}_N}}{1+\tau  \Big( \frac{a_1^2}{9a_0^2}+\sum_{j=1}^{N-1}(1+\lambda_{N,j}) \tilde{X}_{m,j}^2  \Big)^6 }  \\
			\end{aligned}
		\end{equation} 
		By substituting \eqref{math induction 2} into \eqref{math induction 1} and using the symmetry of  $ \{\tilde{X}_{k,j}\}_{j=1}^{N-1}$ and  $\{\beta_{k,j}\}_{j=1}^{N-1}$, we deduce that \eqref{math induction} also holds for  $m=k+1$.
		Thus, by the principle of mathematical induction, the claim holds for all $m\in \mathbb{N}$.
		
		It then follows that $ \mathbb{E} \left[\phi_{v,1,\alpha_2} (u^{\tau,N}( t_m,\cdot)  )\right]=0$ for all $m\in \mathbb{N}$. 
		From \eqref{eq: ergodicity M N}, we conclude that 
		$$  \int_{V^N} \phi_{v,1,\alpha_2}(v) \tilde{\pi}^{\tau,N}( \mathrm{ d} v ) =0. $$
		Finally, applying Theorem \ref{th: error: invariant measure } , we obtain
		$$  \int_{H_\alpha} \phi_{v,1,\alpha_2}(v)  \pi ( \mathrm{ d} v ) =0,$$
		which completes the proof. 
	\end{proof}   

\end{document}